\documentclass[10 pt,reqno]{amsart}
\usepackage{color}
\usepackage{enumerate}
\usepackage{amssymb,amsmath,amsthm,latexsym,mathrsfs}
\usepackage{amsbsy,amsfonts,mathtools,slashed}
\usepackage{graphicx,color,yfonts}
\usepackage[numbers,sort]{natbib}
\usepackage[latin9]{inputenc}

\usepackage[colorlinks=true]{hyperref}
\hypersetup{linkcolor=red,citecolor=blue,filecolor=dullmagenta,urlcolor=blue}

\newtheorem{thm}{Theorem}[section]
\newtheorem{lemma}[thm]{Lemma}
\newtheorem{prop}[thm]{Proposition}

\theoremstyle{plain}

\newtheorem{rem}[thm]{Remark}

\makeatother
\numberwithin{equation}{section}
\numberwithin{figure}{section}
\newcommand{\les}{\lesssim}

\newcommand{\al}{{\alpha}}

\newcommand{\Gam}{{\Gamma}}
\newcommand{\de}{{\delta}}

\newcommand{\ve}{{\varepsilon}}
\newcommand{\lam}{{\lambda}}

\newcommand{\vp}{{\varphi}}
\newcommand{\ka}{{\kappa}}

\newcommand{\Om}{{\Omega}}
\newcommand{\tT}{{\theta}}

\newcommand{\R}{{\mathbb R}}
\newcommand{\Z}{{\mathbb Z}}

\newcommand{\C}{{\mathbb C}}
\newcommand{\N}{{\mathbb N}}

\def\cE{\mathcal E}
\def\cF{\mathcal F}
\def\cG{\mathcal G}

\def\cI{\mathcal I}
\def\cJ{\mathcal J}
\def\cL{\mathcal L}

\def\cN{\mathcal N}
\def\cO{\mathcal O}
\def\cP{\mathcal P}

\def\cR{\mathcal R}

\def\cU{\mathcal U}
\def\cV{\mathcal V}

\def\cK{\mathcal K}

\def\normo#1{\left\|#1\right\|}

\def\bra#1{\left\langle #1\right\rangle}

\def\wt#1{\widetilde{#1}}
\def\wh#1{\widehat{#1}}
\def\ol#1{\overline{#1}}

\newcommand{\bm}{\mathbf}

\newcommand{\p}{\partial}
\newcommand{\supp}{{\rm supp}}

\newcommand{\brad}{{\bra{D}}}
\newcommand{\braxi}{{\bra{\xi}}}
\newcommand{\brat}{{\bra{t}}}

\newcommand{\cl}{{\mathcal L}}
\newcommand{\vpjk}{{\vp_{j,k}}}

\newcommand{\xil}{{\xi_l}}

\newcommand{\pk}{{P_{k}}}

\newcommand{\pko}{{P_{k_1}}}
\newcommand{\pkt}{{P_{k_2}}}

\newcommand{\pkf}{{P_{k_4}}}

\newcommand{\thez}{{\theta_0}}
\newcommand{\theo}{{\theta_1}}
\newcommand{\thet}{{\theta_2}}
\newcommand{\theth}{{\theta_3}}
\newcommand{\thef}{{\theta_4}}

\newcommand{\thej}{{\theta_j}}

\newcommand{\doth}{{\dot{H}}}
\newcommand{\hang}{{H_\Omega^{0,1}}}

\newcommand{\qjk}{Q_{jk}}

\newcommand{\psit}{\psi_\theta}

\newcommand{\cm}{{\rm CM}}

\begin{document}
\title[Maxwell-Dirac system]{The global dynamics for the  Maxwell-Dirac system}

\author{Yonggeun Cho}
\address{Department of Mathematics, and Institute of Pure and Applied Mathematics,
	Jeonbuk National University, Jeonju, 54896, Republic of Korea}
\email{changocho@jbnu.ac.kr}
\author{Kiyeon Lee}
\address{Stochastic Analysis and Application Research Center(SAARC), Korea Advanced Institute of Science and Technology, 291 Daehak-ro, Yuseong-gu, Daejeon, 34141, Republic of	Korea}
\email{kiyeonlee@kaist.ac.kr}


\begin{abstract}
	In this paper, we study the global existence and modified scattering of solutions to (1+3) dimensional massive Maxwell-Dirac system in the Lorenz gauge. We employ a vector-field energy method combined with a delicate analysis of the space-time resonance argument. This approach allows us to establish decay estimates and energy bounds crucial for proving the main theorem.  Especially, we provide an explicit phase correction arising from the strong nonlinear resonances.
\end{abstract}

\maketitle

\tableofcontents

\section{Introduction}

\subsection{Maxwell-Dirac system}
In this paper we consider (1+3)-dimensional Maxwell-Dirac system:
\begin{align}\label{md}
	\left\{ \begin{aligned}
		i\al^\mu \textbf{D}_{\mu} \psi &= m \beta \psi,\\
		\partial^\nu F_{\mu\nu} &= - \bra{\psi, \al_\mu \psi},
	\end{aligned}
	\right. \tag{MD}
\end{align}
where  the spinor field  is $\psi : \R^{1+3} \to \C^{4}$, the gauge fields are $A_\mu : \R^{1+3} \to \R$,  the covariant derivative $\textbf{D}_\mu$ denotes $\partial_\mu -iA_\mu$ for $\mu = 0, 1, 2, 3$, and $\partial_0 = \partial_t$.  The curvature is defined by $F_{\mu\nu} = \partial_\mu A_\nu - \partial_\nu A_\mu$. $\al^\mu$ and $\beta$ are Hermitian $4\times 4$ matrices. In particular, $\al^0$ denotes the $4\times 4$ identity $I_4$ and $\al^1, \al^2, \al^3, \beta$ have the relation
\begin{align*}
	 \al^j \al^k + \al^k \al^j = 2\delta^{jk}I_4, \;\;\al^j\beta + \beta\al^j = 0\;\;(j,k = 1, 2, 3), \;\;\mbox{and}\;\; \beta^2 = I_4.
\end{align*}
We use the standard Pauli-Dirac representation \cite{book:bjorken-drell,book:thaller} in this paper. The Greek indices indicate the space-time components $\mu,\nu = 0,\, 1,\,2,\,3$ and Roman indices mean the spatial components $j=1,\, 2,\,3,$ in the sequel.  The Einstein summation convention is in effect with Greek indices summed over $\mu=0,\cdots, 3$ and Roman indices summed over the spatial variables $j = 1, 2, 3$. We lower and raise indices with Minkowski metric $\eta = {\rm diag}(-1,1,1,1)$. Thus $\al_\mu = \eta_{\mu\nu}\alpha^\nu$ and $\p^\mu = \eta^{\mu\nu}\p_\nu$ . The $\langle \phi , \psi \rangle = \phi^\dagger \psi$ denotes a standard complex inner product. We call \emph{massive} and \emph{massless} \eqref{md} if the mass parameter is $m > 0$ and $m=0$, respectively.

Maxwell-Dirac system is the Euler-Lagrange equations for  $\textbf{S}[A_\mu,\psi]$, where
\begin{align*}
	\textbf{S}[A_\mu,\psi] = \int_{\R^{1+3}}  \left[ -\frac14 F^{\mu\nu}F_{\mu\nu} + i \bra{\psi, \al^\mu \textbf{D}_\mu \psi} - m \bra{\psi, \beta\psi}  \right] dxdt.
\end{align*}
%Computing this Euler-Lagrangian, we arrive at the \eqref{md}.
The system \eqref{md} models an electron in electromagnetic field and form a fundamental system in quantum electrodynamics. For detailed description, we refer to \cite{book:bjorken-drell,book:schwartz}.

One of the basic features of \eqref{md} is the gauge invariance. Indeed, \eqref{md} is invariant under the gauge transformation $
(\psi,A) \longmapsto (e^{i\chi}\psi, A-d\chi),$ for a real-valued function $\chi$ on $\R\times\R^{3}$. For the sake of concreteness of discussion, let us choose the Lorenz gauge
\begin{align}\label{gauge:lorenz}
	\partial^\mu A_\mu =0.
\end{align}
Under \eqref{gauge:lorenz}, the system \eqref{md} becomes
\begin{align}
	\left\{ \begin{aligned}(-i\partial_t +  \alpha \cdot D  + m{\beta})\psi &=  A_\mu\alpha^\mu \psi\quad\mathrm{in}\;\;\mathbb{R}^{1+3},\\
		-\square A_{\mu} & = \bra{\psi , \alpha_\mu \psi },
	\end{aligned}
	\right.\label{maineq:md-lorenz}
\end{align}
The momentum operator $D = (D_1, D_2, D_3)$ is defined by $D_j = -i\partial_j\;(j = 1, 2,3)$ and $\al = (\al^1,\al^2,\al^3)$ and $\square=-\partial_{t}^{2}+\Delta$. %For other common choices, there are Coulomb gauge condition $\p^j A_j =0$ and temporal gauge condition $A_0=0$.
We will consider a Cauchy problem of \eqref{maineq:md-lorenz} with initial data
\begin{align}\label{eq:initial}
	\psi(0) = \psi_0,\quad A_\mu(0)& = a_{ \mu}, \quad\partial_tA_\mu(0) = \dot{a}_{ \mu}.
\end{align}
If the solution to $\eqref{maineq:md-lorenz}$ with  \eqref{eq:initial} is sufficiently smooth, the system possesses the charge conservation law
\begin{align*}
	\|\psi(t)\|_{L^2} = \|\psi_0\|_{L^2}.
\end{align*}

\subsection{Previous works} There is a large amount of literature dealing with the local and global well-posedness, and the asymptotic behavior of solutions of IVP to \eqref{md}. For early work in \cite{gross1966,bour1996}, the authors considered the local well-posedness of \eqref{md} on $\R^{1+3}$ and Georgiev \cite{geor1991} proved the global existence for small, smooth initial data in the Lorenz gauge. Later, D'Ancona-Foschi-Selberg \cite{anfosel2010} obtained an almost optimal regularity $(\psi_0, a_\mu, \dot{a}_\mu) \in H^\ve \times H^{\ve+\frac12} \times H^{\ve-\frac12}$ on $\R^{1+3}$ of local solution in the Lorenz gauge. They exploited the spinorial null structures, which stem from Dirac projection operators.  D'Ancona-Selberg \cite{ansel} extended their previous approach to \eqref{md} on $\R^{1+2}$ and proved the global well-posedness in the charge class $L^2 \times H^\frac12 \times H^{-\frac12}$. Regarding the local well-posedness for \eqref{md} in the Coulomb gauge ($\p^j A_j = 0$), we refer to \cite{bemaupoup1998,masnaka2003-imrn}. Moreover, Masmoudi and Nakanishi \cite{masnaka2003-cmp} showed the unconditional uniqueness results for (1+3) dimensional \eqref{md} in the Coulomb gauge.

Concerning the asymptotic behavior of global solution to \eqref{md} in the $\R^{1+3}$ Minkowski space, we refer to \cite{flasimon1997, psa2005}. In \cite{flasimon1997}, Flato et al considered a final state problem of \eqref{md} in the Lorenz gauge and they showed the asymptotic behavior and asymptotic completeness which leads  to the global well-posedness of IVP to \emph{massive} \eqref{md} for a data set in Schwartz class and also to a modified scattering. The result of \cite{flasimon1997} seems to be the first nonlinear scattering result obtained without compact support condition. Later, Psarelli \cite{psa2005} showed the global behavior for the IVP to \emph{massive} \eqref{md} in the Lorenz gauge with compactly supported initial data.
Regarding %Dirac equations with the magnetic potential, D'Ancona and Okamoto  \cite{ancooka2017} showed an scattering results.
\eqref{md} with vanishing magnetic field, in \cite{CKLY2022, cloos}, the authors established the modified scattering results in the Lorenz gauge, independently. Recently, Herr, Ifrim, and Spitz \cite{herrifrim2024} considered \eqref{md} in the Lorenz gauge and showed the global existence and modified scattering based on the method of testing with wave packets. Especially, they investigated the asymptotic behavior of the solutions via asymptotic expansion inside the light cone. %However, we believe that our analysis method for modified scattering used in this paper has its own specific advantages that would be of independent interest in various applications.

%It should be noticed that the method of \cite{heif2024} is totally different from ours and the works were carried out independently\footnote{While preparing the current paper, we learned  in the conference ??? that the authors of \cite{h} were preparing.}.

For other problems related to \eqref{md}, we first refer to \cite{gaoh} for a modified scattering in the energy critical space on $\mathbb R^{1+d}\;(d\ge 4)$ in the Coulomb gauge and also \cite{lee2023} for a linear scattering on $\mathbb R^{1+4}$ in the Lorenz gauge.  % results are shown by some results.  established the global well-posedness of \emph{massless} \eqref{md} with Coulomb gauge condition for (1+d) Minkowski space $d \ge 4$ and modified scattering of the solution. They also showed the global well-posedness in $\doth^\frac{d-3}2\times \doth^{\frac{d-2}2} \times \doth^{\frac{d-4}2}$, which is absolutely critical regularity.
As a scalar counterpart of \eqref{md}, the Maxwell-Klein-Gordon system (MKG) has been also studied by many authors. Regarding the global well-posedness for the (1+4) dimensional \emph{massless} (MKG), we refer the reader to \cite{kristertata2015-duke, rodtao2004}. In \cite{krieluhr2015, ohtataru2016-inven, ohtata2018}, the authors studied independently the scattering results for \emph{massless} (MKG) in the Coulomb gauge on $\mathbb R^{1+4}$. Concerning the global existence for (1+3) dimension, we refer to \cite{seltes2010}. In particular, in \cite{yang2018, yangyu2019, cankaulin2019, he2021}, the authors presented an asymptotic behavior of the solution to the \emph{massless} (MKG) in the Coulomb and Lorenz gauge on $\mathbb R^{1+3}$. We also refer to \cite{iopau2019,ouyang2023} for the modified scattering results of the \emph{massive} Maxwell-Klein-Gordon type system in the Lorenz gauge. While the \emph{massless} (MKG) are concerned by many authors, there is a few scattering results for \emph{massive} (MKG). We would like to mention the recent result in \cite{gav2019} which establishes the global regularity for the modified scattering results of the higher dimensional \emph{massive} (MKG) on $\R^{1+d}(d \ge 4)$ in the Coulomb gauge. See also \cite{klaiwangyang2020, fangwangyang2021} and \cite{dowya2024, dolimayu2024} for the asymptotic behavior of the (1+3) dimensional \emph{massive} (MKG) and wave-Klein-Gordon type systems, respectively.

\subsection{Main Theorem}

To identify the asymptotic behavior for Dirac spinor, we introduce Dirac projection operator and decompose spinor by using this projection. Let us denote $\bra{D} := \mathcal F^{-1}(\bra{\xi})$ and $|D| := \mathcal F^{-1}(|\xi|)$, where $\cF, \cF^{-1}$ are Fourier transform and its inverse, respectively, and $\bra{\,\cdot\,} := (1+ |\cdot|^2)^{\frac12}$.
Then we define projection operators $\Pi_{\theta}$ for $\theta \in \{+, -\}$
by
\begin{align}\label{eq:projection}
	\Pi_{\theta} \bigg(\!\!=\Pi_{\theta}(D)  \bigg):=\frac{1}{2}\left( I_4 + \theta \frac{\alpha \cdot  D + {\beta}}{\bra{D}}\right).
\end{align}
We denote the symbol of $\Pi_\theta$ by $\Pi_\theta(\xi)$. This simply implies $\psi = \Pi_+\psi  + \Pi_-\psi$.

Let us define the standard vector fields as follows:
\begin{align*}
	\p_\mu\,(\mu =0,1,2,3),  \quad \Om_{jk}:= x_j \p_k - x_k\p_j \,(jk \in \{23,31,12\}), \quad \Gam_j:= t\p_j + x_j \p_t \,(j=1,2,3),
\end{align*}
which are the infinitesimal generators of translation, rotation, and Lorentz boost, respectively. For $n \in \N$, we define sets of differential operators by
\begin{align*}
	\cV_0 := \{ I \}\;\;\mbox{and}\;\;	\cV_n := \left\{   [\p]^a  [\Om]^b [\Gam]^c: |a| + |b| + |c| \le n  \right\},
\end{align*}
where $a=(a_0,a_1,a_2,a_3)\, b=(b_1,b_2,b_3), c=(c_1,c_2,c_3)$, $|a|=\sum_{\mu=0}^4a_\mu$, $|b|=\sum_{j=1}^3 b_j$, $|c|=\sum_{j=1}^3c_j$, and
\begin{align*}
	[\p]^a:= \p_0^{a_0}\p_1^{a_1}\p_2^{a_2}\p_3^{a_3}, \quad [\Om]^b := \Om_{23}^{b_1}\Om_{31}^{b_2}\Om_{12}^{b_3}, \quad [\Gam]^c:= \Gam_1^{c_1}\Gam_2^{c_2}\Gam_1^{c_2}.
\end{align*}

Our main theorem is stated as follows:
\begin{thm} \label{mainthm}
	Let  $N(n)$ be indexed as follows:
	\vspace{-0.5cm}
	\begin{center}
		\begin{align*}
			\begin{tabular}{|c||c|c|c|c|}
				\hline
				n & 0 & 1 & 2 &3 \\
				\hline
				N(n) & 70 & 30 & 20 & 10\\
				\hline
			\end{tabular}
		\end{align*}
	\end{center}
	Them, there exist $\ve_{0}>0$ such that:
	\begin{itemize}
		\item[(i)] Suppose that the initial data \eqref{eq:initial} satisfies that
		\begin{align}
			\|\bra{\xi}^{20}|\xi|^{\frac12}\widehat{\psi_{0}}\|_{L_{\xi}^{\infty}} + \sum_{n=0}^3\left[\|\bra{x}^n\bra{D}^{n}\psi_{0}\|_{H^{N(n)}}+ \|\bra{x}^n\bra{D}^n|D|^{\frac12}(a_\mu,\dot a_\mu)\|_{H^{N(n)}}\right] < \ve_{0}.\label{condition-initial}
		\end{align}
		Then there exists a unique global solution $(\psi, A_\mu)$ with $(\psi, |D|^{\frac12} A_\mu )\in C([0, \infty); H^{N(0)})$ to \eqref{md} in the Lorenz gauge \eqref{gauge:lorenz}.
		
		\item[(ii)] The solution $(\psi(t), A_\mu(t) )$ has the following asymptotic behavior: There exists $(\psi^\infty(t), A_\mu^\infty(t))$ and phase modification $B_\theta (t,D) (\theta \in \{+, -\})$  such that for some $0 < \de, \zeta \ll 1$,
		\begin{align}\label{mainthm:asymptotic}
			\left\{\begin{array}{c}
						\normo{ \Pi_\theta \psi(t) - e^{iB_\theta(t,D)}\Pi_\theta\psi^\infty(t)}_{L^2}\\
  \normo{(A_\mu(t), \p_t A_\mu(t))- (A_\mu^\infty(t),\p_t A_\mu^\infty(t))}_{\doth^{\frac12+\zeta}\times \doth^{-\frac12 + \zeta}}
			\end{array}\right\} = O(t^{-\de}) \;\;\mbox{as}\;\;t\to\infty,
		\end{align}
		where the phase modification is defined by the symbol
		\begin{align}\label{mainthm:correction}
			B_\theta(t,\xi)  &:= \Pi_\theta(\xi) \al^\mu\int_0^t (P_{\le K} A_\mu)\left(s, \theta \frac{ s \xi}{\bra{\xi}}\right) ds
		\end{align}
		for $K = K(s)$ which is the largest integer such that $2^K \le \bra{s}^{-\frac23 - 2\zeta}$,
		and $(\psi^\infty, A_\mu^\infty ) \in C(\mathbb R ;L^2 \times \dot H^{\frac12+\zeta})$ is a solution to the  linear system \eqref{md}:
		\begin{align*}
			\left\{ \begin{aligned}i\al^\mu \p_\mu \psi^\infty &= m\beta \psi^\infty,\\
				\square A_{\mu}^\infty & = 0.
			\end{aligned}
			\right.
		\end{align*}
		\item[(iii)] For some small $\ol{\de} >0$,  the solution $(\psi,A_\mu)$ satisfies the energy bounds with slow growth,
		\begin{align*}
			\sum_{n=0}^3	\sup_{\cL \in \cV_n}\| \cL \Pi_\theta \psi(t)\|_{H^{N(n)}} + \||D|^\frac12\cL A_\mu(t)\|_{H^{N(n)}} \les \ve_0 \bra{t}^{\ol \de},
		\end{align*}
		for $t \in [0, \infty)$. Moreover, $(\psi,A_\mu)$ decays as follows:
		\begin{align}\label{mainthm:decay}
			\sum_{n=0}^3	\sup_{\cL \in \cV_n}	\|\cL\Pi_\theta \psi(t)\|_{L^\infty} + \| \cL A_\mu (t)\|_{L^\infty} \les \ve_0 \bra{t}^{-1+\ol \de}.
		\end{align}
	\end{itemize}
	%%%%%%%%%%%%%%%%%%%%%%%%%%%%%%%%%%%%%%%%%%%%%%%%%%%%%%%%%%%%%%%%%%%%%%%%%%%%%%%%%%%%%%%%%%%%%%%%%%%%%%%%%%%%%%%%%%%%%%%%%%%%%%%%%%%%%%%%%%%%%%%%%%%%%%%%%%%%%%%%%%%%%%%%%%%%%%%%%%%%%%%%%%%%%%%%%%%%%%%%%%%%%%%%%%%%%%%%%%%%%%%%%%%%
	%%%%%%%%%%%%%%%%%%%%%%%%%%%%%%%%%%%%%%%%%%%%%%%%%%%%%%%%%%%%%%%%%%%%%%%%%%%%%%%%%%%%%%%%%%%%%%%%%%%%%%%%%%%%%%%%%%%%%%%%%%%%%%%%%%%%%%%%%%%%%%%%%%%%%%%%%%%%%%%%%%%%%%%%%%%%%%%%%%%%%%%%%%%%%%%%%%%%%%%%%%%%%%%%%%%%%%%%%%%%%%%%%%%%
	
\end{thm}

The novelty of this paper is to provide an explicit form of interaction between spinor and gauge fields resulting in modified scattering. The proof is based on the space-time resonance argument developed by Germain-Masmoudi-Shatah \cite{gemasha2008,gemasha2012-annals,gemasha2012-jmpa} and the vector-field energy method by Ionescu-Pausader \cite{iopau2019}. We believe that our method for modified scattering used in this paper has its own specific advantages that would be of independent interest in various applications.
It remains still open to obtain the global existence and asymptotic behavior of \eqref{md} in the other gauges (especially the Coulomb gauge). The Maxwell part consists of elliptic and wave equations in the Coulomb gauge. Due to the lack of Lorentz invariance in the elliptic equation, the extension of current results to the problem in the Coulomb gauge will be nontrivial. This will be addressed as a next issue.

\begin{rem} 
$(1)$ The proof of Theorem \ref{mainthm} relies on a bootstrap argument under the a priori assumptions \eqref{assum:energy}--\eqref{assum:s-norm}. The initial data condition \eqref{condition-initial} is essential for satisfying these assumptions. Specifically, the Fourier amplitude condition in \eqref{condition-initial} is crucial to demonstrate modified scattering behavior for the Dirac component.\\

$(2)$ The phase correction $B_\theta (t,\xi)$ is a real-valued and exhibits $t^{\ol \de}$-growth due to the slow decay  \eqref{mainthm:decay} of the gauge fields (see \eqref{eq:decay-correction}). In view of the dispersive effects, even if we obtain the full decay $t^{-1}$ of the Maxwell part, this phase correction still possesses at least logarithmic divergence.  \\

$(3)$ In \eqref{mainthm:asymptotic}, this asymptotic behavior for Dirac spinor implies that there exists $B(t,D)$ such that
	\begin{align*}
		\normo{ \psi(t) - e^{iB(t,D)}\psi^\infty(t)}_{L^2} \xrightarrow{t\to \infty} 0.
	\end{align*}
	Here, $B(t,D)$ can be defined as
	\begin{align*}
		e^{iB(t,D)}: = e^{iB_+(t,D)}\Pi_+ +  e^{iB_-(t,D)}\Pi_-.
	\end{align*}

$(4)$ Theorem \ref{mainthm} describes the linear scattering behavior for the wave part, while the spinor component exhibits modified scattering effects. %According to $(3)$ in this remark, we can infer that $\psi$ has  optimal decay.
The regularity of the homogeneous space plays a role of a null structure, which implies additional time decay. Based on this observation, we also obtain
	\begin{align}\label{eq:linear-sc-high}
		\normo{(A_\mu(t), \p_t A_\mu(t))- (A_\mu^\infty(t),\p_t A_\mu^\infty(t))}_{\doth^{m}\times \doth^{m-1}} \les \ve_0 \bra{t}^{-\de},
	\end{align}
	for $m \in [1, N(2)]$.\\

$(5)$ Within the existence time the solution has the same regularity as the initial data. We brief on this in Section \ref{main proof}.\\

$(6)$ Some exponents appearing in Theorem \ref{mainthm} are obtained roughly. For instance, the regularity exponent $N(n)$ could be improved. In this paper, we set $\de = 10^{-10}$, $\zeta = 1050 \de$, and $\ol{\de} = 410 \de$ and we will not pursue the sharpness of these parameters.\\

 $(7)$ We tried to find the minimal order of $\cV_n (n \in \{0, 1, \cdots, n_0\})$, and $n_0 = 3$ seems to be best in our analysis. One may consider higher orders $n_0 \ge 4$ by adjusting $N(n)$ and $\de$ suitably. In this higher order case, one may get $t^{-\frac32 + \ol \de}$ decay for Dirac part.\\ %$N(n) = 10 + 20(n_0 - n)$ and $\de = 10^{-10(n_0-2)}$.

%	$(7)$ In view of the standard dispersive estimates for the Klein-Gordon and wave equations, our decay results \eqref{mainthm:decay} are nearly optimal for the wave part but not for the Dirac part. However, once the order $n_0$ is raised up to $7$, by the classical decay estimate for Klein-Gordon equation (for instance see Appendix B of \cite{{geor1991}}) together with energy estimates, Propositions \ref{prop:energy} and \ref{prop:weighted} below, we can show the pointwise bound
%$$
%|\psi(t, x)| \les \ve_0(1 + t + |x|)^{-\frac32 + \de'} \;(0 < \de' = \de'(\de) \ll 1).
%$$

%}	
\end{rem}

\subsection{Main idea}\label{sec:ideas}
We prove the existence of global solution and the asymptotic behavior with a bootstrap argument starting from an a priori assumption with energy and weighted energy estimates, which are equipped with differential operators $\cL \in \cV _n$ and scattering norms that we devise to control Fourier amplitude throughout the paper (see \eqref{assum:energy}--\eqref{assum:s-norm}). Our argument is summarized as the vector-field energy method, which enables us to handle the higher-order energy estimates by the normal form approach exploiting space-time resonances suitably.

%By observing the nonlinearities of \eqref{md},
The nonlinear terms of \eqref{maineq:md-lorenz} give rise to several oscillatory integrals with the stationary points of the phase interactions. Our analysis begins with addressing the space, time, and space-time resonances of phase functions originating from the Dirac and wave propagators. %within the oscillatory integrals.
Specifically, through a standard reformulation with the Dirac projection operator and first-order wave equations (see Section \ref{sec:reformul} below), we obtain the nonlinearities $\mathbf{\Phi}_{\Theta}$ and $\mathbf{V}_{\mu, \Theta'}$ for the profiles $\phi_\thez = e^{\thez it\brad}\psi_{\thez}$  and $V_{\mu, \thez'} = |D|^{\frac12}e^{-\thez' it|D|}A_{\mu, \thez'}$, respectively:
\begin{align}\label{intro:nonlinear}
	\begin{aligned}
		&\int_{\R^{1+3}} e^{is\,p_\Theta(\xi, \eta)  }\Pi_{\theta_0}(\xi) |\eta|^{-\frac12}\widehat{ V_{\mu, \theta_2}}(s, \eta) \alpha^\mu \widehat{ \phi_{\theta_1}}(s, \xi-\eta)\,d\eta ds, \;\; p_\Theta(\xi, \eta) = \theta \bra{\xi} - \theta_1\bra{\xi-\eta} + \theta_2|\eta|,\\
		&\int_{\R^{1+3}} |\xi|^{-\frac12} e^{is q_{\Theta'}(\xi,\eta)} \bra{\widehat{\phi_{\theta_1'}}(s,\eta),\alpha_\mu\wh{\phi_{\theta_2'}}(s,\xi+\eta)}\, d \eta ds, \;\; q_{\Theta'}(\xi, \eta) = -\theta' |\xi| + \theta'_1\bra{\eta} - \theta'_2 \bra{\xi+\eta},
	\end{aligned}
\end{align}
where 3-tuples $\Theta = (\theta, \theo, \thet)$ and $\Theta' = (\theta', \theo', \thet')$ where $\thej, \thej' \in \{+,-\}\,(j=1,2)$. We call the functions $p_\Theta$ and $q_{\Theta'}$ \emph{phase interactions} in this paper. As we can observe obviously, the nonlinearity of the Dirac part and wave part consists of Dirac-wave interaction and Dirac-Dirac interaction, respectively. Hence the phase interactions exhibit various sign relations, requiring consideration of different resonance cases.

While the phases experience space-time resonance, the nonlinearities of \eqref{maineq:md-lorenz} lack sufficient null structures to eliminate these resonances. As observed in \cite{anfosel2010, huhoh2016, gaoh, CKLY2022, lee2023}, the relation between projection operators and Dirac matrices $\al_j$ induces sign-changed projection parts and Riesz transform parts:
\begin{align}\label{eq:lack-null}
	\al^j \Pi_{\theta}(\xi) =  \Pi_{-\theta}(\xi)\al^j + \theta \frac{\xi_j}{\bra{\xi}}.
\end{align}
Depending on the sigh relation, some terms with null structures emerge, whereas the interaction generates terms without null structures, regardless of the sign relation. To address the lack of null structure, we use a feature of Lorentz invariance in the Lorenz gauge. More precisely, we exploit the vector-field energy method alongside the space-time resonance argument. It is crucial to introduce the vector fields in the energy estimates in order to avoid the time growth, which is a cost of weighted estimates occurring in the space-time resonance estimates under the a priori assumptions. In fact, the number of weights implies the same number of time growth in the Duhamel's formula. %, the more weights are used, the more time growth it should be treated.
However, by imposing the vector fields in both energy and weighted estimates, we can convert weighted estimates into nonlinear estimates without additional time growth arising from Lorentz boosts (see \eqref{eq:weight-lorentz-dirac} and \eqref{eq:weight-lorentz-wave}). Since the Dirac operator does not commute with Lorentz boots in contrast to the Klein-Gordon equation, % and the Dirac projection does not commute with both the rotation and Lorentz boosts.
we find out the exact additional terms coming from the commutator between vector fields and Dirac operators (see Section \ref{sec:vector}). %It turns out that the additional terms are very smooth and have no harm in the energy estimates. Hence we can handle them as a smooth error through the entire estimates.
It turns out that the additional terms have no harm in the energy estimates.  These terms arising in the exchange procedure can be decomposed into two parts. One part comprises vector fields of lower order, while the other part includes at least one time-translation vector field. In particular, the latter does not involve Lorentz boosts. Therefore, we can treat these terms as an easier case due to the lower order of vector fields throughout the entire estimates.

To show the asymptotic behavior of spinor, we need to control the scattering norm $\|\cdot\|_{\bm D}$ as defined in \eqref{def:s-norm-dirac}. At the same time, a phase modification is required. In view of the dispersive estimates for Dirac and Maxwell equations, the nonlinearities are not integrable in time due to the resonance in the infrared regime. Hence the phase modification is necessary to get rid of the resonance effect of slow decay of the low frequency part of gauge fields $A_\mu$ associated with the spinor. Indeed, the phase interaction of Dirac part has a decomposition $p_\Theta= p_{res} + p_{non}$, in which $p_{res} = \theta \frac{s\xi}{\bra{\xi}}$ (when $\theta = \theta_1)$ has the most strong resonance and $p_{non}$ can give a suitable decay effect (see \eqref{eq:phase-approximation}). By cutting off the low frequency part from $A_\mu$ as stated in the main theorem, we will see that
$$
\int_0^t\left[\mathbf \Phi_{\,\Theta}(t, \xi) - \int_{\R^3} e^{\theta is\,\frac{\xi\cdot\eta}{\bra{ \xi}} }\Pi_{\theta}(\xi)\widehat{P_{\le K} A_{\mu}}(s, \eta) \alpha^\mu \widehat{ \phi_{\theta}}(s, \xi)\,d\eta\right] ds = O(t^{-\de})\;\;(t \to \infty).
$$
This implies that
\begin{align*}
	\p_t\wh{\phi_{\theta}}(t, \xi) =  i\Pi_{\theta}(\xi)(P_{\le K} A_{\mu})\left(t, \theta \frac{t\xi}{\bra{\xi}}\right) \alpha^\mu \widehat{ \phi_{\theta}}(t, \xi) + [\mbox{lower order terms}].
\end{align*}
From this idea we construct the phase correction \eqref{mainthm:correction} and learn how to define the scattering norm $\|\cdot\|_{\bm D}$. For more diverse phase corrections, we refer to \cite{deiopau2017,guoiopau2016,hapautzvvis2015}.

Compared to the asymptotic behavior of Dirac spinor, Maxwell part exhibits the linear scattering feature via the scattering norm $\|\cdot\|_{\bm M}$. In view of \eqref{intro:nonlinear}, the obstacle preventing the dispersion effect for scattering is a singularity of the factor $|\xi|^{-\frac12}$. Moreover, another obstacle is the space-time resonance of the phase appearing in \eqref{intro:nonlinear} when $\xi =0$. Fortunately, these obstacles can be eliminated by the norm of homogeneous Sobolev space $\dot H^{\frac12 +\zeta}$ for some $\zeta > 0$. As a consequence we get the anticipated linear scattering like \eqref{mainthm:asymptotic} and \eqref{eq:linear-sc-high} for the Maxwell part.

\subsection{Organization of paper}
This paper is organized as follows. In Section 2 the half-Klein-Gordon and half-wave  equations of \eqref{maineq:md-lorenz} are derived via projections and also a system of profiles is considered. We classify the resonance sets of phase interactions. Then in order to describe the fields and profiles with differential operators, we look into the commutator $[\cL, \Pi_\theta]$ for $\cL \in \cV_n$, which generates nonlinear terms of smooth differential operators essentially in $\cV_n$. In Section 3 we provide a bootstrap argument. For this purpose, we define target solution spaces defined by the energy of fields and weighted profiles equipped with differential operators $\cL$, and to close bootstrapping we define control norms through the phase-space localization. Then we establish linear estimates under bootstrapping assumptions, which consists of estimates on the time decay and on the localized profile. In Section 4 we consider nonlinear estimates based on the bootstrapping assumptions and linear estimates established in Section 3. Since the nonlinear estimates are carried out in $L^2$ space, the nonlinear terms generated by commutator turn out be harmless throughout the whole estimates. In Section 5 we prove the main theorem by assuming bootstrap argument. The Section 6 is devoted to proving the energy estimate parts of bootstrap and the Section 7 to proving the estimates of weighted profiles. In Section 8, 9 we show the asymptotic behaviors of spinor field and gauge field, respectively, from which we can control the scattering norms and close the bootstrap argument. We describe how to extract the phase modification of the Dirac spinor from the resonance interaction between spinor and gauge fields in Section 8. We also show the linear scattering results for gauge fields in Section 9.

\subsection{Notations}
\begin{enumerate}
	\item (Mixed-normed spaces) For a Banach space $X$ and an interval $I$, $u \in L_I^q X$ iff $u(t) \in X$ for a.e. $t \in I$ and $\|u\|_{L_I^qX} := \|\|u(t)\|_X\|_{L_I^q} < \infty$. Especially, we abbreviate $L^p=L_x^p$ for the spatial norm and indicate the subscripts for only Fourier space norm. \\
	
	\item (Sobolev spaces)  Let $s \in \R$. We define homogeneous Sobolev spaces $\|u\|_{\doth^s} := \||D|^s u\|_{L^2}$.  For inhomogeneous spaces, we define $\|u\|_{H^s} := \|\bra{D}^s u\|_{L^2}$.\\

	\item Different positive constants depending only on $n$, $\ve_0$ are denoted by the same letter $C$, if not specified. $A \lesssim B$ and
	$A\gtrsim B $ mean that $A \le CB$ and $A \geq  C^{-1}B$  respectively for some $C>0$. $A \sim B$  means  that $A \lesssim B$ and $A \gtrsim B$.\\
	
	\item (Littlewood-Paley operators) Let $\rho$ be a
	Littlewood-Paley function such that $\rho\in C_{0}^{\infty}(B(0,2))$
	with $\rho(\xi)=1$ for $|\xi|\le1$ and define $\rho_{k}(\xi):=\rho\left(\frac{\xi}{2^k}\right)-\rho\left(\frac{\xi}{2^{k-1}}\right)$
	for $2 \in \Z$. Then we define the frequency projection
	$P_{k}$ by $\mathcal{F}(P_{k}f)(\xi)=\rho_{k}(\xi)\widehat{f}(\xi)$,
	and also  $\rho_{\le k}:=\sum_{k'\le k}\rho_{k'}$ and  $\rho_{> k} =  1- \rho_{\le k}$. Then we also denote $\cF (P_{\le k} f)(\xi)  = \rho_{\le k}(\xi) \wh{f}(\xi) $ and $\cF (P_{> k} f)(\xi)  = \rho_{> k}(\xi) \wh{f}(\xi) $. For $k\in {\mathbb{Z}}$
	we denote $\rho_{[\ell_1, \ell_2]} = \sum_{\ell_1 \le k \le \ell_2}\rho_k$ and $\widetilde{\rho}_{k} := \rho_{[k-2, k+2]}$. In
	particular, $\widetilde{P_{k}}P_{k}=P_{k}\widetilde{P_{k}}=P_{k}$
	where $\widetilde{P_{k}}=\mathcal{F}^{-1}\widetilde{\rho_{k}}\mathcal{F}$.
	%Especially, we denote $P_{k}f$ by $f_{k}$ for any measurable function
	%$f$.
	\\

	%\item Let $k^- := \min(k, 0)$ and $k^+ := \max(k, 0)$ for any $k \in \mathbb Z$.\\

		\item (Localization on phase space) Let $\cU_k = \{j \in \Z : k+j \ge 0\,(k \le 0)\mbox{ and } j=0\,(k \ge 1)\}$ for fixed $k\in \Z$. Then $j \in \cU_k$ means $j \ge -\min(k, 0)$. For any $j \in \cU_k$, let
		$$
		\bar \rho^{(k)}_j(x) = \left\{\begin{array}{ll} \rho_{\le -k}(x) & \mbox{if}\; k+j = 0 \;\mbox{and}\;k \le 0,\\
			\rho_{\le 0}(x)  & \mbox{if}\; j = 0 \;\mbox{and}\; k \ge 1,\\
			\rho_j(x)        & \mbox{if}\; k+j \ge 1 \;\mbox{and}\;j \ge 1.
		\end{array}\right.
		$$
		Then, $\sum_{j \in \cU_k}\bar \rho_j^{(k)} = 1$. For $k \in \Z, \, j \in \cU_k$, let $Q_{jk}$ denote the phase space localization operator
		$$
		Q_{jk}f(x) = \bar \rho_j^{(k)}(x) P_kf(x).
		$$
We note that the uncertainty principle allows us to consider $\qjk$ only when $2^{j+k} \ge 1$.
\end{enumerate}

\section{Preliminaries}\label{sec:pre}

\subsection{Set up for the Maxwell-Dirac system}\label{sec:reformul}
By the definition of Dirac projection \eqref{eq:projection},  $\Pi_\theta$ satisfies the properties:
\[
\Pi_\theta + \Pi_{-\theta} = I_4, \quad \Pi_\theta \Pi_{-\theta} = 0, \quad \mbox{ and }\quad \Pi_\theta\Pi_\theta = \Pi_\theta.
\]
We can decompose the Dirac spinor $\psi$ into half Klein-Gordon waves, i.e., $\psi_+$ and $\psi_-$ by projection operator $\Pi_\theta$.
Let $\psi_\theta = \Pi_\theta \psi$ for $\theta \in \{+,-\}$. Then we obtain the following decoupled equations from the Dirac part of \eqref{maineq:md-lorenz}:
\begin{align*}%\label{eq:half-dirac}
	\left\{
	\begin{aligned}
		(-i\partial_t + \theta\brad)\psi_\theta &= \Pi_\theta\big(A_\mu \alpha^\mu \psi\big),\\
		\psi_\theta(0)&:= \psi_{0, \theta}.
	\end{aligned}\right.
\end{align*}
Note that $\psi = \psi_+ + \psi_-$.

We now decompose gauge field as $A_\mu = A_{\mu,+} + A_{\mu,-}$ with
$$
A_{\mu, \theta'} = \frac12\Big(1 + \theta' |D|^{-1}(-i\partial_t)\Big)A_\mu,
$$
for $\theta' \in \{+, -\}$. Then the Maxwell part of \eqref{maineq:md-lorenz} is rewritten as the following equations:
\begin{align*}%\label{eq:half-maxwell}
	\left\{\begin{aligned}
		(i\partial_t + \theta'|D|)A_{\mu, \theta'} &= \theta'\frac{1}2|D|^{-1}\bra{\psi, \alpha_\mu \psi},\\
		A_{\mu, \theta'}(0) &:= a_{\mu, \theta'} := \frac12(a_{ \mu} - \theta' i|D|^{-1}\dot a_{ \mu}).
	\end{aligned}\right.
\end{align*}
Therefore \eqref{maineq:md-lorenz} becomes the system
\begin{align*}
	\left\{\begin{aligned}
		(-i\partial_t + \theta\bra{D})\psi_\theta &= \Pi_\theta\big(A_\mu \alpha^\mu \psi\big),\\
		(i\partial_t + \theta'|D|)A_{\mu, \theta'} &= \theta'\frac{1}2|D|^{-1}\bra{\psi, \alpha_\mu \psi}
	\end{aligned}\right.
\end{align*}
with initial data $(\psi_\theta(0), A_{\mu,\theta'}(0)) = (\psi_{0,\theta}, a_{\mu,\theta'})$, for $\theta, \theta' \in \{+,-\}$. Defining $W_{\mu,\theta'}:= |D|^\frac12 A_{\mu,\theta'}$ ($W_{\mu}:= W_{\mu,+} + W_{\mu,-}$), we also have
\begin{align}\label{eq:maineq-half}
	\left\{\begin{aligned}
		(-i\partial_t + \theta\bra{D})\psi_\theta &= \Pi_\theta\left[ \left(|D|^{-\frac12}W_\mu\right) \alpha^\mu \psi\right],\\
		(i\partial_t + \theta'|D|)W_{\mu, \theta'} &= \theta'\frac{1}2|D|^{-\frac12}\bra{\psi, \alpha_\mu \psi}.
	\end{aligned}\right.
\end{align}

Now by Duhamel's principle, \eqref{eq:maineq-half} can be converted into
\begin{align*}
	\psi_{\theta}(t) &= e^{-\theta it\brad} \psi_{0, \theta}  +   i \sum_{\theta_1, \theta_2 \in \{\pm \}}\int_0^t e^{-\theta i(t-s)\brad}\Pi_{\theta}\left(\left( |D|^{-\frac12}W_{\mu, \theta_2} \right)\alpha^\mu \psi_{\theta_1}\right)(s)\,ds, \\%\label{eq:duhamel-dirac}\\
	W_{\mu, \theta'}(t) &= e^{\theta'i t|D|}|D|^{\frac12}a_{\mu, \theta'} -\theta' \frac{i}2\sum_{\theta'_1, \theta'_2 \in \{\pm\}}\int_{0}^{t} e^{\theta'i(t-s)|D|}|D|^{-\frac12}\bra{\psi_{\theta'_1},\alpha_\mu\psi_{\theta'_2}}(s)\,ds.  %\label{eq:duhamel-maxwell}
\end{align*}
To keep track of the scattering state, we define profile fields $\phi_{\theta}$ and $V_{\mu, \theta'}$ by
\begin{align*}%\label{eq:interaction}
	\phi_\theta(t):= e^{\theta i t \brad} \psi_\theta(t) \;\;\mbox{ and }\;\; V_{\mu, \theta'}(t) := e^{-\theta' it|D|}W_{\mu, \theta'}(t).	
\end{align*}
%Then by acting propagators on both sides of \eqref{eq:duhamel-dirac} and \eqref{eq:duhamel-maxwell}, respectively, we have
%\begin{align}
%	\phi_{\theta_0}(t) &= \psi_{0, \theta_0} + i\sum_{\theta_1, \theta_2 \in \{\pm \}}\int_0^t e^{\theta_0 is\brad}\Pi_{\theta_0}\left(\left(|D|^{-\frac12}W_{\mu, \theta_2} \right) \alpha^\mu \psi_{\theta_1}\right)(s)\,ds, \label{eq:profile-dirac} \\
%	V_{\mu, \theta'_0}(t) &= |D|^\frac12 a_{ \mu, \theta'_0} - \theta'_0 \frac{i}2\sum_{\theta'_1, \theta'_2 \in \{\pm \}} \int_{0}^{t} e^{-\theta'_0is|D|}|D|^{-\frac12}\bra{\psi_{\theta'_1},\alpha_\mu\psi_{\theta'_2}}(s)\,ds,\label{eq:profile-maxwell}
%\end{align}
%for $\thez, \thez' \in \{+,-\}$.
By taking Fourier transform and time derivative, one gets the frequency representation as follows:
\begin{align*}\begin{aligned}%\label{eq:profile-eqn}
		\p_t\widehat \phi_{\theta}(t, \xi) &=  \frac{i}{(2\pi)^3}\sum_{\theta_1, \theta_2 \in \{ \pm \}}\mathbf \Phi_{\,\Theta}(t, \xi)\;\;(\Theta = (\theta, \theta_1, \theta_2)), \\
		\p_t\widehat{ V_{\mu, \theta'}}(t, \xi) &=  - \theta'\frac{i}{2(2\pi)^3} \sum_{\theta'_1, \theta'_2 \in \{ \pm \}} \mathbf V_{\,\mu, \Theta'}(t, \xi)\;\;(\Theta' = (\theta',\theta_1',\theta_2')),
\end{aligned}\end{align*}
where
\begin{align}
	\mathbf \Phi_{\,\Theta}(t, \xi) &:= \int_{\R^3} e^{is\,p_\Theta(\xi, \eta)  }\Pi_{\theta}(\xi)|\eta|^{-\frac12}\widehat{ V_{\mu, \theta_2}}(s, \eta) \alpha^\mu \widehat{ \phi_{\theta_1}}(s, \xi-\eta)\,d\eta,\label{eq:interaction-dirac}\\
	\mathbf V_{\,\mu, \Theta'}(t, \xi) &:= |\xi|^{-\frac12}\int_{\R^3} e^{is q_{\Theta'}(\xi,\eta)} \bra{\widehat{\phi_{\theta_1'}}(s,\eta),\alpha_\mu\wh{\phi_{\theta_2'}}(s,\xi+\eta)}\, d \eta,\label{eq:interaction-maxwell}\\
	p_\Theta(\xi, \eta) &:= \theta \bra{\xi} - \theta_1\bra{\xi-\eta} + \theta_2|\eta|,\label{eq:phase-dirac}\\
	q_{\Theta'}(\xi, \eta) &:= -\theta' |\xi|   + \theta'_1\bra{\eta} - \theta'_2 \bra{\xi+\eta}.\label{eq:phase-maxwell}
\end{align}

\begin{rem}
		 Our energy estimates rely on the quasilinear characteristics of \eqref{md}, rather than directly utilizing the specific forms given in \eqref{eq:interaction-dirac} and \eqref{eq:interaction-maxwell}. However, from the significance of observing the oscillations within the nonlinearities of \eqref{md}, we include the forms\eqref{eq:interaction-dirac} and \eqref{eq:interaction-maxwell} for convenience.
\end{rem}

\subsection{Space-time resonance for Maxwell-Dirac system}\label{sec:resonance} We proceed our proof by a bootstrap argument based on the space-time resonance argument
\cite{gemasha2008,gemasha2012-jmpa,gemasha2012-annals} (see also \cite{gunakatsa2009}). In this section we investigate  time, space, and space-time resonances for \emph{massive} Maxwell-Dirac system, respectively. These types of resonance have been already investigated in \cite{lee2023}. For the readers' convenience, we mention them again in this section. For the time resonant set of wave and Klein-Gordon type systems, we refer to \cite{iopau2019,book:iopau2022}. We also refer to \cite{pusa,canher2018-analpde} for the analogous observation of Dirac operator. According to the definition in \cite{gemasha2008}, we define time, space, and space-time resonant sets as follows:
\begin{align*}
	\begin{aligned}
		\mathcal T_{p_\Theta} &:= \left\{ (\xi,\eta) : p_\Theta(\xi,\eta)=0\right\},\\
		\mathcal S_{p_\Theta} &:= \left\{ (\xi,\eta) : \nabla_\eta p_\Theta(\xi,\eta)=0\right\},\\
		\mathcal R_{p_\Theta} &:= \mathcal T_{p_\Theta} \cap \mathcal S_{p_\Theta},
	\end{aligned}\hspace{2cm} \begin{aligned}
		\mathcal T_{q_{\Theta'}} &:= \left\{ (\xi,\eta) : q_{\Theta'}(\xi,\eta)=0\right\},\\
		\mathcal S_{q_{\Theta'}} &:= \left\{ (\xi,\eta) : \nabla_\eta q_{\Theta'}(\xi,\eta)=0\right\},\\
		\mathcal R_{q_{\Theta'}} &:= \mathcal T_{q_{\Theta'}} \cap \mathcal S_{q_{\Theta'}}.
	\end{aligned}
\end{align*}
To identify these sets, it suffices to examine the lower bounds of   phase interactions $p_\Theta, q_{\Theta'}$ and their gradients $\nabla_\eta p_\Theta, \nabla_\eta q_{\Theta'}$.

Let us observe that $p_{\Theta}(\xi, \eta)$ never vanishes for any sign $\theta,\theo,\thet \in \{+, -\}$, when $\xi\neq 0$ and $\eta \neq 0$. In fact, if for some $\xi, \eta$, $p_{\Theta}(\xi, \eta) = 0$, then
$$
(\theta\bra\xi + \theta_2|\eta|)^2 = 1 + |\xi -\eta|^2.
$$
This implies that
$$
\braxi|\eta| = |\xi \cdot \eta|\;\;\mbox{and hence}\;\; \braxi \le |\xi|,
$$
which is a contradiction. Hence we have
\begin{align*}
	|p_\Theta(\xi, \eta)| &=  \left|\theta\braxi - \theta_1\bra{\xi-\eta} + \theta_2|\eta|\right| =  \left|\frac{(\theta\braxi + \theta_2|\eta|)^2 - \bra{\xi-\eta}^2 }{\theta\braxi  + \theta_1\bra{\xi-\eta} + \theta_2|\eta| }\right|\\
	&=  \left|\frac{2\theta\theta_2\braxi|\eta| + 2 \xi \cdot \eta }{\theta\braxi  + \theta_1\bra{\xi-\eta} +\theta_2|\eta|  }\right|.
\end{align*}
If $\theta = \theta_1$, then
\begin{align}\label{eq:resonance-time}\begin{aligned}
		|p_\Theta(\xi,\eta)| &\ge \frac{2|\eta|(\braxi - |\xi|)}{\braxi + \bra{\xi-\eta} + |\eta|}\gtrsim \frac{|\eta|}{\bra{\xi}(\bra{\xi}+\bra{\xi-\eta}+
			\bra{\eta})}.
\end{aligned}\end{align}
If $\theta\neq \theta_1, \theta \neq \theta_2$, then
\begin{align*}
	|p_\Theta(\xi,\eta)| \ge \frac{2|\eta|(\braxi - |\xi|)}{|\braxi - \bra{\xi-\eta} - |\eta||} \gtrsim (\braxi - |\xi|) \gtrsim \bra{\xi}^{-1}.
\end{align*}
Also if $\theta \neq \theta_1, \theta = \theta_2$, then trivially $|p_\Theta(\xi,\eta)| \ge \bra{\xi} +\bra{\xi-\eta}$. Then we have, for $\thez \neq \theo$,
\begin{align*}%\label{eq:nonresonance-time}
	|p_\Theta(\xi,\eta)|  \gtrsim \bra{\xi}^{-1}.
\end{align*}
Therefore $\mathcal{T}_{p_\Theta} = \{\eta=0\}$ when $\thez =\theo$ and $\mathcal {T}_{p_\Theta} = \emptyset$, otherwise.

 Now one can readily observe that there is no space resonance of \eqref{eq:interaction-dirac} due to the fact
\begin{align}\label{eq:nonresonance-space}
	\left| \nabla_\eta p_\Theta (\xi,\eta)\right| = \left| \theo\frac{\xi-\eta}{\bra{\xi-\eta}} - \thet \frac{\eta}{|\eta|}\right| \gtrsim 1 - \frac{|\xi-\eta|}{\bra{\xi-\eta}} \gtrsim \bra{\xi-\eta}^{-2}.
\end{align}
Therefore   $\mathcal{S}_{p_\Theta} = \emptyset$ for any sign and hence $\mathcal{R}_{p_\Theta} = \emptyset$.

By virtue of the similarity between $p_\Theta(\xi,\eta)$ and $q_{\Theta'}(\xi,\eta)$, we readily see that
\begin{align}\label{eq:resonance-time-maxwell}
	|q_{\Theta'}(\xi,\eta)| \gtrsim \begin{cases}
		\frac{|\xi|}{\bra{\eta}(\bra{\xi}+\bra{\xi-\eta} +\bra{\eta})}  & \mbox{ if } \theta_1' = \theta_2',\\
		\bra{\eta}^{-1}  & \mbox{ otherwise},
	\end{cases}
\end{align}
and a direct calculation shows that
\begin{align}\label{eq:resonance-space-maxwell}
	|\nabla_\eta q_{\Theta'}(\xi,\eta)| \gtrsim \begin{cases}
		\frac{|\xi|}{\bra{\xi+\eta}^3}  & \mbox{ if } \theta_1' = \theta_2',\\
		\left|\frac{\eta}{\bra{\eta}} + \frac{\xi+\eta}{\bra{\xi+\eta}} \right|  & \mbox{ otherwise}.
	\end{cases}
\end{align}
Note that space resonance of $q_{\Theta'}$ is different from that of $p_\Theta$. Therefore, we have $\mathcal T_{q_{\Theta'}} = \mathcal S_{q_{\Theta'}} =\mathcal R_{q_{\Theta'}} = \{\xi=0\}$ when $\theta_1' =\theta_2'$.
On the other hand,  if $\theta_1' \neq \theta_2'$, then
$\mathcal T_{q_{\Theta'}} = \emptyset$,  $\mathcal S_{q_{\Theta'}}= \{\xi=-2\eta\}$, and $\mathcal R_{q_{\Theta'}}= \emptyset$.

\subsection{Vector fields on Dirac and Maxwell parts}\label{sec:vector}
By the Lorentz invariance the wave operator $\square$ commutes with $\cL$ for any $\cL \in \cV_n$. Hence
\begin{align}\label{vec-wave eqn}
	-\square \cL W_{\mu} =  |D|^{\frac12}\cL\bra{\psi,\al_\mu \psi}.
\end{align}
On the other hand, the vector fields of rotations and boosts do not commute with both of the half Klein-Gordon operators $-i\p_t + \theta \brad$ and projection operators $\Pi_\theta$, we need to investigate carefully the commutators between vector fields and those operators.

\emph{Rotations.} The vector fields $\Omega$ commute with the half Klein-Gordon operator. Indeed,
\begin{align*}
		(-i\p_t  +\theta \bra{D})\Omega_{jk}\psi_\theta = \Omega_{jk} \Pi_\theta(A_\mu \al^\mu \psi).
\end{align*}
On the other hand, $\Om$ does not commute with the projection operators and  the commutator is the following:
\begin{align}\label{eq:comm-rot-proj}
	\left[\Om_{jk},\Pi_{\theta} \right] = %\theta \frac12 \left[  \left(\frac{\p_j^2}{\brad^3} - \frac{\p_k^2}{\brad^3}\right) +  \left( \frac{(-i\al\cdot \nabla)\p_j^2}{\brad^3} - \frac{(-i\al\cdot \nabla)\p_k^2}{\brad^3} \right) +i \left( \frac{\al_j \p_j}{\brad} - \frac{\al_k \p_k}{\brad}\right) \right]\\
	\theta\frac{i}{2}\left(\frac{\al_j \p_k}{\brad} - \frac{\al_k\p_j}{\brad}\right).%  =: \theta \cO_{jk}.
\end{align}

\emph{Lorentz boosts.} The Lorentz vector field $\Gamma$ does not commute with half-Klein-Gordon operators and their commutators are as follows:
\begin{align*}
		(-i \p_t +\theta \brad)\Gamma_j\psi_\theta &= \Gamma_j (-i\p_t +\theta \brad)\psi_\theta -i \p_j \psi_\theta - \theta \frac{\p_j}{\brad}\p_t \psi_\theta.
\end{align*}
The first formula implies that
\begin{align*}
	(-i\p_t +\theta \brad)\Gam_j\psi_\theta &= \Gam_j \Pi_\theta(A_\mu \al^\mu \psi) -\theta  \frac{\p_j}{\brad}\Pi_\theta(A_\mu \al^\mu \psi).
\end{align*}
A direct calculation shows that
%\begin{lemma}\label{lem:comm-lorentz}
 for $j=1,2,3$,
	\begin{align}\label{eq:comm-lorentz-proj}
		[\Gam_j, \Pi_\theta ]  =  \theta  \frac i2 \left( \frac{\al_j}{\brad} + i\frac{(\al\cdot D)\p_j + m\beta \p_j}{\brad^3} \right) \p_t.%  =: \theta \cG_{j}\p_t.
	\end{align}	
%\end{lemma}
%
%{\color{black} Let us define the extended differential operators and the corresponding set $\cW_n$ as follows:
	%	\begin{align}\label{}
		%		\cW_n := \left\{   [\p]^a[\Om]^b[\cG]^c  \mbox{ or }  [\p]^a  [\cO]^b [\Gam]^c : |a|+|b|+|c| \le n    \right\},
		%	\end{align}
	%	where we denoted $ [\cO]^b = \cO_{23}^{b_1}\cO_{31}^{b_2}\cO_{12}^{b_3}$ and $[\cG]^c = \cG_1^{c_1}\cG_2^{c_2}\cG_1^{c_3}$.
	Then for any $\cL \in \cV_n$ it follows  from \eqref{eq:comm-rot-proj} and \eqref{eq:comm-lorentz-proj} that
	\begin{align}\label{eq:commu-all}
		\begin{aligned}
			(-i\p_t + \theta \brad)\cL \psi_\theta &= \Pi_\theta \cL (A_\mu \al^\mu \psi) + [\cL, \Pi_\theta](A_\mu \al^\mu \psi)\\
			&= \Pi_\theta \cL (A_\mu \al^\mu \psi) + \sum_{\cL' \in \cV_n}C(\cL')\cR_{\cL'}\cL'(A_\mu \al^\mu \psi).
			%\theta [\p]^a[\Om]^b[\cG]^c \p_0^{|c|}(A_\mu \al^\mu \psi)  + \theta [\p]^a  [\cO]^b [\Gam]^c(A_\mu \al^\mu \psi)\\
			%		& \hspace{0.5cm}  -\theta [\cG]^{c}\p_0^{|c|} \Pi_\theta(A_\mu \al^\mu \psi) + \theta  \sum_{j=1}^3c_j[\p]^a [\cO]^b R_j^{c_j}(A_\mu \al^\mu \psi),
			%\end{aligned}
	\end{aligned}\end{align}
	%	by \eqref{eq:commute-rotation}, \eqref{eq:commute-lorentz}, \eqref{eq:commu-r}, and \eqref{eq:commu-l}, we deduce that there exist $\widetilde{\cL}_1, \widetilde{\cL}_2 \in \cW_n$ such that
	%	\begin{align}
		%		\begin{aligned}
			%			(-i\p_t + \theta \brad)\cL \psi_\theta &= \Pi_\theta \cL (A_\mu \al^\mu \psi) + \theta \left( \p_0^{|c|} \widetilde{\cL}_1  + \widetilde{\cL}_2\right)  (A_\mu \al^\mu \psi)  \\
			%			& \hspace{0.5cm}  -\theta [\cG]^{c}\p_0^{|c|} \Pi_\theta(A_\mu \al^\mu \psi) + \theta  \sum_{j=1}^3c_j[\p]^a [\cO]^b R_j^{c_j}(A_\mu \al^\mu \psi),
			%		\end{aligned}
		%	\end{align}
	Here, $\cR_{\cL'}$ are the operator-valued coefficients bounded in $L^p (1 < p < \infty)$ of the form $R^{c}$, where $C(\cL')$ is a $4\times 4$ matrix defined by multiplications of $\al^\mu$ and $\beta$, $c = (c_1, c_2, c_3)$, $R^c = R_1^{c_1}R_2^{c_2}R_3^{c_3}$, and $R_j$ are Riesz type transforms defined by $\frac{-i\p_j}{\brad}$. The coefficients $C$ and multi-indices $c$ may vary in $\cL' \in \cV_n$. In fact,	let us denote $\cO_{jk} := \theta\left[\Om_{jk},\Pi_{\theta} \right]$ and $\cG_{j}\p_t := \theta[\Gam_j, \Pi_\theta ]$. Then
	for any $\cL = [\p]^a[\Om]^b[\Gam]^c \in \cV_n$, the commutator can be written as
	\begin{align*}
		[\cL, \Pi_\theta] = &\theta [\p]^a[\Om]^b[\cG]^c \p_0^{|c|}  + \theta [\p]^a  [\cO]^b [\Gam]^c -\theta [\cG]^{c}\p_0^{|c|}  + \theta  \sum_{j=1}^3c_j[\p]^a [\cO]^b R_j^{c_j},
	\end{align*}
	where $ [\cO]^b = \cO_{23}^{b_1}\cO_{31}^{b_2}\cO_{12}^{b_3}$ and $[\cG]^c = \cG_1^{c_1}\cG_2^{c_2}\cG_1^{c_3}$.
	Rearranging four terms in order of $\p, \Om, \Gam$, one can get the expression $C(\cL')R^c$. Especially, we note that, for $\cL \in \cV_n$,  there exists $\ol \cL \in \cV_{\ol n}$ such that $n_0,\ol n   \le n-1$ and  
	\begin{align}\label{eq:commu-decom}
		[\cL, \Pi_\theta] =  \cL_0  +  \sum_{ \ol \cL \in \ol{\cV}_{\ol n}}C(\ol \cL) \cR_{\ol \cL} \ol \cL \p_t ,
	\end{align}
where
\begin{align*}
\ol{ \cV}_{0} := \cV_0 \;\;\mbox{ and }\;\;	\ol{ \cV}_{\ol n} := \{ [\p]^a [\Om]^b: |a|+|b| \le  \ol n \}.
\end{align*}

\begin{rem}
In \eqref{eq:commu-decom}, $\cL_0$ consists of lower order vector fields compared to $\cL$. Similarly, $\ol \cL$ in the second term of the right-hand side of \eqref{eq:commu-decom} also comprises lower order vector fields. As well as, the Lorentz vector field $\Gam$ is not included in $\ol \cL$. This fact implies that $\ol \cL$ commutes not only with $\p_t$, but also with the differential operators $|D|$ and $\bra{D}$.
\end{rem}

%%%%%%%%%%%%%%%%%%%%%%%%%%%%%%%%%%%%%%%%%%%%%%%%%%%%%%%%%%%%%%%%%%%%%%%%%%%%%%%%%%%%%%%%%%%%%%%%%%%%%%%%%%%%%%%%%%%%%%%%%%%%%%%%%%%%%%%%%%%%%%%%%%%%%%%%%%%%%%%%%%%%%%%%%%%%%%%%%%%%%%%%%%%%%%%%%%%%%%%%%%%%%%%%%%%%%%%%%%
%\begin{rem}
%	In the right-hand side of \eqref{eq:commu-all}, only standard vector fields appear in the first term. On the other hand, other vector fields  produced by commutators  appear in the last three terms.
%\end{rem}
%%%%%%%%%%%%%%%%%%%%%%%%%%%%%%%%%%%%%%%%%%%%%%%%%%%%%%%%%%%%%%%%%%%%%%%%%%%%%%%%%%%%%%%%%%%%%%%%%%%%%%%%%%%%%%%%%%%%%%%%%%%%%%%%%%%%%%%%%%%%%%%%%%%%%%%%%%%%%%%%%%%%%%%%%%%%%%%%%%%%%%%%%%%%%%%%%%%%%%%%%%%%%%%%%%%%%%%%%%

 Throughout the paper we use the following notations for the fields and profiles equipped with vector fields: For $\cL \in \cV_n$,
	\begin{align*}%\label{not-vector-field}
			&\qquad\qquad\qquad\qquad \psi_{\theta,\cL} := \cL \psi_{\theta},\quad \phi_{\theta, \cL} := e^{\theta it\brad}\psi_{\theta, \cL}, \\
			&\quad\qquad\qquad A_{\mu, \cL} := \cL A_\mu,\quad A_{\mu, \cL, \theta'} := \frac12 \left(1 + \theta' \frac{-i\p_t}{|D|}\right) A_{\mu, \cL},\\
			&W_{\mu, \cL} := |D|^\frac12 A_{\mu, \cL},	\quad W_{\mu, \cL, \theta'} := |D|^{\frac12}A_{\mu, \cL, \theta'}, \quad V_{\mu, \cL, \theta'} := e^{-\theta' i t|D|}W_{\mu, \cL, \theta'}.
	\end{align*}

	The Lorentz boosts play the role of weight generators. Indeed, we observe
\begin{align}
	\begin{aligned}
		\Gam_j \psi_\theta &= x_j(-\theta i \brad \psi_\theta + i \Pi_\theta(A_\mu \al^\mu \psi)) + t\p_j \psi_\theta\\
		&=-\theta i e^{-\theta it\brad}x_j \left(\brad \phi_\theta \right) +i x_j \Pi_\theta(A_\mu \al^\mu \psi)
	\end{aligned}\label{eq:weight-lorentz-dirac}
\end{align}
for $j=1,2,3$.  On the other hand, by \eqref{eq:maineq-half} and \eqref{vec-wave eqn}, there holds
	\begin{align}\label{eq:vec-wave}
		(i\p_t + \theta'|D|)W_{\mu, \cL, \theta'} = \theta'\frac12|D|^{-\frac12}\cL\bra{\psi,\al_\mu \psi}.
	\end{align}
	Analogously to \eqref{eq:weight-lorentz-dirac}, we have
%\begin{align*}
%	\begin{aligned}
	%		\Gam_j W_{\mu,\theta'} &= x_j(\theta' i |D| W_{\mu, \theta'} -\theta'  i \frac12 |D|^{-\frac12}\bra{\psi,\al^\mu \psi}) + t\p_j W_{\mu,\theta'}\\
	%&=\theta' i e^{\theta' it|D|}x_j \left(|D| V_{\mu,\theta'} \right) -\theta' i   \frac12 x_j |D|^{-\frac12}\bra{\psi,\al^\mu \psi}.
	%	\end{aligned}
%\end{align*}
%This implies
\begin{align}
	\begin{aligned}\label{eq:weight-lorentz-wave}
		\Gam_j W_{\mu,\cL,\theta'} &= x_j(\theta' i |D| W_{\mu,\cL,\theta'} -\theta'  i \frac12 |D|^{-\frac12}\cL\bra{\psi,\al_\mu \psi}) + t\p_j W_{\mu,\cL,\theta'}\\
		&=\theta' i e^{\theta' it|D|}x_j \left(|D| V_{\mu,\cL,\theta'} \right) -\theta' i   \frac12 x_j |D|^{-\frac12}\cL\bra{\psi,\al_\mu \psi}.
	\end{aligned}
\end{align}

%%%%%%%%%%%%%%%%%%%%%%%%%%%%%%%%%%%%%%%%%%%%%%%%%%%%%%%%%%%%%%%%%%%%%%%%%%%%%%%%%%%%%%%%%%%%%%%%%%%%%%%%%%%%%%%%%%%%%%%%%%%%%%%%%%%%%%%%%%%%%%%%%%%%%%%%%%%%%%%%%%%%%%%%%%%%%%%%%%%%%%%%%%%%%%%%%%%%%%%%%%%%%%%%%%%%%%%%%%%%%%%%%%%%%%%%%%%%%%%%%%%%%%%%%%%%%%%%%%%%%%%%
\newcommand{\bD}{{\bf D}}
\newcommand{\bM}{{\bf M}}
%%%%%%%%%%%%%%%%%%%%%%%%%%%%%%%%%%%%%%%%%%%%%%%%%%%%%%%%%%%%%%%%%%%%%%%%%%%%%%%%%%%%%%%%%%%%%%%%%%%%%%%%%%%%%%%%%%%%%%%%%%%%%%%%%%%%%%%%%%%%%%%%%%%%%%%%%%%%%%%%%%%%%%%%%%%%%%%%%%%%%%%%%%%%%%%%%%%%%%%%%%%%%%%%%%%%%%%%%%%%%%%%%%%%%%%%%%%%%%%%%%%%%%%%%%%%%%%%%%%%%%%%

\section{Linear estimates}\label{sec:linear}

\subsection{A priori assumption}
Here we define the solution spaces for the bootstrap argument. We also recall the regularity depending on the number of vector fields and set time growth index table as follows:
\vspace{-0.8cm}
\begin{center}
	\begin{align*}%\label{table:regularity}
		\begin{tabular}{|c||c|c|c|c|}
			\hline
			n & 0 & 1 & 2 &3 \\
			\hline
			N(n) & 70 & 30 & 20 & 10\\
			\hline
			H(n) & 1 & 10 & 210 & 410 \\
			\hline
		\end{tabular}
	\end{align*}
\end{center}
Let $\ve_1 := \ve_0^\frac23$ for a small $\ve_0 > 0$, $\de = 10^{-10}$. Then the scattering norms $\|\cdot\|_\bD, \|\cdot\|_\bM$ to control the bootstrap argument are defined as follows:
\begin{align}
	\|\varphi\|_{\bD} &:= \sup_{k \in \mathbb Z} \left\{\bra{2^k}^{20}2^{(\frac12-\frac1{100})k} \|\rho_k \widehat \varphi\|_{L^\infty} + \bra{2^k}^{38}2^{-(1-\frac1{100})k}\|P_k \varphi\|_{L^2} \right\},\label{def:s-norm-dirac}\\
	\|v\|_{\bM} &:= \sup_{k \in \mathbb Z} \left\{\bra{2^k}^{25}2^{(1+5H(2)\de)k}\sum_{j \in \cU_k }2^j\|Q_{jk} v\|_{L^2} \right\}.\label{def:s-norm-maxwell}
\end{align}

Given any $T>0$, assume that $(\psi,A_\mu)$ is a solution to \eqref{maineq:md-lorenz} on $[0,T]$ and that, for any $t \in [0,T]$, $n \in \{0, 1, 2, 3\}$, $n_1 \in \{0, 1, 2\}$, $l \in \{1, 2, 3\}$, and $\theta, \theta' \in \{+,-\}$, and the solution satisfies the bootstrap hypotheses
\begin{align}
	&\sup_{\mathcal L \in \mathcal V_n} \left( \normo{ \psi_{\theta, \cl}(t)}_{H^{N(n)}} + \normo{W_{\mu,\cl, \theta'}(t)}_{H^{N(n)}}\right) \le \ve_1 \bra{t}^{H(n)\de},	\label{assum:energy}\\
&\begin{aligned}
		&\sup_{\mathcal L \in \mathcal V_{n_1}} \sup_{k \in \Z} \bra{2^k}^{N(n_1+1)} \left( \bra{2^k} \normo{\rho_k(\xi) (\p_{\xi_l} \wh{\phi_{\theta, \cL}})(t,\xi)}_{L_\xi^2} + 2^{k/2} \normo{\rho_k(\xi) (\p_{\xi_l} \wh{V_{\mu,\cL, \theta'}})(t,\xi)}_{L_\xi^2}\right)\\
	&\qquad\qquad\qquad\qquad\qquad\qquad\qquad\quad\quad\qquad \le \ve_1 \bra{t}^{H(n_1+1)\de},
\end{aligned}	\label{assum:weight}
\end{align}
and
\begin{align}
	\normo{\phi_\theta(t)}_{\bD} + \normo{V_{\mu, \theta'}(t)}_{\bM}\le \ve_1. \label{assum:s-norm}
\end{align}

\subsection{Time decay estimates}
In this section, we give the time decay properties of the Maxwell-Dirac system. To this end, we introduce some abbreviations as follows:
\begin{align*}
	\vp_{j,k}:= \wt P_k Q_{jk}\vp, \quad Q_{\le jk}\vp := \sum_{j' \in \cU_k, j' \le j} Q_{j'k}\vp,  \quad \vp_{\le j,k} := \wt P_k Q_{\le jk}\vp.
\end{align*}

\begin{lemma}
	Let $\vp \in L^2$ and $j \in \mathcal U_k$ for some $k\in \Z$. Then we get the following:
	\begin{itemize}
		\item[(i)] For any $\al \in (\Z_+)^3$, we have
		\begin{align}
			\normo{\nabla_\xi^\al \wh{\vpjk}}_{L_\xi^2} &\les 2^{|\al|j}\normo{\wh{Q_{jk}\vp}}_{L_\xi^2}, \;\;	\normo{\nabla_\xi^\al \wh{\vpjk}}_{L_\xi^\infty} \les 2^{|\al|j}\normo{\wh{Q_{jk}\vp}}_{L_\xi^\infty},\label{eq:derivative-j}\\
			\normo{\wh{\vpjk}}_{L_\xi^\infty} &\les \min \left(2^{\frac{3j}2} \normo{Q_{jk}\vp}_{L_x^2} , 2^{\frac j2 -k }2^{\frac{\de(j+k)}8}\normo{Q_{jk}\vp}_{H_\Omega^{0,1}}\right),\label{eq:linear-amplitude}
		\end{align}
where $\|v\|_{H_\Omega^{0, 1}} = (\|\Omega_{12}v\|_{L^2}^2 + \|\Omega_{23}v\|_{L^2}^2 + \|\Omega_{31}v\|_{L^2}^2)^\frac12$.
		%		\item[(ii)] We have
		%		\begin{align}
			%			\normo{\wh{\vpjk}}_{L_\xi^\infty} &\les \min \left(2^{\frac{3j}2} \normo{Q_{jk}\vp}_{L_x^2} , 2^{\frac j2 -k }2^{\frac{\de(j+k)}8}\normo{Q_{jk}\vp}_{H_\Omega^{0,1}}\right),\label{eq:linear-amplitude}
			%			\normo{\wh{\vpjk}(r\theta)}_{L_r^2(r^2dr)L_\theta^\infty} &\les 2^{j+k} \normo{Q_{jk}\vp }_{L_x^2},\\	\normo{\wh{\vpjk}(r\theta)}_{L_r^2(r^2dr)L_\theta^p} &\les  \normo{Q_{jk}\vp }_{H_\Omega^{0,1}},  \quad p \in [2,\infty],
			%		\end{align}
		%		and
		%		\begin{align*}
			%			\normo{\wh{\qjk \vp} - \wh{\vp_{j,k}}}_{L^\infty} \les 2^{\frac{3j}2} 2^{-4(j+k)} \|P_k f\|_{L^2}.
			%		\end{align*}
		\item[(ii)] We have
		\begin{align*}
			\sup_{j \in \cU_k} \|\qjk \vp\|_{H_\Omega^{0,1}} \le A \;\; \mbox{and}\;\; \sup_{j \in \cU_k} 2^{j+k}\|\qjk \vp\|_{H_\Omega^{0,1}} \le B,
		\end{align*}
		for some $k \in \Z$ and $A \le B \in [0,\infty)$, then we obtain
		\begin{align}\label{eq:interpolation-linfty}
			\normo{\wh{P_k \vp}}_{L^\infty} \les 2^{-\frac{3k}2} A^{\frac{1-\zeta}2}B^{\frac{1+\zeta}2}
		\end{align}
		for any $\zeta \in (0, 1)$.
	\end{itemize}
\end{lemma}

\begin{proof}
	See (i) of Lemma 3.4 of \cite{iopau2019}.
\end{proof}

\begin{lemma}[Time decay for Dirac part]
	Let $t \in \R$, $k \in \Z$, $\theta \in \{+,-\}$, and $\vp \in L^2$. Then, we have
	\begin{align}\label{eq:dispersive-dirac}
		\|e^{- \theta it\brad} \vp_{j,k}\|_{L^\infty} \les \min\left(2^{\frac{3k}2}, \bra{2^k}^3\bra{t}^{-\frac32}2^{\frac{3j}2}\right)\|\qjk \vp\|_{L^2}
	\end{align}
	for $j \in \cU_k$. Moreover, if $\bra{2^k}^2 \ll 2^{2k}\bra{t}$ and $j \in \cU_k$,  then we have
	\begin{align}
		\|e^{-\theta it\brad} \vp_{j,k}\|_{L^\infty} &\les \bra{2^k}^5 2^{\frac{j}2-k}\bra{t}^{-\frac32} (\bra{t}2^{2k})^\frac\de8 \normo{\qjk \vp}_{\hang} \;\;\mbox{for}\;\;2^j \ll 2^{k}\bra{t}\label{eq:dispersive-dirac-im1}
	\end{align}
 and also
	\begin{align}
		\|e^{- \theta it\brad} \vp_{\le j,k}\|_{L^\infty} &\les \bra{2^k}^{5}\bra{t}^{-\frac32}  \normo{\wh{Q_{\le jk} \vp}}_{L^\infty}\;\;\mbox{for}\;\; 2^j \les \bra{t}^\frac12.\label{eq:dispersive-dirac-im2}
	\end{align}
\end{lemma}

\begin{proof}
	See (iii) of Lemma 3.4 of \cite{iopau2019}.
\end{proof}

\begin{lemma}[Time decay for Maxwell part]
	Let $t \in \R$, $k \in \Z$, $\theta \in \{+,-\}$, and $\vp \in L^2$. Then, we have
	\begin{align}\label{eq:dispersive-maxwell}
		\|e^{- \theta it|D|} \vp_{j,k}\|_{L^\infty} \les 2^{\frac{3k}2}\min\left(1, 2^j\bra{t}^{-1}\right)\|\qjk \vp\|_{L^2},
	\end{align}
	for $j \in \cU_k$.	Moreover, if $|t|\ge1$ and $j \in \cU_k$, then we have
	\begin{align*}
		\normo{\rho_{[-100,100]}\left(\frac x{\bra{t}}\right) e^{- \theta it|D|} \vp_{j,k}}_{L^\infty} &\les \bra{t}^{-1}2^{\frac k2} (1+\bra{t}2^k)^\frac \de8 \|\qjk \vp\|_{\hang},%\label{eq:dispersive-maxwell-im1}
	\end{align*}
	\begin{align*}%\label{eq:dispersive-maxwell-im2}
		\|e^{- \theta it|D|} \vp_{j,k}\|_{L^\infty} &\les \bra{t}^{-1}2^{\frac k2}(1+\bra{t}2^k)^\frac \de8 \normo{\qjk \vp}_{\hang}\;\;\mbox{for}\;\; 2^j \ll \bra{t},
	\end{align*}
 and
	\begin{align*}
		\|e^{- \theta it|D|} \vp_{\le j,k}\|_{L^\infty} &\les \bra{t}^{-1}  2^{2k} \normo{\wh{Q_{\le jk} \vp}}_{L^\infty}\;\;\mbox{for}\;\; 2^j \les \bra{t}^\frac12 2^{-\frac k2}.%\label{eq:dispersive-maxwell-im3}
	\end{align*}
\end{lemma}

\begin{proof}
	See (ii) of Lemma 3.4 of \cite{iopau2019}.
\end{proof}

\subsection{Universal tools}

\begin{lemma}[Coifman-Meyer operator estimates]\label{lem:coif-mey}
	Let $1 \le p, q \le \infty$ satisfy that $\frac{1}{p}+\frac{1}{q}=\frac{1}{2}$. Assume that a multiplier $\textbf{m}$ satisfies
	\[
	\|\textbf{m}\|_{\rm CM}:=\left\Vert \iint_{\mathbb{R}^{3+3}}\mathbf{m}(\xi,\zeta)e^{ix\cdot\xi}e^{iy\cdot\eta}d\xi d\eta\right\Vert _{L_{x,y}^{1}}\le C_{\mathbf{m}} < \infty \quad(\zeta = \eta \;\,\mbox{or}\,\;\xi-\eta ).
	\]
	Then
	\begin{align*}
		\left\Vert \int_{\mathbb{R}^{3}}\mathbf{m}(\xi, \eta)\widehat{v}(\eta)\widehat{w}(\xi-\eta)d\eta\right\Vert _{L_{\xi}^{2}}\les C_{\mathbf{m}}\|v\|_{L^{p}}\|w\|_{L^{q}}.
	\end{align*}
\end{lemma}

%\begin{cor}
%	Let $k,k_1,k_2 \in \Z$ and $\bm k = (k,k_1,k_2)$. Then we have
%	\begin{align*}
	%		\| \phi(\xi,\eta) m \rho_{\bm k}(\xi,\eta)\|_{L^2} \les 2^{-\frac k2} \bra{2^k}^{4\max(\bm k)}
	%	\end{align*}
%\end{cor}

\begin{lemma}[Lemma 3.5 of \cite{iopau2019}]
	Let $\vp \in L^2$ and $k \in \Z$. Then, for
	\begin{align*}
		A_k:= \|P_k \vp\|_{L^2} + \sum_{l=1}^3 \normo{\rho_k(\xi) (\p_\xil \wh{\vp})(\xi)}_{L_\xi^2}, \qquad B_k:= \left[ \sum_{j \in \cU_k} 2^{2j} \|\qjk \vp\|_{L^2}^2 \right]^\frac12,
	\end{align*}
	we have
	\begin{align*}%\label{eq:esti-hardy1}
		A_k \les \sum_{2^k \sim 2^{k'}} B_{k'}
	\end{align*}
	and
	\begin{align}\label{eq:esti-hardy2}
		B_k \les \left\{ \begin{aligned}
			&\sum_{2^k \sim 2^{k'}} A_{k'} & \mbox{ for } k \ge 0, \\
			&	\sum_{k' \in \Z} 2^{-\frac{|k-k'|}2}\min(1, 2^{k'-k})A_{k'} & \mbox{ for } k < 0.
		\end{aligned}
		\right.
	\end{align}
	Especially, we have
	\begin{align}\label{eq:esti-hardy3}
		2^{-k}\|P_k\vp\|_{L^2} \les \sum_{k'\in \Z}2^{- \frac{|k-k'|}2} \min(1,2^{k'-k})A_{k'}.
	\end{align}
\end{lemma}

\subsection{Profile estimates}

We define the localized profiles as follows:
\begin{align*}
	&	\phi_{\theta,\cL}^{j,k}(t):=	\wt P_k\qjk \phi_{\theta,\cL}(t),\qquad\quad  \phi_{\theta ,\cL}^{\le J,k}(t):= \sum_{j \le J} \phi_{\theta,\cL}^{j,k}(t), \qquad\quad
	\phi_{\theta,\cL}^{> J,k}(t):= \sum_{j > J} \phi_{\theta,\cL}^{j,k}(t),\\
	&	 V_{\mu, \cL, \theta'}^{j,k}(t):=	\wt P_k\qjk V_{\mu, \cL, \theta'}(t),  \quad V_{\mu, \cL, \theta'}^{\le J,k}(t):= \sum_{j \le J} V_{\mu, \cL, \theta'}^{j,k}(t),\quad\;\; V_{\mu, \cL, \theta'}^{> J,k}(t):= \sum_{j > J} V_{\mu, \cL, \theta'}^{j,k}(t).
\end{align*}

\begin{lemma}
	Assume that $(\psi,A_\mu)$ is a solution to \eqref{maineq:md-lorenz} on $[0,T]$, for some $T>1$ and satisfies \eqref{assum:energy}--\eqref{assum:s-norm}. Let $t \in[0,T]$, $\cL \in \cV_n, n \in \{0, 1, 2, 3\}$, and $\theta, \theta' \in \{+,-\}$. Then we have
	\begin{align}\label{eq:high}
		\normo{P_k\psi_{\theta, \cL}(t)}_{L^2} + 	\|P_k W_{\mu,\cL, \theta'}(t)\|_{L^2} \les \ve_1 G_{n}(t,k),
	\end{align}
	where
	\begin{align*}
		G_{n}(t,k) := \bra{t}^{H(n)\de} \bra{2^k}^{-N(n)}.
	\end{align*}
	Moreover, if $n \le 2, k \in \Z$ and $l \in \{1,2,3\}$, then
	\begin{align}\label{eq:weight}
		\bra{2^k} \normo{\rho_k(\xi) \left(\p_\xil \wh{\phi_{\theta,\cL}}\right)(t,\xi) }_{L_\xi^2} + 2^{\frac k2}\normo{\rho_k(\xi) \left(\p_\xil \wh{V_{\mu,\cL, \theta'}}\right)(t,\xi) }_{L_\xi^2} \les \ve_1 G_{n+1}(t,k),
	\end{align}
	
	As a consequence, if $n \le 2$ and $j \in \cU_k$, then
	\begin{align}\label{eq:elliptic-q}
		2^j\bra{2^k} \normo{\qjk \phi_{\theta, \cL} (t) }_{L^2} + 2^j2^{k} \normo{\qjk V_{\mu,\cL,\theta'} (t) }_{L^2} \les \ve_1 G_{n+1}(t,k)
	\end{align}
	and
	\begin{align}\label{eq:esti-l2-k}
		\bra{2^k} \normo{P_k \phi_{\theta,\cL} (t) }_{L^2} + 2^{k} \normo{P_k V_{\mu,\cL, \theta'} (t) }_{L^2} \les \ve_1 2^{k}G_{n+1}(t,k).
	\end{align}
\end{lemma}
\begin{proof}
	By the a priori assumptions \eqref{assum:energy} and \eqref{assum:weight}, we directly obtain \eqref{eq:high} and \eqref{eq:weight}, respectively. From \eqref{assum:s-norm} and \eqref{eq:esti-hardy2}, we also have \eqref{eq:elliptic-q}. By \eqref{eq:elliptic-q} with the summation over $j$,  the bound \eqref{eq:esti-l2-k} holds.
\end{proof}

\begin{lemma}
	Assume that $(\psi,A_\mu)$ is a solution to \eqref{maineq:md-lorenz} on $[0,T]$ for some $T>1$ and satisfies \eqref{assum:energy}--\eqref{assum:s-norm}. Let $t \in [0,T]$, $\cL \in \cV_n, n \in \{0, 1, 2\} $. Then we have
	\begin{align}\label{eq:decay-dirac}
		\sum_{j \in \cU_k} \normo{e^{-\theta it\brad} \phi^{j,k}_{\theta,\cL}(t)}_{L^\infty} \les \ve_1 G_{n+1}(t, k)\bra{2^k}^{2} 2^{\frac{k}2} \min (\bra{t}^{-1},2^{2k}),
	\end{align}
	and
	\begin{align}\label{eq:decay-maxwell}
		\sum_{j \in \cU_k} \normo{e^{\theta' it|D|} V^{j,k}_{\mu,\cL, \theta'}(t)}_{L^\infty} \les \ve_1 \bra{t}^{ \frac\de2}G_{n+1}(t, k)  \bra{2^k}^{2}  2^{\frac k2} \min (\bra{t}^{-1},2^{k}),
	\end{align}
	for  $k \in \Z$. Moreover, if  $n \le 1$ and $\bra{t}^{-1 }\ll 2^{2k}$, then
	\begin{align}\label{eq:decay-dirac-2}
		\sum_{2^j \in [2^{-k},2^{k}\bra{t}]} \normo{e^{-\theta it\brad} \phi^{j,k}_{\theta,\cL}(t)}_{L^\infty} \les \ve_1 \bra{t}^{- \frac32 + \frac\de2}G_{n+2}(t, k)\bra{2^k}^{6}   2^{-\frac{k}2}.
	\end{align}
\end{lemma}

\begin{proof}
	By \eqref{eq:elliptic-q} and \eqref{eq:dispersive-dirac}, we have \eqref{eq:decay-dirac}. Similarly, the bound \eqref{eq:decay-maxwell} follows from \eqref{eq:dispersive-maxwell} and \eqref{eq:elliptic-q}. Using \eqref{eq:elliptic-q}, we see that, for $|b| \le 1$,
	\begin{align}\label{eq:qjk-omega}
		\normo{\qjk \Om^b \phi_{\theta,\cL}(t)}_{L^2} \les \ve_1 G_{n+|b|+1}(t, k)\bra{2^k}^{-1}2^{-j}.
	\end{align}
	This together with \eqref{eq:dispersive-dirac-im1} implies  \eqref{eq:decay-dirac-2}.
\end{proof}

	\begin{lemma}
		Assume that $(\psi,A_\mu)$ is a solution to \eqref{maineq:md-lorenz} on $[0,T]$ for some $T>1$ and satisfies \eqref{assum:energy}--\eqref{assum:s-norm}. Let $t \in [0,T]$, $\cL \in \cV_n, n \in \{0, 1, 2\}$. Then we have
		\begin{align}\label{eq:decay-cor}
			\sum_{j \in \cU_k} \normo{e^{-\theta it\brad} \phi^{j,k}_{\theta,\cL}(t)}_{L^\frac1\zeta} \les \ve_1 G_{n+1}(t, k)\bra{2^k}^{2}  \min \left(2^{(\frac12 -3\zeta){k}}\bra{t}^{-1}, 2^{\frac{5k}2} \right),
		\end{align}
		for $0 < \zeta \ll 1$.
	\end{lemma}
	\begin{proof}
		By \eqref{eq:dispersive-dirac}, we have
		\begin{align*}
			\|e^{-\theta it \brad}\phi_{\theta,\cL}^{j,k}(t)\|_{L^\frac1\zeta} \les \min\left(2^{\frac{3k}2 - 3\zeta k}, 2^{\frac{3j}2 - 3\zeta j}\bra{2^k}^{3-6\zeta} \bra{t}^{-\frac32 + 3\zeta}\right) \|\qjk \phi_{\theta,\cL}(t)\|_{L^2}.
		\end{align*}	
		Then, \eqref{eq:elliptic-q} implies \eqref{eq:decay-cor} directly.
	\end{proof}

	\begin{lemma}[Vector fields free estimates]
		Assume that $(\psi,A_\mu)$ is a solution to \eqref{maineq:md-lorenz} on $[0,T]$, for some $T>1$ and satisfies \eqref{assum:energy}--\eqref{assum:s-norm}. Let $t \in [0,T]$ and $k \in \Z$. Then we have
		\begin{align}\label{eq:zero-vec-1}
			\begin{aligned}
				\sum_{j \in \cU_k}  2^{j} \left\|Q_{jk}A_\mu(t) \right\|_{L^2} &\les \ve_1 2^{- (1+ 5H(2)\de)k} \bra{2^k}^{-25},\\
				\normo{\wh{\pk \psi_{\theta}}(t)}_{L_\xi^\infty} &\les \ve_1 2^{-(\frac12-\frac1{100})k} \bra{2^k}^{-25 },\\
				\normo{\pk \psi_{\theta}(t)}_{L^2} &\les \ve_1 2^{(1+ \frac1{100}) k} \bra{2^k}^{-N(0) +2 }.
			\end{aligned}
		\end{align}
		We also get
		\begin{align}\label{eq:zero-vec-3}
			\sum_{j\in\cU_k} \normo{e^{\theta it|D|}V^{j,k}_{\mu,\theta}(t)}_{L^\infty} \les \ve_1 \min\left(\bra{t}^{-1},2^k\right)2^{\left(\frac12 -5H(2)\de\right)k}\bra{2^k}^{-22}.
		\end{align}
		Moreover, for $\bra{t} \gg 2^{-2k}$ and $2^J \in [2^{-k}, C 2^{k}\bra{t}]$ for some $C \ll 1$,  we have		
		\begin{align}\label{eq:zero-vec-2}
			\normo{e^{-\theta it \bra{D}} \phi_{\theta}^{\le J,k}(t)}_{L^\infty} \les \ve_1 \bra{t}^{-\frac32} 2^{-(\frac12 - \frac{1}{1000})k}\bra{2^k}^{-19}.
		\end{align}	
	\end{lemma}
	
	\begin{proof}
		The estimates in \eqref{eq:zero-vec-1} follow from the definition of $\|\cdot\|_{\bm D}$ and $\|\cdot\|_{\bm M}$. Then \eqref{eq:dispersive-maxwell} and \eqref{eq:zero-vec-1} imply \eqref{eq:zero-vec-3}. To prove \eqref{eq:zero-vec-2}, we consider the case $2^k \ge \bra{t}^{\frac1{30}}$. The bound \eqref{eq:high} yields \eqref{eq:zero-vec-2} for $2^k \ge \bra{t}^{\frac1{30}}$. If $2^k \le \bra{t}^{-\frac12 + \frac1{1000}}$ or $2^J \le \bra{t}^{\frac12}$, it also follows from  \eqref{eq:zero-vec-1} and \eqref{eq:dispersive-dirac-im2}. The remaining case is that  $2^k \in [\bra{t}^{-\frac12 +\frac1{1000}},\bra{t}^{\frac1{30}}]$ and $2^J \ge \bra{t}^{\frac12}$. By \eqref{eq:dispersive-dirac-im1} and \eqref{eq:qjk-omega}, we estimate
		\begin{align*}
			\sum_{2^j \ge \bra{t}^{\frac12}}	\normo{e^{-\theta it \bra{D}} \phi_{\theta}^{j,k}(t)}_{L^\infty} &\les \sum_{2^j \ge \bra{t}^{\frac12}} \ve_1\bra{t}^{-\frac32}  2^{-\frac j2 -k} \bra{t}^{11\de}\bra{2^k}^{-26}\\
			&\les \ve_1 \bra{t}^{-\frac32} 2^{-(\frac12 - \frac{1}{1000})k}\bra{2^k}^{-19}.
		\end{align*}
		This finishes the proof of \eqref{eq:zero-vec-2}.
	\end{proof}

The following lemma will be used in Section \ref{sec:s-dirac}.
\begin{lemma}\label{lem:esti-wei-time} Assume that $(\psi,A_\mu)$ is a solution to \eqref{maineq:md-lorenz} on $[0,T]$, for some $T>1$ and satisfies \eqref{assum:energy}--\eqref{assum:s-norm}. Let  $k\in \Z$, $t\in [0,T]$, and $l \in \{ 1,2,3\}$. If $2^{k}\ge \bra{t}^{-\frac23-10H(2)\de m}$, we have
	\begin{align*}
		\left\| \p_t P_k (x_l  V_{\mu,\theta})(t) \right\|_{L^2} \les \ve_1^2 \bra{t}^{-\frac18}.
	\end{align*}
\end{lemma}

\begin{proof}
	Taking Fourier transform  and then time derivative, one gets
	\begin{align}\label{eq:aim-wei-time}
		\p_t \wh{P_k x_l  V_{\mu,\theta}}(t) =	\rho_k(\xi) \p_{\xi_l} \left[|\xi|^{-\frac12}\int_{\R^3} e^{itq_\Theta(\xi,\eta)}\bra{\wh{\phi_\theo}(t,\eta),\al_\mu \wh{\phi_\thet}(t,\xi+\eta)}d\eta\right].
	\end{align}
	If the derivative $\p_{\xi_l}$ falls on $\wh{\phi_\thej}$, the bound of \eqref{eq:aim-wei-time} follows from \eqref{eq:weight}. Thus, we consider the cases that the derivative falls on $|\xi|^{-\frac12}$ or the phase modulation $e^{itq_\Theta(\xi,\eta)}$. Let us first handle the case that the derivative falls on the phase modulation:
	\begin{align*}
		t \rho_k(\xi) |\xi|^{-\frac12} \int_{\R^3} e^{itq_\Theta(\xi,\eta)} \nabla_{\xi}q_\Theta(\xi,\eta)\bra{\wh{\phi_\theo}(t,\eta),\al_\mu \wh{\phi_\thet}(t,\xi+\eta)}d\eta.
	\end{align*}
	In view of \eqref{eq:resonance-time-maxwell} and \eqref{eq:resonance-space-maxwell}, it suffices to consider only the case $\theo =\thet$. By integration by parts in $\eta$, we get the following terms:
	\begin{align}
		& \rho_k(\xi) |\xi|^{-\frac12} \int_{\R^3} e^{itq_\Theta(\xi,\eta)} \nabla_\eta \frac{\nabla_{\xi}q_\Theta(\xi,\eta)\nabla_{\eta}q_\Theta(\xi,\eta)}{|\nabla_{\eta}q_\Theta(\xi,\eta)|^2}\bra{\wh{\phi_\theo}(t,\eta),\al_\mu \wh{\phi_\thet}(t,\xi+\eta)}d\eta,\label{eq:esti-wei-time-1}\\
		& \rho_k(\xi) |\xi|^{-\frac12} \int_{\R^3} e^{itq_\Theta(\xi,\eta)} \frac{\nabla_{\xi}q_\Theta(\xi,\eta)\nabla_{\eta}q_\Theta(\xi,\eta)}{|\nabla_{\eta}q_\Theta(\xi,\eta)|^2}\bra{\wh{x \phi_\theo}(t,\eta),\al_\mu \wh{\phi_\thet}(t,\xi+\eta)}d\eta,\label{eq:esti-wei-time-2}
	\end{align}
	and additional symmetric term. To estimate \eqref{eq:esti-wei-time-1}, we make the dyadic decomposition for $|\eta|$ and $|\xi+\eta|$ into $2^{k_1}$ and $2^{k_2}$, respectively. If $2^{\min(k_1,k_2)} \le \bra{t}^{-\frac12}$, then by \eqref{eq:esti-l2-k} we have
	\begin{align*}
		\sum_{2^{\min(k_1,k_2)} \le \bra{t}^{-\frac12}} \left\|\eqref{eq:esti-wei-time-1}\right\|_{L^2} &\les \sum_{2^{\min(k_1,k_2)} \le \bra{t}^{-\frac12}} 2^{-\frac {3k}2} 2^{\frac{-2k_2 + 3\min(k_1,k_2)}2} \normo{\pko \phi_\theo(t)}_{L^2}\normo{\pk_2\phi_\thet(t)}_{L^2} \\
		&\les \sum_{2^{\min(k_1,k_2)} \le \bra{t}^{-\frac12}} \ve_1^2 2^{-\frac {3k}2 }\bra{t}^{2H(1)\de}2^{\frac{3\min(k_1,k_2)}2+k_1}\bra{2^{k_1}}^{-N(1)}\bra{2^{k_2}}^{-N(1)}\\
		&\les \ve_1^2  \bra{t}^{-\frac14}.
	\end{align*}
	On the other hand, we handle the case $2^{\min(k_1,k_2)} \le \bra{t}^{-\frac12}$. We decompose the profiles into $P_{k_\ell}\phi_\thej=\phi_{\theta_\ell}^{\le J, k_\ell}+\phi_{\theta_\ell}^{> J, k_\ell}$ for $\ell=1,2$ with $2^J = C\bra{t}2^{k_\ell}$ for some $C\ll 1$, and  we denote $\phi_{\theta_\ell}^{1} = \phi_{\theta_\ell}^{\le J,k_\ell}$ and $\phi_{\theta_\ell}^{2} = \phi_{\theta_\ell}^{> J,k_\ell}$ for $a,b =1,2$. Using \eqref{eq:elliptic-q}, for $(a,b) \neq (1,1)$, we see that
	\begin{align*}
		\sum_{2^{\min(k_1,k_2)} \ge \bra{t}^{-\frac12}} \left\|\eqref{eq:esti-wei-time-1}\right\|_{L^2} &\les \sum_{2^{\min(k_1,k_2)} \ge \bra{t}^{-\frac12}} \normo{ \phi_\theo^{a}(t)}_{L^2}\normo{\phi_\thet^{b}(t)}_{L^2} \les \ve_1^2  \bra{t}^{-\frac12+2H(2)\de}.
	\end{align*}
	If $(a,b)= (1,1)$, by Lemma \ref{lem:coif-mey}, \eqref{eq:elliptic-q}, and \eqref{eq:zero-vec-2}, we then estimate
	\begin{align*}
		\sum_{2^{\min(k_1,k_2)} \ge \bra{t}^{-\frac12}} \left\|\eqref{eq:esti-wei-time-1}\right\|_{L^2} \les 	\sum_{2^{\min(k_1,k_2)} \ge \bra{t}^{-\frac12}} 2^{-\frac {3k}2} \normo{\phi_\theo^{\le J, k_1}(t)}_{L^2}\normo{\phi_\thet^{\le J, k_2}(t)}_{L^\infty}\les \ve_1 \bra{t}^{-\frac18 }.
	\end{align*}
	The estimates for \eqref{eq:esti-wei-time-2} can be treated similarly.
	
	When the derivative falls on $|\xi|^{-\frac12}$ in \eqref{eq:aim-wei-time},  an additional factor of $|\xi|^{-1}$ appears. Though we are given this singular factor, since it corresponds to the growth $\brat^{\frac23 + 10H(2)\de}$ from the assumption of this lemma, we can handle this case by using the similar methods to those detailed above.
\end{proof}

\newcommand{\fN}{{\mathfrak{N}}}

\section{Nonlinear estimates}
In this section we consider nonlinear estimates for Dirac part and Maxwell part based on the a priori assumptions \eqref{assum:energy}--\eqref{assum:s-norm} and linear estimates of Section \ref{sec:linear}.

\subsection{Nonlinear estimates for Dirac part}
Let us invoke the nonlinearity of Dirac part \eqref{eq:commu-all} and define the nonlinearity by
\begin{align*}
	\fN_{\theta, \cL}^{\bm D} := &\Pi_\theta \cL (A_\mu \al^\mu \psi) + \sum_{\cL'\in \cV_n} \cR_{\cL' }\cL'(A_\mu \al^\mu \psi).
\end{align*}
Then
\begin{align}\label{eq:dirac-pt}
	(-i\p_t + \theta \brad) \psi_{\theta, \cL} = \fN_{\theta, \cL}^{\bm D}, \quad -i\p_t\phi_{\theta, \cL} = e^{\theta it\brad}\fN_{\theta, \cL}^{\bm D}.
\end{align}

We will first prove the $L^2$ boundedness of $\fN_{\theta, \cL}^{\bm D}$.
\begin{lemma} \label{lem:esti-nonlinear-dirac}
	Assume that $(\psi,A_\mu)$ satisfies the a priori assumptions \eqref{assum:energy}--\eqref{assum:s-norm} on $[0,T]$, for some $T>1$. Let $t\in[0,T]$, $k \in \Z$, $\cL \in \cV_n, n \in \{0, 1, 2, 3\}$, and $\theta \in \{+, -\}$. Then we have
	\begin{align}\label{eq:esti-non-dirac-time}
		\normo{P_k \mathfrak N_{\theta, \cL}^{\bm D}(t)}_{L^2} + \normo{P_k \p_t \phi_{\theta,\cL}(t)}_{L^2} \les \ve_1^2  \bra{t}^{\wt{H}(n) \de}\bra{2^k}^{-N(n+1)- 5} \min(\bra{t}^{-1}, 2^{k}),
	\end{align}
	where
	\begin{align*}
		N(4)= 0, \;\; \wt{H}(0):=5,\;\; \wt{H}(n):= H(n) +160.
	\end{align*}
\end{lemma}
\begin{proof}
	%	\begin{align}
		%		\begin{aligned}
			%			(-i\p_t + \theta \brad)\cL \psi_\theta &= \Pi_\theta \cL (A_\mu \al^\mu \psi) + \theta \left( \p_0^{|c|} \widetilde{\cL}_1  + \widetilde{\cL}_2\right)  (A_\mu \al^\mu \psi)  \\
			%			& \hspace{0.5cm}  -\theta [\cG]^{c}\p_0^{|c|} \Pi_\theta(A_\mu \al^\mu \psi) + \theta  \sum_{j=1}^3c_j[\p]^a [\cO]^b R_j^{c_j}(A_\mu \al^\mu \psi)
			%		\end{aligned}
		%	\end{align}
	By the boundedness of $\Pi_\theta$, $R_{\cL'}$, and $e^{\theta it\brad }$ it suffices to show the following bilinear estimate
	\begin{align}\label{eq:aim-nonlinear-dirac}
	\begin{aligned}
			&\sum_{k_1,k_2\in \Z} 2^{-\frac{k_2}{2}}\left \|P_k\left( P_{k_2} W_{\mu, \cL_2, \thet}(t) \al^\mu P_{k_1}\psi_{\theo, \cL_1}(t)\right)\right\|_{L^2} \\
			& \hspace{4cm} \les \ve_1^2\bra{t}^{\wt{H}(n) \de}\bra{2^k}^{-N(n+1)- 5}  \min(\bra{t}^{-1}, 2^{k}),
	\end{aligned}
	\end{align}
	for $\cL_l \in \cV_l, \, \theta_l \in \{+,-\} \,(l=1,2)$,  $n_1 + n_2 \le n$, and $k \in \Z$.
	By H\"older's inequality and \eqref{eq:high}, we have
	\begin{align*}
		\left \|P_k\left( P_{k_2} W_{\mu, \cL_2, \thet}(t) \al^\mu P_{k_1}\psi_{\theo,\cL_1}(t)\right)\right\|_{L^2} &\les  2^{\frac{3\min(\bm k)-k_2}2}\|\psi_{\theo,\cL_1}(t)\|_{L^2} \|P_{k_2} W_{\mu, \cL_2, \thet}(t)\|_{L^2} \\
		&\les \ve_1^2 \bra{t}^{\left[ H(n_1)+H(n_2)\right] \de} 2^{\frac{3\min(\bm k)-k_2}2} \bra{2^{k_1}}^{-N(n_1)}\bra{2^{k_2}}^{-N(n_2)},
	\end{align*}
	where $\bm k  := (k,k_1,k_2)$.	This yields that
	\begin{align*}
		\sum_{k_1,k_2 \in \Z} 2^{-\frac{k_2}{2}}\left \|P_k\left( P_{k_2} W_{\mu, \cL_2, \thet} (t)\al^\mu P_{k_1}\psi_{\theo,\cL_1}(t)\right)\right\|_{L^2} \les \ve_1^2 \bra{t}^{\left[ H(n_1)+H(n_2)\right] \de} 2^{k},
	\end{align*}
	which implies \eqref{eq:esti-non-dirac-time} in the case that $2^k \le \bra{t}^{-1}$.

	Let us consider the case $2^k \ge \bra{t}^{-1}$. From \eqref{eq:decay-dirac} and \eqref{eq:decay-maxwell}, we have
	\begin{align}
		\begin{aligned}\label{eq:l-infinity}
			\|\psi_{\theo,\cL_1}(t)\|_{L^\infty} &\les \ve_1 \bra{t}^{H(n_1+1)\de}2^{\frac{k_1}2}\bra{2^{k_1}}^{-N(n_1+1)+2} \min(\bra{t}^{-1},2^{2k_1}),\\
			\|W_{\mu, \cL_2, \thet}(t)\|_{L^\infty} &\les \ve_1 \bra{t}^{H(n_2+1)\de+\frac \de2}2^{\frac{k_2}2}\bra{2^{k_2}}^{-N(n_2+1)+2} \min(\bra{t}^{-1},2^{k_2}),
		\end{aligned}
	\end{align}
	for $0 \le n_1,n_2 \le 2$. We partition the set of numbers of vector fields as follows:	
	\begin{align*}
		\cN_n &= \{ (n_1,n_2) : n_1+ n_2 = n, 0\le n_1,n_2 \le n\} \\
		&= \{ (n_1,n_2) : n_1+ n_2 = n,   n_1,n_2 \ge 1\} \cup \{(n,0), (0,n): n\ge 1\} \cup \{(0,0)\} \\
		&=:\cN_n^1 \cup \cN_n^2 \cup \{(0,0)\}.
	\end{align*}
	
	\emph{ Estimates for $(n_1,n_2) \in \cN_{n}^1$ and $2^k \ge \bra{t}^{-1}$.} When $2^{k} \les 2^{k_2}$, we estimate
	\begin{align*}
		\sum_{2^{k} \les 2^{k_2}} &2^{-\frac{k_2}{2}}\left \|P_k\left( P_{k_2} W_{\mu,\cL_2, \thet}(t) \al^\mu P_{k_1}\psi_{\theo,\cL_1}(t)\right)\right\|_{L^2}\\ &\les \sum_{2^{k} \les 2^{k_2}} 2^{-\frac{k_2}{2}} \|P_{k_1}\psi_{\theo,\cL_1}(t)\|_{L^\infty}\|P_{k_2} W_{\mu,\cL_2, \thet}(t)\|_{L^2}\\
		&\les \ve_1^2 \bra{t}^{\left[H(n_1+1) + H(n_2)\right]\de -1 } \bra{2^k}^{-N(n)}.
	\end{align*}
	Note that $H(n_1+1) + H(n_2) \le H(n) +10$ for $(n_1,n_2) \in \cN_n^1$. Analogously, if $2^{k_2} \ll 2^k$, we obtain
	\begin{align*}
		\sum_{2^{k_2}\ll 2^k} &2^{-\frac{k_2}{2}}\left \|P_k\left( P_{k_2} W_{\mu, \cL_2, \thet}(t) \al^\mu P_{k_1}\psi_{\theo,\cL_1}(t)\right)\right\|_{L^2}\\ &\les \sum_{2^{k_2}\ll 2^k} 2^{-\frac{k_2}{2}} \|P_{k_1}\psi_{\theo,\cL_1}(t)\|_{L^2}\|P_{k_2} W_{\mu, \cL_2, \thet}(t)\|_{L^\infty}\\
		&\les \ve_1^2 \bra{t}^{\left[H(n) +11\right]\de -1 } \bra{2^k}^{-N(n_1)}.
	\end{align*}
	
	%We now prove \eqref{eq:aim-nonlinear-dirac} when $(n_1,n_2) \in N_n^2$ and $2^k \ge \bra{t}^{-1}$.
	
	\emph{ Estimates for $(n_1,n_2) \in \cN_n^2$ and $2^k \ge \bra{t}^{-1}$.} If $(n_1,n_2) = (n, 0)$, then we consider the cases $2^{k_1} \le \bra{t}^{-1}$ and $2^{k_1} \ge \bra{t}^{-1}$, for which we obtain
	\begin{align*}
		\sum_{2^{k_1} \le \bra{t}^{-1}}& 2^{-\frac{k_2}{2}}\left \|P_k\left( P_{k_2} W_{\mu,\thet} (t)\al^\mu P_{k_1}\psi_{\theo,\cL_1}(t)\right)\right\|_{L^2}\\
		 &\les \sum_{2^{k_1} \le \bra{t}^{-1}} 2^{\frac{3\min(\bm k)- k_2}{2}} \|P_{k_1}\psi_{\theo, \cL_1}(t)\|_{L^2}\|P_{k_2} W_{\mu, \thet}(t)\|_{L^2}\\
		&\les \ve_1^2 \bra{t}^{\left[H(n) +2\right]\de -1 } \bra{2^k}^{-N(n)}
	\end{align*}
	and	
	\begin{align*}
		\sum_{2^{k_1} \ge \bra{t}^{-1}} &2^{-\frac{k_2}{2}}\left \|P_k\left( P_{k_2} W_{\mu, \thet}(t) \al^\mu P_{k_1}\psi_{\theo, \cL_1}(t)\right)\right\|_{L^2} \\
		&\les \sum_{2^{k_1} \ge \bra{t}^{-1}} 2^{-\frac{k_2}{2}} \|P_{k_1}\psi_{\theo, \cL_1}(t)\|_{L^2}\|P_{k_2} W_{\mu, \thet}(t)\|_{L^\infty}\\
		&\les \ve_1^2 \bra{t}^{\left[H(n) +12\right]\de -1 } \bra{2^k}^{-N(n)}.
	\end{align*}
	These finish the proof of \eqref{eq:aim-nonlinear-dirac} when $(n_1,n_2) = (n,0)$.
	
	If $(n_1,n_2) = (0,n)$, we only consider the case $2^{k_2} \ll 2^k \sim 2^{k_1}$ since apart from this case, we can obtain the estimates similar to the above. By \eqref{eq:esti-l2-k}, we estimate
	\begin{align*}
		2^{-\frac{k_2}2}\left \|P_k\left( P_{k_2} W_{\mu, \cL_2, \thet}(t) \al^\mu \pko\psi_{\theo}(t)\right)\right\|_{L^2} &\les 2^{k_2} \|\pko\phi_{\theo}(t)\|_{L^2}\|P_{k_2} W_{\mu, \cL_2, \thet}(t)\|_{L^2} \\
		&\les \ve_1^2 \bra{t}^{[H(1)+H(n)]\de} 2^{k_1+k_2} \bra{2^{k_1}}^{-N(1)-1}\bra{2^{k_2}}^{-N(n)}.
	\end{align*}
	This enables us to get \eqref{eq:aim-nonlinear-dirac} if $2^{k_1 + k_2} \le \bra{t}^{-1}$.
	
	To handle the case $2^{k_1 + k_2} \ge \bra{t}^{-1}$, we make the decomposition $P_{k_1}\phi_{\theo} = \phi_{\theo}^{\le J, k_1} + \phi_{\theo}^{>J,k_1}$ with $2^J= C2^{k_1}\bra{t}$ for some $C \ll 1$. In view of \eqref{eq:zero-vec-2} and \eqref{eq:elliptic-q}, we see that
	\begin{align}\label{eq:decay-decomposition}
\begin{aligned}
			\|e^{-\theo it\brad}\phi_{\theo}^{\le J,k_1}(t)\|_{L^\infty} &\les \ve_1 \bra{t}^{-\frac32} 2^{-\frac{k_1}2 + \frac{k_1}{1000}} \bra{2^{k_1}}^{-19},\\
		\|\phi_{\theo}^{> J,k_1}(t)\|_{L^2} &\les \ve_1 \bra{t}^{-1+ H(1)\de} 2^{-k_1} \bra{2^{k_1}}^{-N(1)}.
\end{aligned}
	\end{align}
	Then we have
	\begin{align*}
		2^{-\frac{k_2}2}&\left \|P_k\left( P_{k_2} W_{\mu, \cL_2, \thet} (t) \al^\mu e^{-\theo it\brad}\phi_{\theo}^{\le J,k_1} (t)\right)\right\|_{L^2} \\
		&\les  2^{-\frac{k_2}2}\|e^{- \theo it\brad}\phi_{\theo}^{\le J,k_1}(t)\|_{L^\infty}\|P_{k_2} W_{\mu, \cL_2, \thet}(t)\|_{L^2} \\
		&\les \ve_1^2 \bra{t}^{-\frac32+ H(n)\de} 2^{-\frac{k_1+k_2}2 + \frac{k_1}{2000}}\bra{2^{k_1}}^{-19} \bra{2^{k_2}}^{-N(n)}
	\end{align*}
	and
	\begin{align*}
		2^{-\frac{k_2}2}&\left \|P_k\left( P_{k_2} W_{\mu,\thet,\cL_2} (t)\al^\mu e^{-\theo it\brad}\phi_{\theo}^{> J,k_1}(t)\right)\right\|_{L^2}\\
		 &\les  2^{k_2}\|\phi_{\theo}^{> J,k_1}(t)\|_{L^2}\|P_{k_2} W_{\mu,\thet,\cL_2}(t)\|_{L^2} \\
		&\les \ve_1^2 \bra{t}^{-1+\left[H(n)+H(1)\right]\de} 2^{k_2 -k_1}\bra{2^{k_1}}^{-N(1)} \bra{2^{k_2}}^{-N(n)}.
	\end{align*}
	These give us the desired estimates in case  $2^{k_2} \ll 2^{k} \sim 2^{k_1}$ and $2^{k_1 + k_2} \ge \bra{t}^{-1}$.
	
	%It remains to consider the case $(n_1,n_2) = (0,0)$.
	
	\emph{Estimates for $(n_1,n_2) = (0,0)$.} By H\"older's inequality, we estimate
	\begin{align*}
		\left \|P_k\left( P_{k_2} W_{\mu,\thet}(t) \al^\mu P_{k_1}\psi_{\theo}\right) (t) \right\|_{L^2} &\les 2^{\frac{3\min(\bm k)}{2}} \|P_{k_1}\psi_{\theo}(t)\|_{L^2}\|P_{k_2} W_{\mu,\thet}(t)\|_{L^2} \\
		&\les \ve_1^2\bra{t}^{2\de}2^{\frac{3\min(\bm k)}{2}} \bra{2^{k_1}}^{-N(0)}\bra{2^{k_2}}^{-N(0)},
	\end{align*}
	which shows \eqref{eq:aim-nonlinear-dirac} if $2^{\min(\bm k)} \le \bra{t}^{-1}$ or $\bra{2^{\max(\bm k)}} \ge \bra{t}^{\frac1{30}}$.
	
	Let us now consider
	$$
	\mbox{Case A: } 2^{\min(\bm k)} \ge \bra{t}^{-1} \mbox{ and } \bra{2^{\max(\bm k)}} \le \bra{t}^{\frac1{30}}.
	$$
	By Plancherel's theorem, we have
	\begin{align*}
		\|P_k\left( P_{k_2} W_{\mu, \thet} \al^\mu P_{k_1}\psi_{\theo }\right)\|_{L^2} = \|\cF \left[P_k\left( P_{k_2} W_{\mu, \thet} \al^\mu P_{k_1}\psi_{\theo}\right)\right]\|_{L_\xi^2}.
	\end{align*}
	In particular, one gets
	\begin{align*}
		\cF \left[P_k\left( P_{k_2} W_{\mu, \thet} (t)\al^\mu P_{k_1}\psi_{\theo}(t)\right)\right](\xi) = \int_{\R^3}  \rho_k(\xi)  e^{itp_{12}(\xi,\eta)} \wh{P_{k_2} V_{\mu,\thet}}(t,\eta) \al^\mu  \wh{P_{k_1}\phi_{\theo}}(t,\xi-\eta) d\eta,
	\end{align*}
	where the phase
	\begin{align*}
		p_{12}(\xi,\eta)= \theo \bra{\xi-\eta} + \thet |\eta|.
	\end{align*}	
	As we observed in \eqref{eq:nonresonance-space}, this phase interaction does not exhibit a space resonance. We exploit this non-resonance feature as follows: By the relation
	\begin{align*}
		e^{itp_{12}(\xi,\eta) } = -it^{-1} \frac{\nabla_\eta p_{12}(\xi,\eta) \cdot \nabla_\eta e^{itp_{12}(\xi,\eta)}}{|\nabla_\eta p_{12}(\xi,\eta)|^2},
	\end{align*}
	we integrate by parts with respect to $\eta$ to obtain the following integrals
	\begin{align}
		& t^{-1}\int_{\R^3}   e^{itp_{12}(\xi,\eta)} \nabla_\eta M_{12,\bm k}(\xi,\eta) \wh{P_{k_2} V_{\mu,\thet}}(t,\eta) \al^\mu  \wh{P_{k_1}\phi_{\theo}}(t,\xi-\eta) d\eta,\label{eq:nonlinear-space-1}\\
		& t^{-1}\int_{\R^3}    e^{itp_{12}(\xi,\eta)} M_{12,\bm k}(\xi,\eta) \wh{P_{k_2} (xV_{\mu,\thet})}(t,\eta) \al^\mu  \wh{P_{k_1}\phi_{\theo}}(t,\xi-\eta) d\eta,\label{eq:nonlinear-space-2}\\
		& t^{-1}\int_{\R^3}    e^{itp_{12}(\xi,\eta)} M_{12,\bm k}(\xi,\eta) \wh{P_{k_2} V_{\mu,\thet}}(t,\eta) \al^\mu  \wh{P_{k_1}x\phi_{\theo}}(t,\xi-\eta) d\eta,\label{eq:nonlinear-space-3}
	\end{align}
	where
	\begin{align*}
		M_{12,\bm k}(\xi,\eta) = \frac{\nabla_\eta p_{12}(\xi,\eta) }{|\nabla_\eta p_{12}(\xi,\eta)|^2} \rho_{k}(\xi) \rho_{k_1}(\xi-\eta) \rho_{k_2}(\eta).
	\end{align*}
	Then, a direct calculation yields that
	\begin{align*}
		|\nabla_\eta M_{12,\bm k}| \les 2^{-k_2}\bra{2^{k_1}}^{4}.
	\end{align*}
	Using H\"older's inequality and \eqref{eq:high}, we see that
	\begin{align*}
		\sum_{\rm Case \, A}2^{-\frac{k_2}2}\|\eqref{eq:nonlinear-space-1}\|_{L_\xi^2} \les 	\sum_{\rm Case \, A}\bra{t}^{-1+2\de}2^{\frac{3\min(\bm k)-3k_2}{2}} \bra{2^{k_1}}^{-N(0)+4} \bra{2^{k_2}}^{-N(0)} \les \bra{t}^{-1+3\de} \bra{2^k}^{-N(1)-5}.
	\end{align*}
	This leads us \eqref{eq:aim-nonlinear-dirac}. For \eqref{eq:nonlinear-space-2}, we have
	\begin{align*}
		\|M_{12,\bm k}\|_{\cm} \les \bra{2^{k_2}}^{8}.
	\end{align*}
	Then, by H\"older's inequality and  Lemma \ref{lem:coif-mey}, this implies that
	\begin{align*}
	\sum_{\rm Case\, A}2^{-\frac{k_2}2}	\|\eqref{eq:nonlinear-space-2}\|_{L_\xi^2} &\les  	\sum_{\rm Case\, A}\bra{t}^{-1}2^{\frac{3\min(\bm k) - k_2}2}\bra{2^{k_2}}^{2} \| P_{k_2} (xV_{\mu,\thet})(t)\|_{L^2} \|\pko \phi_\theo(t)\|_{L^2}\\
	&\les 	\ve_1^2\sum_{\rm Case \,A}2^{\frac{\min(\bm k)}2} \bra{t}^{-1+[H(1)+H(0)]\de }  \bra{2^{k_1}}^{-N(0)} \bra{2^{k_2}}^{-N(1)+2}
\end{align*}
and
	\begin{align*}
		\sum_{\rm Case\, A}2^{-\frac{k_2}2}	\|\eqref{eq:nonlinear-space-2}\|_{L_\xi^2} &\les  	\sum_{\rm Case\, A}\bra{t}^{-1}2^{-\frac{k_2}2}\bra{2^{k_2}}^{8} \| P_{k_2} (xV_{\mu,\thet})(t)\|_{L^2} \|\pko \phi_\theo(t)\|_{L^\infty}\\
		&\les 	\ve_1^2\sum_{\rm Case \,A}2^{-k_2} \bra{t}^{-2+2H(1)\de + \frac12 \de}  \bra{2^{k_1}}^{-N(1)+10} \bra{2^{k_2}}^{-N(1)+8},
	\end{align*}
	respectively. By decomposing $2^{k_2} \le \bra{t}^{-\frac12}$ and $2^{k_2} \ge \bra{t}^{-\frac12}$, we complete the bound for \eqref{eq:nonlinear-space-2}. The estimates for \eqref{eq:nonlinear-space-3} can be obtained analogously.
\end{proof}

 We will use  the following lemma to prove weighted energy estimates and asymptotic behavior.
\begin{lemma}\label{lem:nonlinear-wei-dirac}
	Assume that $(\psi,A_\mu)$ satisfies the a priori assumptions \eqref{assum:energy}--\eqref{assum:s-norm} on $[0,T]$, for some $T>1$. Let  Let $t\in[0,T]$, $k \in \Z$, $\cL \in \cV_n, 0 \le n \le 2$, and $l \in\{1,2,3\}$. Then we have
	\begin{align}\label{eq:esti-nonlinear-dirac-space}
		\normo{P_k \left( x_l\fN_{\theta, \cL}^{\bm D} \right)(t)}_{L^2}  \les \ve_1^2  \bra{t}^{\wt{H}(n) \de} \bra{2^k}^{-N(n+1)-5}.
	\end{align}	
\end{lemma}

\begin{proof}
	%The proof is quite similar to that of Lemma \ref{lem:esti-nonlinear-dirac}. Indeed,
	Since $[x_l, \cR_{\cL'}]$ is smoother than $\cR_{\cL'}$, it suffices to consider only the first term of $\fN_{\theta,\cL}^{\bm D}$ as previously.
	By Plancherel's theorem, we see that
	\begin{align}\label{eq:aim-nonlinear-wei-dirac}
		2^{-\frac{k_2}2}\normo{\int_{\R^3}\rho_k(\xi)  \p_{\xi_l}\left[ \wh{W_{\mu, \thet, \cL_2} }(t,\eta)\al^\mu e^{-\theo it\bra{\xi-\eta}} \wh{\phi_{\theo,\cL_1}}(t,\xi-\eta) \right] d\eta}_{L_\xi^2}.
	\end{align}
	The derivative $\p_{\xi_l}$ falls on either $e^{-\theo it\bra{\xi-\eta}}$ or $\wh{\phi_{\theo,\cL_1}}(t,\xi-\eta)$. If the derivative falls on the phase modulation $e^{-\theo it\bra{\xi-\eta}}$, Lemma \ref{lem:esti-nonlinear-dirac} leads us to the bound of \eqref{eq:esti-nonlinear-dirac-space}. Thus, we have only to consider the case when the derivative falls on $\wh{\phi_{\theo,\cL_1}}(t,\xi-\eta)$.
	
	By \eqref{eq:high}, \eqref{eq:weight}, and H\"older's inequality, we have
	\begin{align}\label{eq:nonlinear-wei-dirac-l2}
		\begin{aligned}
			&2^{-\frac{k_2}2}\normo{\int_{\R^3}\rho_k(\xi)  \p_{\xi_l} \wh{P_{k_2}W_{\mu, \cL_2, \thet} }(t,\eta)\al^\mu e^{-\theo it\bra{\xi-\eta}} \wh{P_{k_1}x_l\phi_{\theo,\cL_1}}(t,\xi-\eta) d\eta}_{L_\xi^2} \\
			&\les 2^{\frac{3\min(\bm k) -k_2}2} \left\|P_{k_1}\left(x_l\phi_{\theo,\cL_1}\right)(t)  \right\|_{L^2} \|P_{k_2}W_{\mu, \cL_2, \thet}(t)\|_{L^2}\\
			&\les \ve_1^2 \bra{t}^{\left[H(n_1+1) + H(n_2)\right]\de} 2^{\frac{3\min(\bm k) -k_2}2} \bra{2^{k_1}}^{-N(n_1+1)-1}\bra{2^{k_2}}^{-N(n_2)},
		\end{aligned}
	\end{align}
	which implies the bound of \eqref{eq:esti-nonlinear-dirac-space} if $n_1 \neq n$.
	
	Let us consider the case $(n_1,n_2)= (n,0)$. Using \eqref{eq:nonlinear-wei-dirac-l2}, we get the bound of \eqref{eq:esti-nonlinear-dirac-space} if $2^{k_1} \ll 2^{k} \sim 2^{k_2}$ and $2^{k_1} \le \bra{t}^{-10\de}$. By \eqref{eq:elliptic-q} and \eqref{eq:decay-maxwell}, we have
	\begin{align}\label{eq:nonlinear-wei-dirac-infty}
		\begin{aligned}
			&2^{-\frac{k_2}2}\normo{\int_{\R^3}\rho_k(\xi)  \p_{\xi_l} \wh{P_{k_2}W_{\mu, \thet} }(t,\eta)\al^\mu e^{-\theo it\bra{\xi-\eta}} \wh{P_{k_1}x_l\phi_{\theo,\cL_1}}(t,\xi-\eta) d\eta}_{L_\xi^2} \\
			&\les 2^{-\frac{k_2}2} \left\|P_{k_1}\left(x_l\phi_{\theo,\cL_1}\right) (t) \right\|_{L^2} \|P_{k_2}W_{\mu,  \thet}(t) \|_{L^\infty}\\
			&\les \ve_1^2 \bra{t}^{\left[H(n_1+1) + H(1)\right]\de -1}  \bra{2^{k_1}}^{-N(n_1+1)-1}\bra{2^{k_2}}^{-N(1)}.
		\end{aligned}
	\end{align}
	If $ 2^{k_1} \ll 2^{k} \sim 2^{k_2}$ and $2^{k_1} \ge \bra{t}^{-10\de}$, then we get the desired bound by \eqref{eq:nonlinear-wei-dirac-infty}. In addition, we are done for the case $2^k \les 2^{k_1} \sim 2^{k_2}$ or the case $2^{k_2} \ll 2^{k} \sim 2^{k_1}$ and $2^k \le \bra{t}^\frac1{20} $.

	Now it remains to consider the case $2^{k_2} \ll 2^{k} \sim 2^{k_1}$ and $2^k \ge \bra{t}^{\frac1{20}}$. Making the change of variables in the integral of \eqref{eq:aim-nonlinear-wei-dirac}, for \eqref{eq:aim-nonlinear-wei-dirac} we need to estimate
	\begin{align*}
		2^{-\frac{k_2}2}\normo{\int_{\R^3}\rho_k(\xi)   e^{\thet it|\xi-\eta|}\wh{P_{k_2}x_lV_{\mu,  \thet} }(t,\xi-\eta)\al^\mu  \wh{P_{k_1}\psi_{\theo,\cL_1}}(t,\eta) d\eta}_{L_\xi^2}.
	\end{align*}
	This leads us that, by Bernstein's inequality,
	\begin{align*}
		&2^{-\frac{k_2}2}\normo{\int_{\R^3}\rho_k(\xi)   e^{\thet it|\xi-\eta|}\wh{P_{k_2}x_lV_{\mu,  \thet} }(t,\xi-\eta)\al^\mu  \wh{P_{k_1}\psi_{\theo,\cL_1}}(t,\eta) d\eta}_{L_\xi^2}\\
		&\les 2^{-\frac{k_2}2} \|P_{k_1}\psi_{\theo,\cL_1}(t)\|_{L^2} \|P_{k_2}\left(x_lV_{\mu,  \thet}\right)(t)\|_{L^\infty}\\
		&\les \ve_1^2 \bra{t}^{\left[H(n)+H(1)\right]\de} 2^{k_2} \bra{2^{k_1}}^{-N(n)}  \bra{2^{k_2}}^{-N(1)}.
	\end{align*}
	Here we used $2^k \ge \bra{t}^{\frac1{20}}$ to the bound $t^{\wt H(n)\de}$.  This completes the proof of \eqref{eq:esti-nonlinear-dirac-space}.
\end{proof}

\begin{lemma} \label{lem:decay-nonlinear-dirac}
	Assume that $(\psi,A_\mu)$ satisfies the a priori assumptions \eqref{assum:energy}--\eqref{assum:s-norm} on $[0,T]$, for some $T>1$. Let $t\in[0,T]$, $k \in \Z$, $\cL \in \cV_n, n \in \{0, 1, 2, 3\}$, and $\theta \in \{+, -\}$. Then we have
	\begin{align}\label{eq:decay-non-dirac}
		\normo{\mathfrak N_{\theta, \cL}^{\bm D}(t)}_{L^2}  \les \ve_1^2  \bra{t}^{-\frac23 + 2H(3) \de}.
	\end{align}
\end{lemma}
We note that Lemma \ref{lem:decay-nonlinear-dirac} is improved in terms of the regularity condition, whereas decay effect has a loss.

\begin{proof}[Proof of Lemma \ref{lem:decay-nonlinear-dirac}]
	For $\cL_1 \in \cV_{n_1}$ and $\cL_2 \in \cV_{n_2}$ ($0\le n_1 +n_2 =n$), we have
	\begin{align*}
		\cL (A_\mu(t) \al^\mu \psi(t)) = \sum_{\theo,\thet \in \{\pm\}}  A_{\mu,\thet, \cL_2}(t) \al^\mu \cL_1 \psi_{\theo, \cL_1}(t). 
	\end{align*}
	Then it suffices to show 	
	\begin{align*}
		\|A_{\mu,\cL_2,\thet}(t) \al^\mu \psi_{\theo,\cL_1}(t) \|_{H^{N(n) }} \les \ve_1^2 \bra{t}^{- \frac23 + 2H(3)\de}
	\end{align*}
	for $\cL_l \in \cV_{n_l}(l=1,2)$.  	To this end, we decompose $A_{\mu,\cL_2, \thet} = \sum_{k_2 \in \Z} \pkt A_{\mu,\cL_2, \thet}$. If $2^{k_2} \le \bra{t}^{-\frac23}$, H\"older inequality with \eqref{eq:high} yields that
	\begin{align*}
		\sum_{2^{k_2} \le \bra{t}^{-\frac23}} \|\pkt A_{\mu,\cL_2,\thet}(t) \al^\mu \psi_{\theo,\cL_1}(t) \|_{H^{N(n)}} &\les \sum_{2^{k_2} \le \bra{t}^{-\frac23}} 2^{k_2} \|\psi_{\theo,\cL_1}(t)\|_{H^{N(n)}}\|\pkt W_{\mu,\cL_2,\thet}(t)\|_{H^{N(n)}} \\
		&\les \ve_1^2 \bra{t}^{-\frac23 + 2H(3)\de}.
	\end{align*}
	Let us consider $2^{k_2} \ge \bra{t}^{-\frac23}$. If $n \le 2 $, using \eqref{eq:high}, \eqref{eq:decay-dirac}, and \eqref{eq:decay-maxwell},	we see that
	\begin{align}\label{eq:l2linfty}
\begin{aligned}
			\| \pkt A_{\mu,\cL_2,\thet}(t) \al^\mu \psi_{\theo,\cL_1}(t) \|_{H^{N(n)}} &\les   2^{-\frac {k_2}2} \|\psi_{\theo,\cL_1}(t) \|_{H^{N(n)}}\| \pkt W_{\mu,\cL_2,\thet}(t)\|_{L^\infty} \\
		&\qquad +   2^{-\frac {k_2}2} \|\psi_{\theo,\cL_1}(t) \|_{L^\infty}\| \pkt W_{\mu,\cL_2,\thet}(t)\|_{H^{N(n)}}\\
		&\les  \ve_1^2 \bra{t}^{-\frac23 + 2H(3)\de}.
\end{aligned}
	\end{align}
	
	Next, we treat the case $n=3$.  When    $(n_1,n_2) = (n,0)$ or $(n_1,n_2) = (0,n)$,	we similarly have
	\begin{align*}
	\| \pkt A_{\mu,\cL_2,\thet}(t) \al^\mu \psi_{\theo,\cL_1}(t) \|_{H^{N(n)}} &\les  2^{-\frac {k_2}2} \|\bra{D}^{N(n)}\psi_{\theo,\cL_1}(t) \|_{L^p}\|\bra{D}^{N(n)} \pkt W_{\mu,\cL_2,\thet}(t)\|_{L^q} \\
	&\les  \ve_1^2 \bra{t}^{-\frac23 + 2H(3)\de},
\end{align*}
where 	 $(p,q)= (\infty,2)$ when $(n_1, n_2) = (0,n)$ and $(p,q)= (2,\infty)$ when  $(n_1, n_2) = (n,0)$. Then it remains to handle the case $1 \le n_1,n_2 \le n-1$ when $n=3$.  However, this case can be done by the same estimates to \eqref{eq:l2linfty}. Therefore,    we complete the proof of \eqref{eq:decay-non-dirac}. 
\end{proof}

\subsection{Nonlinear estimates for Maxwell part}
We denote the nonlinear term of Maxwell part \eqref{eq:vec-wave} by
\begin{align*}
	\fN_{\mu, \cL}^{\bm M} := \frac12 |D|^{-\frac12}\cL\bra{\psi, \al_\mu \psi}.
\end{align*}
Then for $\theta' \in \{+, -\}$
\begin{align}\label{eq:maxwell-pt}
	(i\p_t + \theta'|D|)W_{\mu, \cL, \theta'} = \theta'\fN_{\mu, \cL}^{\bm M},\qquad i\p_tV_{\mu, \cL, \theta'} = \theta' e^{-i\theta' t |D|}\fN_{\mu, \cL}^{\bm M}.
\end{align}

%\begin{lemma} \label{lem:esti-profile-maxwell}
%	Assume that $(\psi, A_\mu)$ is a solution to \eqref{maineq:md-lorenz} on $[0,T]$, for some $T > 1$ and satisfies \eqref{assum:energy}--\eqref{assum:s-norm}. Let $\cL \in \cV_n, n \in \{0, 1, 2, 3\}$. Then we have
%	\begin{itemize}
	%		\item[(i)] \begin{align}\label{eq:nonlinear1}
		%			\normo{P_k \fN_{\mu, \cL}^M(t)}_{L^2} + \normo{P_k \p_t V_{\mu, \cL, \theta'}(t)}_{L^2} \les \ve_1^2 2^{\frac {k^-}2} \bra{t}^{H_M(n) \de} \min(2^{k^-},\bra{t}^{-1})\bra{2^k}^{-N(n)+ 5},
		%		\end{align}
	%		where
	%		\begin{align*}
		%			H''_M(0):=5, \quad H''_M(n):= H(n) +160.
		%		\end{align*}
	%		\item[(ii)] For $k \in \Z$, we have \begin{align}\label{eq:pro-decay-maxwell}
		%			\sum_{j \ge -k^-} \normo{e^{\theta' it|D|} V_{\mu, \theta'}^{j,k}}_{L^\infty} \les \ve_1 2^{k(1-\de_1)}\min(\bra{t}^{-1}, 2^k) \bra{2^k}^{-N_0 + d_1+2}
		%		\end{align}
	%	\end{itemize}
%	
%\end{lemma}

%\begin{proof}[Proof of \eqref{eq:pro-decay-maxwell}] By the linear dispersive estimate \eqref{eq:dispersive-maxwell} for wave equations, we have
%\begin{align*}
%	\|e^{\theta' it |D|}V_{\mu,\theta'}^{j, k}\|_{L^\infty} \les 2^{\frac32 k} \min \left( 1, 2^{j}\bra{t}^{-1}\right)\|\qjk V_{\mu,\theta'}\|_{L^2}.
%\end{align*}
%Then, the first inequality in \eqref{eq:zero-vec-1} implies \eqref{eq:pro-decay-maxwell}.
%\end{proof}

\begin{lemma} \label{lem:esti-nonlinear-wave}
	Assume that $(\psi,A_\mu)$ is a solution to \eqref{maineq:md-lorenz} on $[0,T]$, for some $T>1$ and satisfies \eqref{assum:energy}--\eqref{assum:s-norm}. Let $t\in[0,T]$, $k \in \Z$, $\cL \in \cV_n, n \in \{0, 1, 2, 3\}$, and $\theta' \in \{+, -\}$. Then we have
	\begin{align}\label{eq:esti-nonl-max-time}
		\normo{P_k \fN_{\mu, \cL}^{\bm M}(t)}_{L^2} + \normo{P_k \p_t V_{\mu, \cL, \theta'}(t)}_{L^2} \les \ve_1^2 \bra{t}^{\wt H(n) \de} \bra{2^k}^{-N(n) + 5} \min(2^{k^-},\bra{t}^{-1}).
	\end{align}
\end{lemma}
Note that $\wt H(n)$ is defined in Lemma \ref{lem:esti-nonlinear-dirac}.
\begin{proof}[Proof of Lemma \ref{lem:esti-nonlinear-wave}]
	We partition the range of $|\eta|$ and $|\xi+\eta|$ with dyadic numbers $2^{k_1}$ and $2^{k_2}$ $(k_1,k_2 \in \Z)$, respectively. Then, it suffices to prove
	\begin{align}\label{eq:aim-nonlinear-maxwell}
	\begin{aligned}
			&\sum_{k_1, k_2 \in \Z}\|P_k |D|^{-\frac12} \bra{P_{k_1}\psi_{\theo,\cL_1}(t),\al_\mu P_{k_2} \psi_{\thet,\cL_2}(t) }\|_{L^2} \\
		&\hspace{4cm} \les \ve_1^2 \bra{t}^{\wt H(n)} \bra{2^k}^{-N(n+1)+5} \min (\bra{t}^{-1}, 2^k)
	\end{aligned}
	\end{align}
	for $\cL_j \in \cV_{n_l} \,(l=1,2)$ and $n_1 + n_2 =n$. We also assume that $n_1 \ge n_2$ due to the symmetry between $\psi_\theo$ and $\psi_\thet$. By \eqref{eq:high}, a direct calculation shows that
	\begin{align*}
		&\|P_k |D|^{-\frac12} \bra{P_{k_1}\psi_{\theo,\cL_1}(t),\al_\mu P_{k_2} \psi_{\thet,\cL_2}(t)}\|_{L^2} \\
		&\les 2^{\frac{3\min(\bm k)-k}2} \|P_{k_1} \psi_{\theo,\cL_1}(t)\|_{L^2}\|P_{k_2} \psi_{\thet,\cL_2}(t)\|_{L^2}\\
		&\les \ve_1^2 \bra{t}^{\left[H(n_1) + H(n_2)\right]\de}2^{\frac{3\min(\bm k)-k}2} \bra{2^{k_1}}^{-N(n_1)}\bra{2^{k_2}}^{-N(n_2)}.
	\end{align*}
	Since $H(n_1)+H(n_2) \le \wt{H}(n)$ and $\min(N(n_1),N(n_2)) = N(n)$, we get \eqref{eq:aim-nonlinear-maxwell} if $2^k \le \bra{t}^{-1}$.
	
	Let us assume that $2^k \ge \bra{t}^{-1}$. In this case, we estimate by dividing the distribution of vector fields as follows:
	\begin{align*}
		(n_1,n_2) \in \{(2,1),(1,1)\} \cup \{(n,0): n \ge 1\} \cup \{(0,0)\}.	
	\end{align*}
	
	\emph{Estimates for $(n_1,n_2) \in \{(2,1),(1,1)\}$.} Analogously to the proof of Lemma \ref{lem:esti-nonlinear-dirac}, we exploit the space resonance. Taking Fourier transform to integrand of left-hand side in \eqref{eq:aim-nonlinear-maxwell}, we see that
	\begin{align}\label{eq:bilinear-fourier}
		\begin{aligned}
			&\cF \left[	P_k |D|^{-\frac12} \bra{P_{k_1}\psi_{\theo,\cL_1}(t),\al_\mu P_{k_2} \psi_{\thet,\cL_2}(t)}\right](\xi) \\
			&= \int_{\R^3} \rho_k(\xi)  |\xi|^{-\frac12} e^{itq_{12}(\xi,\eta)} \bra{ \wh{P_{k_1}\phi_{\theo,\cL_1}}(t,\xi+\eta),\al_\mu \wh{P_{k_2} \phi_{\thet,\cL_2}}(t,\eta)} d\eta,
		\end{aligned}
	\end{align}
	where
	\begin{align*}
		q_{12}(\xi,\eta) = \theo \bra{\xi+\eta} - \thet\bra{\eta}.
	\end{align*}
	As we observed in \eqref{eq:resonance-space-maxwell}, we exploit the resonance cases according to $\theo =\thet$ and $\theo \neq \thet$.
	
	Let us first consider the case $\theo=\thet$. By performing the integration by parts in $\eta$, \eqref{eq:bilinear-fourier} is the sum of the following:
	\begin{align}
		&t^{-1}	\int_{\R^3}  e^{itq_{12}(\xi,\eta)} \nabla_\eta\wt M_{12,\bm k}(\xi,\eta)\bra{ \wh{P_{k_1}\phi_{\theo,\cL_1}}(t,\xi+\eta),\al_\mu \wh{P_{k_2} \phi_{\thet,\cL_2}}(t,\eta)} d\eta,\label{eq:nonlinear-max-space-1}\\
		&t^{-1}	\int_{\R^3}  e^{itq_{12}(\xi,\eta)} \wt M_{12,\bm k}(\xi,\eta)\bra{ \wh{P_{k_1}x\phi_{\theo,\cL_1}}(t,\xi+\eta),\al_\mu \wh{P_{k_2} \phi_{\thet,\cL_2}}(t,\eta)} d\eta,\label{eq:nonlinear-max-space-2}\\
		&t^{-1}\int_{\R^3}  e^{itq_{12}(\xi,\eta)} \wt M_{12,\bm k}(\xi,\eta)\bra{ \wh{P_{k_1}\phi_{\theo,\cL_1}}(t,\xi+\eta),\al_\mu \wh{P_{k_2} x\phi_{\thet,\cL_2}}(t,\eta)} d\eta,\label{eq:nonlinear-max-space-3}
	\end{align}
	where
	\begin{align*}
		\wt M_{12,\bm k}(\xi,\eta) = |\xi|^{-\frac12} \frac{\nabla_\eta q_{12}(\xi,\eta)}{|\nabla_\eta q_{12}(\xi,\eta)|^2} \rho_{\bm k}(\xi,\eta).
	\end{align*}
	A direct calculation leads us that
	\begin{align*}
		\left| 	\nabla_\eta^\ell\wt M_{12,\bm k}(\xi,\eta)\right| \les 2^{-\frac {3k}2} 2^{-\ell \min(k_1,k_2)} \max\left(\bra{2^{k_1}},\bra{2^{k_2}}\right)^{3+\ell}
	\end{align*}
	for $\ell =0,1$. Then, by H\"older's inequality, we have
	\begin{align*}
		\left\|\eqref{eq:nonlinear-max-space-1}\right\|_{L^2} &\les \bra{t}^{-1} 2^{\frac {3\min(\bm k)-3k}2} 2^{-\min(k_1,k_2)} \max\left(\bra{2^{k_1}},\bra{2^{k_2}}\right)^4 \|P_{k_1}\psi_{\theo,\cL_1}(t)\|_{L^2}\|P_{k_2}\psi_{\thet,\cL_2}(t)\|_{L^2}\\
		&\les \ve_1^2 \bra{t}^{[H(n_1) + H(n_2+1)]\de -1}2^{\frac {3\min(\bm k)-3k}2 - \min(k_1,k_2)+ k_2 } \bra{2^{k_1}}^{-N(n_1)+4}\bra{2^{k_2}}^{-N(n_2+1)+4}.
	\end{align*}
	The above estimate follows from \eqref{eq:high} and \eqref{eq:esti-l2-k} for $\psi_{\theo,\cL_1}$ and $\psi_{\thet,\cL_2}$, respectively. Summing over $k_1, k_2$, we obtain the bound for \eqref{eq:nonlinear-max-space-1}.
	
	By \eqref{eq:high} and \eqref{eq:weight}, we have
	\begin{align*}
		\left\|\eqref{eq:nonlinear-max-space-2}\right\|_{L^2} &\les \bra{t}^{-1} 2^{\frac {3\min(\bm k)-3k}2} \max\left(\bra{2^{k_1}},\bra{2^{k_2}}\right)^3 \|P_{k_1}(x\phi_{\theo,\cL_1})(t)\|_{L^2}\|P_{k_2}\psi_{\thet,\cL_2}(t)\|_{L^2}\\
		&\les \ve_1^2 \bra{t}^{[H(n_1+1) + H(n_2)]\de -1}2^{\frac {3\min(\bm k)-3k}2 } \bra{2^{k_1}}^{-N(n_1+1)+3}\bra{2^{k_2}}^{-N(n_2)+3}.
	\end{align*}
	Since $H(n_1 + 1)+H(n_2) = H(n) +10$ and $N(n) -5 \le N(n_1 + 1)-3 $, this implies \eqref{eq:aim-nonlinear-maxwell}. Similarly, the estimate of \eqref{eq:nonlinear-max-space-3} can be carried out and hence we omit the details.

	We now turn to the case $\theo \neq \thet$. In view of \eqref{eq:resonance-space-maxwell}, the phase interaction $q_{12}(\xi,\eta)$ exhibits the space resonance when $\xi+2\eta = 0$. If $2^{\min(\bm k)} \ll 2^{\max(\bm k)}$, then by the frequency relation, we have $2^{\max(\bm k)} \sim |\xi+2\eta|$. Using the integration by parts in $\eta$, this case can be treated in a similar way to the case $\theo =\thet$. Thus, we have only to consider
	$$
	2^{\min(\bm k)} \sim 2^{\max(\bm k)}.
	$$
	The estimate for this case is straightforward. Indeed, by \eqref{eq:high} and \eqref{eq:decay-dirac}, we see that
	\begin{align*}
		\|\eqref{eq:bilinear-fourier}\|_{L^2} \les \ve_1^2 \bra{t}^{[H(n_1)+  H(n_2+1)]\de-1} 2^{\frac {-k+k_2}2} \bra{2^{k_1}}^{-N(n_1)}\bra{2^{k_2}}^{-N(n_2+1)+2}.
	\end{align*}
	This finishes the proof of \eqref{eq:aim-nonlinear-maxwell} when $(n_1,n_2)\in \{(2,1),(1,1)\} $.
	
	\emph{Estimates for $(n_1,n_2) = (n,0)$ and $n \ge 1$.} From \eqref{eq:high} and \eqref{assum:s-norm}, it follows that
	\begin{align*}
		&\|P_k |D|^{-\frac12} \bra{P_{k_1}\psi_{\theo,\cL}(t),\al_\mu P_{k_1} \psi_{\thet}(t)}\|_{L^2} 	\\
		&\les 2^{\frac{3\min(\bm k)-k}2} \|P_{k_2} \psi_{\theo,\cL}(t)\|_{L^2}\|P_{k_2} \psi_{\thet}(t)\|_{L^2}\\
		&\les \ve_1^2 \bra{t}^{H(n) \de}2^{\frac{3\min(\bm k)-k}2}2^{(1+5H(2)\de)k_2} \bra{2^{k_1}}^{-N(n)}\bra{2^{k_2}}^{-N(0)+2}.
	\end{align*}
	If $2^{\min(\bm k) + k_2} \les \bra{t}^{-1+130\de}$ or $2^{k_2} \ge \bra{t}^{\frac1{40}}$, this estimate implies \eqref{eq:aim-nonlinear-maxwell}. If $2^{\min(\bm k) + k_2} \gg \bra{t}^{-1+130\de}$ and $2^{k_2} \le \bra{t}^{\frac1{40}}$, we decompose $P_{k_2}\phi_{\thet}$ into $ \phi^{\le J,k_2}_{\thet}$ and $ \phi^{> J,k_2}_{\thet}$ with $J= C2^{k_2}\bra{t}$ for some $C \ll 1$. And we obtain
	\begin{align}\label{eq:esti-decompose}
		\begin{aligned}
			\normo{e^{-\theta it\brad} \phi^{\le J,k_2}_{\thet}(t)}_{L^\infty} &\les \ve_1 \bra{t}^{-\frac32 }  2^{-\frac {k_2}2 +  \frac {k_2}{2000}}  \bra{2^{k_2}}^{-19},\\
			\normo{ \phi^{> J,k_2}_{\thet}(t)}_{L^2} &\les \ve_1 \bra{t}^{-1 + H(1)\de} 2^{-k_2}  \bra{2^{k_2}}^{-N(1)},
		\end{aligned}
	\end{align}
	which follow from \eqref{eq:zero-vec-2} and \eqref{eq:elliptic-q}, respectively. On the one hand, $L^2 \times L^\infty$ estimate leads us to the bound
	\begin{align}\label{eq:bilinear-maxwell}
		\begin{aligned}
			&\left\|P_k |D|^{-\frac12} \bra{P_{k_1}\psi_{\theo,\cL}(t),\al_\mu  e^{-\thet it\bra{D}}\phi_{\thet}^{\le J,k} (t) } \right\|_{L^2} 	\\
		&\hspace{4cm}\les \ve_1^2 \bra{t}^{H(n)\de -\frac32 } 2^{-\frac {k+k_2}2+\frac {k_2}{1000}} \bra{2^{k_1}}^{-N(n)} \bra{2^{k_2}}^{-19}.
		\end{aligned}
	\end{align}
	If $2^{k_2} \les 2^{k_1}$, \eqref{eq:bilinear-maxwell} yields \eqref{eq:aim-nonlinear-maxwell}  over $2^{\min(\bm k) + k_2} \gg \bra{t}^{-1+130\de}$ and $2^{k_2} \le \bra{t}^{\frac1{40}}$. If $2^{k_1} \ll 2^{k_2}$, then in view of \eqref{eq:aim-nonlinear-maxwell} with $n=1$, we can observe that the factor $\bra{2^{k_2}}^{-19}$ is not sufficient for $\bra{2^k}^{-N(1)+5}$ when $ 2^{k} \ge 1$. However, from the frequency restriction $2^{k_2} \le \bra{t}^{\frac1{40}}$, we have
	\begin{align*}
		\bra{2^{k_2}}^{-19} \les \bra{t}^{\frac3{20}}\bra{2^k}^{-N(1)+5},
	\end{align*}
	which enables to obtain the desired bound. On the other hand, by H\"older's inequality, we have
	\begin{align*}
		\begin{aligned}
			&\left\|P_k |D|^{-\frac12} \bra{P_{k_1}\psi_{\theo,\cL}(t),\al_\mu  e^{-\thet it\bra{D}}\phi_{\thet}^{> J,k}(t) } \right\|_{L^2} \\
			&\hspace{4cm}\les \ve_1^2 \bra{t}^{[H(n)+H(1)]\de -1} 2^{\frac{3\min(\bm k)}2 - \frac k2 -k_2} \bra{2^{k_1}}^{-N(n)} \bra{2^{k_2}}^{-N(1)}.
		\end{aligned}
	\end{align*}
	Them, \eqref{eq:aim-nonlinear-maxwell} follows from the sum over $2^{\min(\bm k) + k_2} \gg \bra{t}^{-1+130\de}$ and $2^{k_2} \le \bra{t}^{\frac1{40}}$.

	\emph{Estimates for $(n_1,n_2)=(0,0)$.} By the symmetry between $\psi_{\theo}$ and $\psi_{\thet}$, we only consider the case $k_1 \le k_2$. If $2^k \ge \bra{t}^{2\de}$, using \eqref{eq:high} and \eqref{eq:decay-dirac}, we see that
	\begin{align*}
		\sum_{k_1 \le k_2  }\|P_k |D|^{-\frac12} \bra{P_{k_1}\psi_{\theo}(t),\al_\mu P_{k_2} \psi_{\thet}(t) }\|_{L^2} &\les 	\sum_{k_1 \le k_2} 2^{-\frac k2} \|\pko \psi_\theo(t)\|_{L^\infty} \|\pkt \psi_\thet(t) \|_{L^2}\\
		&\les \ve_1^2 	\sum_{k_1 \le k_2}  2^{\frac {-k+k_1}2}\bra{t}^{11\de -1} \bra{2^{k_1}}^{-N(1)+2}\bra{2^{k_2}}^{-N(0)}\\
		&\les \ve_1^2 \bra{t}^{\de-1}\bra{2^k}^{-N(0)+5}.
	\end{align*}
	If $2^{\min(\bm k) + k_2} \le \bra{t}^{-1+5\de}$, \eqref{assum:s-norm} yields that
	\begin{align*}
		\sum_{k_1 \le k_2  }\|P_k |D|^{-\frac12} \bra{P_{k_1}\psi_{\theo}(t) ,\al_\mu P_{k_2} \psi_{\thet}(t)  }\|_{L^2} &\les 	\sum_{k_1 \le k_2} 2^{\frac{3\min(\bm k)-k}2} \|\pko \psi_\theo(t)\|_{L^2} \|\pkt \psi_\thet(t)\|_{L^2}\\
		&\les \ve_1^2	\sum_{k_1 \le k_2} \bra{t}^{\de}2^{\min(\bm k)+ k_2}\\
		&\les \ve_1^2 \bra{t}^{6\de -1}.
	\end{align*}

	Let us consider the mid frequency case  $2^{\min(\bm k) + k_2} \ge \bra{t}^{-1+5\de}$ and  $2^k \le \bra{t}^{2\de}$. Setting $J:= C2^{k_2}\bra{t}$ for some $C\ll 1$, we decompose $P_{k_2}\phi_{\thet}$ into $\phi_{\thet}^{\le J,k_2}$ and $\phi_\thet^{> J, k_2}$. Then we obtain \eqref{eq:esti-decompose} and the remaining estimates can be done similarly to those for the case $(n_1,n_2) = (n,0)$ and $n \ge 1$.
\end{proof}

The following lemma will be used in the proof of weighted estimates and scattering for the Maxwell part.
\begin{lemma}\label{lem:nonlinear-wei-maxwell}
	Assume that $(\psi,A_\mu)$ satisfies the a priori assumptions \eqref{assum:energy}--\eqref{assum:s-norm} on $[0,T]$, for some $T>1$. Let $t \in [0,T]$, $k\in \Z$, $\cL \in \cV_n, 0 \le n \le 2$, and $l \in \{1,2,3\}$. Then we have
	\begin{align}\label{eq:esti-nonlinear-maxwell-space}
		\normo{P_k \left( x_l\fN_\cL^{\bm M} \right)(t)}_{L^2}  \les \ve_1^2  \bra{t}^{\wt{H}(n) \de} \bra{2^k}^{-N(n+1)+5}.
	\end{align}	
\end{lemma}
\begin{proof}
		As described in the proof of Lemma \ref{lem:nonlinear-wei-dirac}, one may consider only 	
		\begin{align}\label{eq:aim-nonlinear-wei-maxwell}
			2^{-\frac k2}\normo{\int_{\R^3}\rho_k(\xi)   \bra{e^{-\theo it\bra{\xi+\eta}}\p_{\xi_l} \wh{\phi_{\theo,\cL_1}}(t,\xi+\eta), \al_\mu \wh{\psi_{\thet,\cL_2}}(t,\eta)} d\eta}_{L_\xi^2}.
		\end{align}
		Note that $n_1 \le n_2$ due to the symmetry. By \eqref{eq:high} and \eqref{eq:weight}, we see that
		\begin{align*}
			\|\pko( x_l \phi_{\theo,\cL_1})(t)\|_{L^2} &\les \ve_1 \bra{t}^{H(n_1+1)\de} \bra{2^{k_1}}^{-N(n_1+1)},\\
			\|\pkt \phi_{\thet,\cL_2}(t) \|_{L^2} &\les \ve_1 \bra{t}^{H(n_2)\de} \bra{2^{k_2}}^{-N(n_2)}.
		\end{align*}
		H\"older's inequality yields that
		\begin{align}\label{eq:esti-nonlinear-maxwell-1}
			\left|\eqref{eq:aim-nonlinear-wei-maxwell}\right| \les \sum_{k_1,k_2 \in \Z} 2^{\frac{3\min(\bm k) - k}2}	\|\pko (x_l \phi_{\theo,\cL_1})(t)\|_{L^2} 	\|\pkt \phi_{\thet,\cL_2}(t)\|_{L^2}.
		\end{align}
	Since $H(n_1+1) + H(n_2) \le \wt H(n)$ for $n =1,2$, we obtain \eqref{eq:esti-nonlinear-maxwell-space} apart from the case	$n=0$. Let us consider the case $n=0$. This case corresponds to the case $(n_1,n_2)=(0,0)$. Then, by the symmetry, it suffices to handle $2^{k_1} \les 2^{k_2}$. If $2^k \le \bra{t}^{-1}$, we simply have \eqref{eq:esti-nonlinear-maxwell-space} due to \eqref{eq:esti-nonlinear-maxwell-1}, If $2^k \ge \bra{t}^{-1}$, using \eqref{eq:decay-dirac}, we estimate
		\begin{align*}
		\left|\eqref{eq:aim-nonlinear-wei-maxwell}\right| &\les \sum_{2^{k_1} \les 2^{k_2}} 2^{ -\frac{ k}2}	\|\pko (x_l \phi_{\theo,\cL_1})(t)\|_{L^2} 	\|\pkt \psi_{\thet,\cL_2}(t)\|_{L^\infty} \\
		&\les \sum_{2^{k_1} \les 2^{k_2}} \ve_1^2 \bra{t}^{-1+2H(1)\de}2^{ \frac{- k+k_2}2} \bra{2^{k_1}}^{-N(1)} \bra{2^{k_2}}^{-N(1)+2}.
	\end{align*}
	This completes the proof of \eqref{eq:esti-nonlinear-maxwell-space}.
\end{proof}

%%%%%%%%%%%%%%%%%%%%%%%%%%%%%%%%%%%%%%%%%%%%%%%%%%%%%%%%%%%%%%%%%%%%%%%%%%%%%%%%%%%%%%%%%%%%%%%%%%%%%%%%%%%%%%%%%%%%%%%%%%%%%%%%%%%%%%%%%%%%%%%%%%%%%%%%%%%%%%%%%%%%%%%%%%%%%%%%%%%%%%%%%%%%%%%%%%%%%%
\section{Proof of main theorem}\label{main proof}

For any initial data \eqref{eq:initial} satisfying \eqref{condition-initial}, one can readily show the existence of local solution $(\psi, A_\mu)$ to \eqref{eq:maineq-half} with $(\psi, W_\mu) \in C([0, T]; H^{N(0)} \cap C^1([0, T]; H^{N(0)-1})$ for some $T > 0$ by the regularity persistence based on the local theory of \cite{anfosel2010}. Then the standard approximation with $e^{-\lambda|x|^2}\bra{x}^N$ is applicable to the system \eqref{eq:maineq-half} and enables us to get, for sufficiently small $T$,
$$
\sup_{0 \le t \le T}\sum_{n = 1, 2, 3 }\left[\|\bra{x}^n\bra{D}^{n}(\psi(t), W_\mu(t))\|_{H^{N(n)}}\right] \le C(T, \psi(0), W_\mu(0)).
$$
Now we show $\sup_{0 \le t \le T}\|\bra{\xi}^{20}|\xi|^\frac12\wh\psi (t, \xi)\|_{L_\xi^\infty} < +\infty$.
In fact, by \eqref{eq:interaction-dirac} we have
\begin{align*}
\bra{\xi}^{20}|\xi|^\frac12|\wh{\phi_{\theta}}(t)| \le  |\bra{\xi}^{20}|\xi|^\frac12|\wh{\psi_{\theta}}(0)| + c\sum_{\theo \in \{\pm\}}\int_0^t\int_{\mathbb R^3}\bra{\xi}^{20}|\xi|^\frac12|\eta|^{-\frac12}|\wh{W_\mu}(s,\eta)||\wh{\phi_{\theta_1}}(s,\xi-\eta)|\,d\eta ds.
\end{align*}
Let $f(t) = \|\bra{\xi}^{20}|\xi|^\frac12\wh{\psi_{\theta}}(t)\|_{L_\xi^\infty}$. Then using 
$$
\bra{\xi}^{20}|\xi|^\frac12 \les \bra{\xi-\eta}^{20}|\xi-\eta|^\frac12 + \bra{\xi-\eta}^{20}|\eta|^\frac12 + \bra{\eta}^{20}|\xi-\eta|^\frac12 + \bra{\eta}^{20}|\eta|^\frac12,
$$
the factor $|\eta|^\frac12$ gets rid of the singularity of $|\eta|^{-\frac12}$ and H\"older's inequality leads us to the bound
$$
f(t) \le C_T + C\int_0^t \|W_\mu(s)\|_{H^{N(0)}}f(s) \,ds\;\;\mbox{and}\;\;C_T = C T \sup_{0 \le t \le T}\|(\psi(t), W_\mu(t))\|_{H^{N(0)}}^2.
$$
Hence Gronwall's inequality shows $\sup_{0 \le t \le T}f(t) \le C_T$.
If $T$ is sufficiently small, then one can readily verify the local solution satisfies the a priori assumptions \eqref{assum:energy}--\eqref{assum:s-norm} via \eqref{eq:weight-lorentz-dirac} and \eqref{eq:weight-lorentz-wave}.

\begin{rem}
The above local properties are not good enough for the global extension due to the bad growth in time. One may take bootstrap argument based on the weighted energy $\||x|^2 (\phi_\theta, V_{\mu, \theta'})\|_{H^m}$ as in \cite{CKLY2022}. But the square weight makes a trouble because it plays a role as the second derivative in the Fourier side and gives rise to a serious singularity to the nonlinearity. This obstacle prevents us from closing bootstrap  due to the lack of null structure like \eqref{eq:lack-null}. This is a reason why we use bootstrap argument based on the vector-field energy method.
\end{rem}

Now we extend the solution globally by time continuity. To this end, we go through several steps of bootstrap argument, which are energy, weighted energy, and scattering norm estimates to be implemented under a priori assumptions \eqref{assum:energy}--\eqref{assum:s-norm}. Each step is described as the proposition below in which solution turns out to have more improvement than a priori assumptions. Since the proof of propositions when $T \les 1$ is similar to (much simpler than) the case $T \gg 1$, we assume that $T \gg 1$ from now. Let us introduce propositions running bootstrap.
\begin{prop}[Energy estimates]\label{prop:energy}
	Assume that $(\psi,A_\mu)$ satisfies a priori assumptions \eqref{assum:energy}--\eqref{assum:s-norm}.  Let $t \in [0,T]$, $n \in \{0, 1, 2, 3\}$, $\theta, \theta' \in \{+, -\}$, and $\cL \in \cV_n$. Then we have
	\begin{align}
		\normo{\psi_{\theta, \cL}(t)}_{H^{N(n)}} &\les \ve_1^2 \bra{t}^{H(n)\de},\label{eq:energy-high-d}\\
		\normo{ W_{\mu, \cl, \theta'}(t)}_{H^{N(n)}} &\les \ve_1^2 \bra{t}^{H(n)\de}.\label{eq:energy-high-m}
	\end{align}
\end{prop}

%Proposition \ref{prop:energy} will be proven in the next two sections. Since the vector fields are commuting with wave operator, we prove \eqref{eq:energy-high-m} by using the standard approach. As observed in Section \ref{sec:vector}, the rotation and Lorentz boost in $\cV_n$ do not commute with the Dirac operator and half Klein-Gordon operator. Thus, we need to handle the case that extended differential operators with smooth operator-valued coefficients fall on the profiles. In Section \ref{sec:energy1} we will prove the energy estimates for the standard differential operators in $\cV_n$. Then we will treat those for the extended differential operators in Section \ref{sec:energy2}.
\begin{prop}[Weighted energy estimates]\label{prop:weighted}
	Assume that $(\psi,A_\mu)$ satisfies a priori assumptions \eqref{assum:energy}--\eqref{assum:s-norm}.  Let $t \in [0,T]$, $k\in \Z$, $n \in \{0, 1, 2\}$, $\theta, \theta' \in \{+, -\}$, $l \in \{1, 2, 3\}$, and $\cL \in \cV_n$.  Then, we have
	\begin{align}
		\bra{2^k}\|P_k\left(x_l \phi_{\theta,\cL}\right)\|_{L^2} &\les \ve_1^2 \bra{t}^{H(n+1)\de}\bra{2^k}^{-N(n+1)}	\label{eq:weighted-dirac},\\
		2^k\|P_k\left(x_l V_{\theta,\cL}\right)\|_{L^2} &\les \ve_1^2 \bra{t}^{H(n+1)\de}\bra{2^k}^{-N(n+1)}	\label{eq:weighted-wave}.
	\end{align}
\end{prop}

\begin{prop}[Scattering norm estimates]\label{prop:s-norm}
	Assume that $(\psi,A_\mu)$ satisfies a priori assumptions \eqref{assum:energy}--\eqref{assum:s-norm}.  Let $t \in [0,T]$, $\theta, \theta' \in \{+, -\}$. Then we have
	\begin{align}
		\normo{\psi_{\theta}(t)}_{\bm D} &\les \ve_1^2 ,\label{eq:s-norm-d}\\
		\| V_{\mu,\theta'}(t)\|_{\bm M} &\les \ve_1^2.\label{eq:s-norm-m}
	\end{align}
\end{prop}
\noindent The proof of propositions will be given in Section \ref{sec:energy1}, Section \ref{sec:weight}, and Sections \ref{sec:s-dirac} and \ref{sec:s-maxwell}, respectively.

\emph{Proof of the asymptotic behavior for the Maxwell-Dirac system.} The above propositions close the bootstrap argument and lead us the global existence immediately. Now we show the asymptotic behavior parts of the main theorem. Let us first consider the modified scattering for Dirac part. In Section \ref{sec:s-dirac} we will show that for $0 < t_1 \le t_2 $,
\begin{align*}
	\left\| \wh{P_k\Psi_\theta}(t_2,\xi) -\wh{P_k\Psi_\theta}(t_1,\xi) \right \|_{L_\xi^\infty} \les \ve_1^2\bra{t_2}^{-\de}2^{-\frac k2 + \frac1{100} k} \bra{2^k}^{-20},
\end{align*}
where $\Psi_\theta (t,\xi) = e^{-iB_\theta(t,\xi)} \wh{\phi_\theta}(t,\xi)$ and $B_\theta(t,\xi)$ is the phase modification symbol as defined in Theorem \ref{mainthm}.
This bound shows the existence of limit $\lim_{t'\to \infty} \phi_{\theta}(t')$ in $L^2$.
By setting
$$
\psi_{\theta}^\infty(t):= e^{-\theta it \bra{D}} \lim_{t'\to \infty} \phi_{\theta}(t'),
$$
we obtain
\begin{align*}
	\|\psi_\theta (t) - e^{iB_\theta(t,D)} \psi_\theta^\infty (t)\|_{L^2} \les	\normo{\bra{\xi}^{20}|\xi|^\frac12 \cF \left[   \psi_\theta (t) - e^{iB_\theta(t,D)} \psi_\theta^\infty (t)\right]}_{L_\xi^\infty} \les \bra{t}^{-\de} \ve_0.
\end{align*}
Note that $\psi^\infty: =\Pi_+ \psi_+^\infty + \Pi_-\psi_-^\infty$ is a solution to linear Dirac equation.

On the other hand, for the proof of \eqref{eq:s-norm-m}, we will show that for $0 < t_1 \le t_2$,
\begin{align*}
	\sum_{j \in \cU_k} 2^j \normo{\qjk V_{\mu,\theta'}(t_2) - \qjk V_{\mu,\theta'}(t_2)}_{L^2} \les \ve_1 \bra{t_2}^{-\de} 2^{-k -  5H(2)\de k} \bra{2^k}^{-N(1)+5}.
\end{align*}
Then this implies the existence of limit $\lim A_{\mu, \theta'}$ in $\dot H^{\frac12 + 5H(2)\de}$.
By setting
$$
A_{\mu,\theta'}^\infty(t):= e^{\theta' it |D|} \lim_{t'\to \infty} V_{\mu,\theta'}(t'),
$$
we see that
\begin{align*}
	\normo{\bra{D}^{25} \left(  A_{\mu,\theta'} (t) -  A_{\mu,\theta'}^\infty(t) \right)}_{\doth^{\frac12 +5H(2)\de}} \les \bra{t}^{-\de} \ve_0.
\end{align*}
Utilizing \eqref{eq:decay-dirac} and \eqref{eq:decay-maxwell}, one can readily establish \eqref{mainthm:decay} with $\overline{\delta} := 11\delta$.
This completes the proof of Theorem \ref{mainthm}.

%%%%%%%%%%%%%%%%%%%%%%%%%%%%%%%%%%%%%%%%%%%%%%%%%%%%%%%%%%%%%%%%%%%%%%%%%%%%%%%%%%%%%%%%%%%%%%%%%%%%%%%%%%%%%%%%%%%%%%%%%%%%%%%%%%%%%%%%%%%%%%%%%%%%%%%%%%%%%%%%%%%%%%%%%%%%%%%%%%%%%%%%%%%%%%%%%%%%%%%%%%%%%%%%%%%%%%%%
\section{Proof of energy estimates}\label{sec:energy1}
This section is devoted to proving Proposition \ref{prop:energy}.

\subsection{Proof of \eqref{eq:energy-high-d}: the bound on $\psi_{\theta,\cL}$} We first prove the energy estimates for the Dirac part.
In order to show \eqref{eq:energy-high-d} let us define the energy functional
\begin{align*}
	\cE_\cL^{\bm D} (t) &:= \int_{\R^3} \bra{\cP^D \psit(t), \cP^D \psit(t)} - \bra{\bra{D}^{-\frac12}\cP^D \psit(t), A_\mu(t) \al^\mu \bra{D}^{-\frac12} \cP^D \psit(t)} dx,
\end{align*}
where $\cP^D := \bra{D}^{N(n)}\cL$ and $\cP^{D'} := \bra{D}^{N(n)}\cL'$.  Then, we have $\cE_\cL^{\bm D} \sim \|\psi_{\theta,\cL}\|_{H^{N(n)}}^2$ since $\ve_1 \ll 1$. Using \eqref{eq:commu-all}, we see that
\begin{align}\label{eq:deri-psi}
	\p_t \cL \psi_\tT %&= -\tT i \brad \cL \psit +	i\Pi_\theta \cL (A_\mu \al^\mu \psi) + \theta \left( \p_0^{|c|} \cM_1  + \cM_2\right)  (A_\mu \al^\mu \psi)  \\
	%& \hspace{0.5cm}  -\theta [\cG]^{c}\p_0^{|c|} \Pi_\theta(A_\mu \al^\mu \psi) + \theta  \sum_{j=1}^3c_j[\p]^a [\cO]^b R_j^{c_j}(A_\mu \al^\mu \psi)\\
	%
	%\theta i [\p]^a[\Om]^b[\cG]^c \p_0^{|c|}(A_\mu \al^\mu \psi)  + \theta i[\p]^a  [\cO]^b [\Gam]^c(A_\mu \al^\mu \psi)\\
	%	& \hspace{0.5cm}  -\theta i [\cG]^{c}\p_0^{|c|} \Pi_\theta(A_\mu \al^\mu \psi) + \theta  i \sum_{j=1}^3c_j[\p]^a [\cO]^b R_j^{c_j}(A_\mu \al^\mu \psi),\\
	= -\theta i \brad\cL \psit + i\Pi_\theta \cL (A_\mu \al^\mu \psi) + i\theta\sum_{\cL' \in \cV_n}C(\cL')\cR_{\cL'}\cL'(A_\mu \al^\mu \psi).
\end{align}
Then, this yields that
\begin{align*}%\label{eq:energy-deri-dirac}
	\begin{aligned}
		\frac d{dt}\cE_\cL^{\bm D} (t) &= 2 {\rm Im} \int_{\R^3} \left[\bra{\cP^D\psit,  \Pi_\theta \cP^D(A_\mu \al^\mu \psi) } + \theta \sum_{\cL' \in \cV_n} C(\cL')\bra{\cP^D\psit, \cR_{\cL'}\brad^{N(n)}\cL'(A_\mu \al^\mu\psi)}\right]\, dx \\
		&\qquad  - \wt \cE(t),
	\end{aligned}
\end{align*}
where
\begin{align*}
		\wt \cE(t) =&   2 {\rm Im} \int_{\R^3} \bra{\bra{D}^{-\frac12} \p_t \cP^D \psit(t), A_\mu(t) \al^\mu \bra{D}^{-\frac12} \cP^D \psit(t)} dx \\
	&\qquad  +  \int_{\R^3} \bra{ \bra{D}^{-\frac12}\cP^D \psit(t), \p_t A_\mu(t) \al^\mu \bra{D}^{-\frac12} \cP^D \psit(t)} dx.
\end{align*}

The following proposition finishes the energy estimates for Dirac part.
\begin{prop}
	Assume that $(\psi,A_\mu)$ satisfies a priori assumptions \eqref{assum:energy}--\eqref{assum:s-norm}. Let $t \in [0,T]$, $n \in \{0, 1, 2, 3\}$, $\cL \in \cV_n$, and $\theta \in \{+,-\}$. Then we have
	\begin{align}
		\left|2{\rm Im}\int_0^t\int_{\R^3} \bra{\cP^D\psit(t),  \Pi_\theta \cP^D(A_\mu(t) \al^\mu \psi(t)) } dx - \wt \cE(s) ds \right| &\les \ve_1^3\bra{t}^{2 H(n)\de},\label{eq:energy-high}\\
		\left|\int_0^t\int_{\R^3} \bra{\cP^D\psit(s), \cR_{\cL'}\brad^{N(n)}\cL'(A_\mu(s) \al^\mu\psi(s))} dx  ds\right|&\les \ve_1^3\bra{t}^{2 H(n)\de}.\label{eq:energy-high-1}
	\end{align}
\end{prop}

\begin{proof} Let us first consider \eqref{eq:energy-high}. Regarding $\wt \cE(s)$, by \eqref{eq:deri-psi} we see that
\begin{align*}
	\wt\cE(s) &= 2{\rm Im} \int_{\R^3}  \bra{ \bra{D}^{\frac12} \cP^D \psit(s), A_\mu(s) \al^\mu \bra{D}^{-\frac12} \cP^D \psit(s)}  dx\\
	 &\quad - 2{\rm Re}\int_{\R^3}  \bra{ \bra{D}^{-\frac12} \Pi_\theta \cP^D (A_\nu(s) \al^\nu \psi(s)) , A_\mu(s) \al^\mu \bra{D}^{-\frac12} \cP^D \psit(s)}  dx\\
	 &\quad - \theta  2{\rm Re} \sum_{\cL' \in \cV_{n'}}C(\cL') \int_{\R^3}  \bra{ \bra{D}^{N(n)-\frac12} \cR_{\cL'} \cL' (A_\nu(s) \al^\nu \psi(s)) , A_\mu(s) \al^\mu \bra{D}^{-\frac12} \cP^D \psit(s)}  dx\\
	 &\quad   +  \int_{\R^3} \bra{ \bra{D}^{-\frac12}\cP^D \psit(t), \p_t A_\mu(t) \al^\mu \bra{D}^{-\frac12} \cP^D \psit(t)} dx\\
	 & =: \sum_{j=1}^4 \wt\cE_j(s).
\end{align*}
Using Lemma \ref{lem:decay-nonlinear-dirac} and \eqref{eq:decay-maxwell}, we estimate
\begin{align*}
	\left|\int_0^t \wt\cE_2(s)+\wt\cE_3(s)ds \right| \les \ve_1^3 \bra{t}^{2H(n)\de}.
\end{align*}
Moreover, by \eqref{eq:zero-vec-3} we obtain that for $t \in [0,T]$,
\begin{align*}
\normo{\p_t A_\mu(t)}_{L^\infty} \les  \sum_{k\in\Z,\theta' \in \{\pm\}}2^{-\frac k2} \normo{\p_t P_k W_{\mu,\theta'}(t)}_{L^\infty} \les \ve_1 \bra{t}^{-1},
\end{align*}
which implies
\begin{align*}
 \left| \int_0^t \wt \cE_4(s) ds\right| \les \ve_1^2 \bra{t}^{2H(n)\de}.
\end{align*}
Thus, it suffices to consider the part
	\begin{align*}
			&\int_0^t\int_{\R^3} \bra{ \cP^D\psit(s), \Pi_\theta \cP^D(A_\mu(s) \al^\mu \psi(s)) } \,dx - \wt \cE_1(s) ds,
\end{align*}
which can be written as 
	\begin{align}\label{eq:e-dirac-1}
		 \int_0^t\int_{\R^3} \bra{\cP^D\psit(s), \Pi_\theta \cP^D(A_\mu(s) \al^\mu \psi(s))  -  \Pi_\theta \bra{D}^{\frac12}\left( A_\mu(s) \al^\mu \bra{D}^{-\frac12}\cP^D \psi(s)  \right)}\, dx ds.
	\end{align}
	%	\cI_2(t) &= {\rm Im} \int_0^t \int_{\R^3} \bra{ \bra{D}^{-\frac12}\cP^D \psit(s), \p_t A_\mu(s) \al^\mu \bra{D}^{-\frac12} \cP^D \psit(s)} dxds.\nonumber
		%		\cI_3(t) &= {\rm Im}\sum_{\theo \in \{\pm\}}\int_{\R^3} \bra{A_\mu \al^\mu \cP \psi_\theo  , \bra{D}^{N(n)} \p_0^{|c|}\cM_1\psit} dx,\label{eq:e-dirac-3}\\
		%		\cI_4(t) &= {\rm Im}\sum_{\theo \in \{\pm\}}\int_{\R^3} \bra{A_\mu \al^\mu \cP \psi_\theo  ,\bra{D}^{N(n)} \cM_2\psit} dx,\label{eq:e-dirac-4}

 By commutation relation one gets
	\begin{align*}
		\cP(A_\mu \al^\mu \psi)   = \bra{D}^{N(n)}( A_\mu \al^\mu \cL \psi)  + \sum_{\cL_\ell \in \cV_{n_\ell}, \ell=1,2} C(\cL_1,\cL_2)  \bra{D}^{N(n)}(\cL_2 A_\mu \al^\mu \cL_1 \psi) ,
	\end{align*}
	where $C(\cL_1, \cL_2)$ are constants depending only on $\cL_1, \cL_2$ and $n_1 \le n-1, n_1 + n_2 = n$. 
	
	We can write \eqref{eq:e-dirac-1} in the Fourier side as the sum of the followig:
	\begin{align}\label{eq:e-3-four}
		\begin{aligned}
			&(2\pi)^{-6}\int_0^t\!\int_{\R^{3+3} }  M_n^1(\xi,\eta)  \bra{\wh{ \psi_{\theta,\cL}}(s,\xi),    \wh{ W_{\mu,\thet}}(s,\eta) \al^\mu \wh{ \psi_{\theo,\cL}}(s,\xi - \eta)  }  d\xi d\eta ds,\\
			&(2\pi)^{-6}\int_0^t \! \int_{\R^{3+3}}  M_n^2(\xi,\eta)  \bra{\wh{\psi_{\theta,\cL}}(s,\xi),  \wh{  W_{\mu, \cL_2, \thet}}(s,\eta) \al^\mu \wh{ \psi_{\theo,\cL_1}}(s,\xi - \eta)  }  d\xi d\eta ds ,
		\end{aligned}
	\end{align}
	where
	\begin{align*}
		M_n^1(\xi,\eta) &:= |\eta|^{-\frac12} \Pi_\theta(\xi) \bra{\xi}^{N(n)+\frac12} \left(\bra{\xi}^{N(n)-\frac12} - \bra{\xi-\eta}^{N(n)-\frac12}\right),\\
		M_n^2(\xi,\eta) &:= |\eta|^{-\frac12}\Pi_\theta(\xi)  \bra{\xi}^{2N(n)}.
	\end{align*}

	We estimate \eqref{eq:e-3-four} by dividing the time $t \in [0,T]$ into a dyadic pieces $2^m$ $(m \in \{0, \cdots, L+ 1\})$ with
	\[
	|L-\log_2(t+2)|\le2
	\]
	and by taking cut-offs $q_m : \R \to [0,1]$ such that
	\begin{align*}
		&\supp(q_0) \subset [0,2], \quad \supp(q_{L+1}) \subset [t-2,t],  \quad \supp(q_m) \subset [2^{m-1},2^{m+1}],\\
		&\qquad \sum_{m=0}^{L+1}q_m(s) = \mathbf{1}_{[0,t]}(s), \quad q_m \in C^1(\R),\qquad\int_0^t |q_m'(s)| ds \les 1 .
	\end{align*}
	Let $I_m$ denote the support of $q_m$ and let
	\begin{align*}
		\cI_{m}^1 : = \int_{I_m}\!\int_{\R^{3+3} } q_m(s) M_n^1(\xi,\eta)  \bra{\wh{ \psi_{\theta,\cL}}(s,\xi),   \wh{ W_{\mu,  \thet}}(s,\eta) \al^\mu \wh{ \psi_{\theo,\cL}}(s,\xi-\eta)    }\,  d\xi d\eta ds,\\
		\cI_{ m}^2 : = \int_{I_m}\!\int_{\R^{3+3} } q_m(s) M_n^2(\xi,\eta)  \bra{\wh{ \psi_{\theta,\cL}}(s,\xi),   \wh{ W_{\mu, \cL_2, \thet}}(s,\eta) \al^\mu \wh{ \psi_{\theo,\cL_1}}(s,\xi-\eta)    }\,  d\xi d\eta ds.
	\end{align*}
	To get the bound \eqref{eq:energy-high}, it suffices to show that for each $l = 1, 2$ and $m \in \{0, \cdots, L+1\}$
	\begin{align}\label{eq:energy-high-j}
		|\cI_{m}^l| \les \ve_1^3 2^{2H(n)\de m}.
	\end{align}
	%for $\theta,\theo,\thet \in \{+,-\}$, and $j=1,2$.
	%Then, we focus on
	%	\begin{align}
		%	&\int_{I_m}\!\int_{\R^{3+3} }  M_n^1(\xi,\eta)  \bra{  \wh{ W_{\mu,\thet}}(s,\eta) \al^\mu \wh{ \psi_{\theo,\cL}}(s,\xi-\eta) ,   \wh{ \psi_{\theta,\cL}}(s,\xi) }  d\xi d\eta ds,\label{eq:e-3-four-1}\\
		%	&\int_{I_m} \! \int_{\R^{3+3}}  M_n^2(\xi,\eta)  \bra{  \wh{ W_{\mu,\thet,\cL_2}}(s,\eta) \al^\mu \wh{ \psi_{\theo,\cL_1}}(s,\xi-\eta) ,   \wh{\psi_{\theta,\cL}}(s,\xi) }  d\xi d\eta ds.\label{eq:e-3-four-2}
		%\end{align}

		\emph{Estimates for $\cI_{ m}^1$.} %We prove
		%\begin{align}\label{eq:claim-e-3-1}
		%	\left|\eqref{eq:e-3-four-1}\right| \les \ve_1^3 2^{2H(n)\de m}.
		%\end{align}
		To this end, we divide the frequencies of $\cI_{ m}^1$ into the dyadic pieces $2^k, 2^{k_1}, 2^{k_2} \in 2^\Z$ as follows:
		\begin{align*}%\label{eq:e-3-dec}
			\cI_{ m, \bm k}^1 :=	\int_{I_m}\int_{\R^{3+3}} q_m(s) M_n^1(\xi,\eta)  \bra{ \wh{P_k\psi_{\theta,\cL}}(s,\xi),  \wh{P_{k_2}W_{\mu,  \thet}}(s,\eta) \al^\mu \wh{P_{k_1}\psi_{\theo,\cL}}(s,\xi-\eta)    }  d\xi d\eta ds ,
		\end{align*}
		where $\bm k = (k,k_1,k_2)$.	For the multiplier estimates, we have
		\begin{align*}
			\normo{ M_n^1(\xi,\eta)\rho_k(\xi)\rho_{k_1}(\xi+\eta)\rho_{k_2}(\eta)}_{\cm} \les 2^{\frac{k_2}2}\bra{2^k}^{N(n)} \bra{2^{\max(\bm k)}}^{N(n)-1}.
		\end{align*}
		Thus by Lemma \ref{lem:coif-mey}, \eqref{eq:high}, and \eqref{eq:zero-vec-3}, we have
		\begin{align*}
			\left|\cI_{ m, \bm k}^1 \right| &\les \int_{I_m} 2^{\frac {k_2}2}\bra{2^{\max(\bm k)}}^{N(n)-1} \|P_{k}\psi_{\tT,\cL}(s)\|_{H^{N(n)}} \|P_{k_1}\psi_{\theo,\cL}(s)\|_{L^2}  \|P_{k_2}W_{\mu,  \thet}(s)\|_{L^\infty} ds\\
			&\les \ve_1^3|I_m| 2^{2H(n) \de m -m} \bra{2^{\max(\bm k)}}^{N(n)-1} \bra{2^{k_1}}^{-N(n)}  2^{k_2 - 5H(2)\de k_2 } \bra{2^{k_2}}^{-22},
		\end{align*}
		which implies \eqref{eq:energy-high-j} in case that $\min(\bm k) = k_2$.
		
		Let us move on to the case $\min(\bm k)= k$ or $k_1$. If $\min(\bm k) \le -\frac23 m -\de m$, we have 
			\begin{align}\label{eq:energy-high-2}
		\begin{aligned}
				&\left|\cI_{ m, \bm k}^1\right| \\
				&\les |I_m| 2^{\frac{3\min(\bm k)}2 +\frac{k_2}2} \bra{2^{\max(\bm k)}}^{N(n)-1} \|P_{k}\psi_{\tT,\cL}(s)\|_{H^{N(n)}} \|P_{k_1}\psi_{\theo,\cL}(s)\|_{L^2}  \|P_{k_2}W_{\mu,\thet}(s)\|_{L^2} \\
			&\les \ve_1^3 |I_m| 2^{[2H(n) +1]\de m } 2^{\frac{3\min(\bm k)+k_2}2} \bra{2^{k_1}}^{-N(n)} \bra{2^{\max(\bm k)}}^{N(n)-1-N(0)}.
		\end{aligned}
		\end{align}
		We consider the high frequency case $2^{\max(\bm k)} \ge 2^{\frac1{60} m}$. In this case the high frequency sum can be easily treated by the regularity condition $\bra{2^{k_2}}^{-N(0)}$ of $W_{\mu, \thet}$ in the energy estimates \eqref{eq:high}. To exploit this regularity gap, we estimate to distinguish $n \le 2$ and $n= 3$.  	Indeed,
		if $n \le 2$, it follows from \eqref{eq:decay-dirac} that
			\begin{align*}
			\left|\cI_{ m, \bm k}^1\right| &\les |I_m| 2^{\frac{k_2}2} \bra{2^{\max(\bm k)}}^{N(n)-1} \|P_{k}\psi_{\tT,\cL}(s)\|_{H^{N(n)}} \|P_{k_1}\psi_{\theo,\cL}(s)\|_{L^\infty}  \|P_{k_2}W_{\mu,\thet}(s)\|_{L^2} \\
			&\les \ve_1^3 |I_m| 2^{[H(n)+H(n+1) +1]\de m - m } 2^{\frac{k_2}2} \bra{2^{k_1}}^{-N(n+1)+3} \bra{2^{\max(\bm k)}}^{N(n)-1-N(0)}.
		\end{align*}
		On the other hand, if $n=3$, \eqref{eq:energy-high-2} yields that
		\begin{align*}
			\left|\cI_{ m, \bm k}^1\right| &\les \ve_1^3 |I_m| 2^{[2H(3) +1]\de m } 2^{\frac{3\min(\bm k)+k_2}2} \bra{2^{k_1}}^{-N(3)} \bra{2^{\max(\bm k)}}^{N(3)-1-N(0)}.
		\end{align*}
		whose $2^{\max(\bm k)} \ge 2^{\frac m{60} }$ leads us to \eqref{eq:energy-high-j}.
		
		Let us now consider $ 2^{\min(\bm k)} \ge  2^{- \frac23 m -\de m}$ and $2^{\max(\bm k) }\le 2^{\frac m{60} }$. For this, we use the normal form approach. Making integration by parts in time, $\cI_{ m, \bm k}^1$ becomes the sum of the integrals:
		\begin{align}
			&\int_{I_m}\! q_m'(s) \int_{\R^{3+3}} \frac{e^{is p_\Theta(\xi,\eta)}}{p_\Theta(\xi,\eta)}	  M_n^1(\xi,\eta)  \bra{\wh{P_k\phi_{\theta,\cL}}(s,\xi),  \wh{P_{k_2}V_{\mu,\thet}}(s,\eta) \al^\mu \wh{P_{k_1}\phi_{\theo,\cL}}(s,\xi-\eta) ,   }  d\xi d\eta ds, \label{eq:normal-dirac1}\\
			&\int_{I_m}\!\! q_m(s)\int_{\R^{3+3}} \frac{e^{is p_\Theta(\xi,\eta)}}{p_\Theta(\xi,\eta)}M_n^1(\xi,\eta)  \p_s \left[    \bra{ \wh{P_k\phi_{\theta,\cL}}(s,\xi),  \wh{P_{k_2}V_{\mu,\thet}}(s,\eta) \al^\mu \wh{P_{k_1}\phi_{\theo,\cL}}(s,\xi - \eta) ,   }  \right] d\eta d\xi ds,\label{eq:normal-dirac2}
		\end{align}
		where $p_\Theta$ is the phase defined in \eqref{eq:phase-dirac}.	By \eqref{eq:high} and \eqref{eq:zero-vec-3} we have
		\begin{align*}
			|\eqref{eq:normal-dirac1}|  &\les |I_m| 2^{-m} 2^{-\frac{k_2}2}  \bra{2^k}^{N(n)}\bra{2^{\max(\bm k)}}^{N(n)-1} \|\p_s P_k \psi_{\theta,\cL}(s)\|_{L^2}\| \pko \psi_{\theta,\cL}(s)\|_{L^2}  \| \pkt W_{\mu,\thet}(s)\|_{L^\infty} \\
			&\les \ve_1^3|I_m|2^{2H(n)\de m -2m} 2^{  -5H(2) \de k_2}  \bra{2^{k_1}}^{-N(n)} \bra{2^{\max(\bm k)}}^{N(n)-23}.
		\end{align*}
		This enables us to sum over  $ 2^{\min(\bm k)} \ge  2^{- \frac23 m -\de m}$ and $2^{\max(\bm k) }\le 2^{\frac m{60} }$. 		To estimate \eqref{eq:normal-dirac2}, we utilize
		\begin{align}\label{eq:energy-max-time}
			\begin{aligned}
					\normo{P_{k} \p_s\phi_{\theta,\cL}}_{L^2} &\les \ve_1 2^{-m+\wt{H}(n)\de m } \bra{2^{k}}^{-N(n+1)-5} ,\\
				\normo{P_{k} \p_sV_{\mu,\thet}}_{L^2} &\les \ve_1 2^{-m+\wt{H}(0)\de m } \bra{2^{k}}^{-N(1)-5} ,
			\end{aligned}
		\end{align}
		obtained by \eqref{eq:esti-non-dirac-time} and \eqref{eq:esti-nonl-max-time}, respectively.  In \eqref{eq:normal-dirac2}, if the time derivative falls on $\phi_{\theta,\cL}$ or $\phi_{\theo,\cL}$, we estimate
		\begin{align*}
			|\eqref{eq:normal-dirac2}| &\les |I_m| 2^{-\frac{k_2}2}  \bra{2^k}^{N(n)}\bra{2^{\max(\bm k)}}^{N(n)-1} \|\p_s P_k \psi_{\theta,\cL} (s)\|_{L^2}\| \pko \psi_{\theta,\cL} (s)\|_{L^2}  \| \pkt W_{\mu,\thet} (s)\|_{L^\infty} \\
			&\les \ve_1^3 |I_m| 2^{[H(n)+ \wt H(n)]\de m -2m} 2^{ -5H(2)\de k_2} \bra{2^k}^{N(n)-N(n+1)-5} \bra{2^{k_1}}^{-N(n)}\bra{2^{\max(\bm k)}}^{N(n)-23}.
		\end{align*}
		We consider the case that the derivative falls on $V_{\mu,\thet}$. In this case, we decompose $n \le 2$ and $n= 3$. If $n \le 2$,  by \eqref{eq:decay-dirac} and \eqref{eq:energy-max-time} we obtain
		\begin{align*}
				|\eqref{eq:normal-dirac2}| &\les |I_m| 2^{-\frac{k_2}2}  \bra{2^{\max(\bm k)}}^{N(n)-1} \|P_k \psi_{\theta,\cL}(s)\|_{H^{N(n)}}\| \pko \psi_{\theta,\cL}(s)\|_{L^\infty}  \|  \p_s \pkt V_{\mu,\thet}(s)\|_{L^2} \\
			&\les \ve_1^3 |I_m| 2^{[H(n) + H(n+1) + \wt H(0) ]\de m -2m} 2^{ \frac{k_1 -  k_2}2}  \bra{2^{k_1}}^{-N(n+1)+2}\bra{2^{\max(\bm k)}}^{N(n)-6-N(1)}.
		\end{align*}
		Let us move on to the case $n=3$. Using Hardy-Littlewood-Sobolev inequality and \eqref{eq:zero-vec-3}, we see that
		\begin{align*}
			&\|e^{\thet it |D|}\p_s \pkt V_{\mu,\thet}(s)\|_{L^\infty}\\
			 &\les  \| |D|^{-\frac12} \pkt \bra{\psi(s),\al_\mu \psi(s)} \|_{L^\infty} 	\\
			& \les  \sum_{\theta_3,\theta_4
				\in \{\pm\}} \sum_{\substack{j_\ell \in \cU_{k_\ell}\\ k_\ell \in \Z, \ell =3,4}} \|Q_{j_3,k_3}\psi_{\theta_3}(s)\|_{L^2}^\frac13\|Q_{j_3,k_3}\psi_{\theta_3}(s)\|_{L^\infty}^\frac23\|Q_{j_4,k_4}\psi_{\theth}(s)\|_{L^\infty} \\
			& \les \ve_1^2  2^{-\frac53m + 20\de m} \bra{2^{k_2}}^{-N(1)  + 3}.
		\end{align*}
		Then we have
			\begin{align*}
			|\eqref{eq:normal-dirac2}| &\les |I_m| 2^{-\frac{k_2}2}  \bra{2^{\max(\bm k)}}^{N(3)-1} \|P_k \psi_{\theta,\cL}(s)\|_{H^{N(3)}}\| \pko \psi_{\theta,\cL}(s)\|_{L^2}  \|  e^{\thet it|D|}\p_s \pkt V_{\mu,\thet}(s)\|_{L^\infty} \\
			&\les \ve_1^3 |I_m| 2^{[2H(3) + 20]\de m -\frac 53m} 2^{ -\frac{ k_2}2}  \bra{2^{k_1}}^{-N(3)}\bra{2^{\max(\bm k)}}^{N(3)-N(1) +2}.
		\end{align*}
		This finishes the proof for \eqref{eq:normal-dirac2}.
		Summing over $ 2^{\min(\bm k)} \ge  2^{-\frac23 m-\de m}$ and $2^{\max(\bm k)} \le 2^{\frac m{60} }$, we get \eqref{eq:energy-high-j}.

		\begin{rem}
			Regarding the restriction $2^{\max(\bm k)} \le 2^{\frac m{60} }$, we need this condition only when $n=0,1$ and $\min(\bm k) = k_1$ in the estimates for \eqref{eq:normal-dirac1}. In view of \eqref{eq:energy-max-time}, we also require this restriction in the estimates for \eqref{eq:normal-dirac2} when $n\le 2$.
		\end{rem}

		\emph{Estimates for $\cI_{ m}^2$.} We postpone a proof for the case $n_1 = 0$ to Section \ref{sec:bulk-dirac}. Here, we consider the cases $(n_1, n_2) = (2,1)$, and $(1,2)$ when $n=3$ and $(n_1,n_2) = (1,1)$ when $n = 2$. Similarly to the previous estimates, we separate $|\xi|,|\xi+\eta|, |\eta|$ by dyadic numbers $2^{k},2^{k_1},2^{k_2}$ and consider the integral
		\begin{align}\label{eq:e-3-dec-2}
			\int_{I_m}\iint_{\R^{3+3} } q_m(s) M_n^2(\xi,\eta)  \bra{\wh{P_k\psi_{\theta,\cL}}(s,\xi),   \wh{P_{k_2}W_{\mu, \cL_2, \thet}}(s,\eta) \al^\mu \wh{P_{k_1}\psi_{\theo,\cL_1}}(s,\xi - \eta)  }  d\xi d\eta ds .
		\end{align}
		
		If $2^{\min(\bm k) } \le 2^{-m}$ or $2^{\max (\bm k)} \le 2^{\frac m{10}}$, by H\"older's inequality and \eqref{eq:high}, we have
		\begin{align*}%\label{eq:e-3-dec-2-esti}
			\begin{aligned}
				|\eqref{eq:e-3-dec-2}| &\les |I_m| 2^{\min(\bm k)} \bra{2^k}^{2N(n)} \|P_k\psi_{\theta,\cL}(s)\|_{L^2}\|P_{k_1}\psi_{\theo,\cL_1}(s)\|_{L^2}\|P_{k_2}W_{\mu, \cL_2, \thet}(s)\|_{L^2} \\
				&\les \ve_1^3 2^{\left(H(n) + H(n_1) +H(n_2) \right) \de m}\bra{2^k}^{N(n)} \bra{2^{k_1}}^{-N(n_1)} \bra{2^{k_2}}^{-N(n_2)}.
			\end{aligned}
		\end{align*}
		Since $1 \le n_1, n_2 \le n-1$, one gets $H(n_1) + H(n_2) \le H(n) - 190$ and $N(n) \le \min \left(N(n_1), N(n_2) \right)- 10$. Summing over $2^{\min(\bm k)} \le 2^{-m}$ or $2^{\max (\bm k)} \le 2^{\frac m{10}}$, one can get \eqref{eq:energy-high-j}.%These make \eqref{eq:e-3-dec-2-esti} to close \eqref{eq:e-3-four-2} over $\bm k$-sum with  .

		If $2^{\min(\bm k) } \ge 2^{- m} $ and $2^{\max (\bm k)} \le 2^{\frac m{10}}$, then we exploit the space resonance \eqref{eq:nonresonance-space}. More precisely, making integration by parts in $\eta$, \eqref{eq:e-3-dec-2} becomes the sum of the following:
		\begin{align}
			&\int_{I_m}\int_{\R^{3+3}} s^{-1}q_m(s) e^{isp_\Theta(\xi,\eta)} \nabla_\eta 	\wt{M_{n, \bm k}^2} \bra{\wh{P_k\phi_{\theta,\cL}}(s,\xi),   \wh{P_{k_2}V_{\mu, \cL_2, \thet}}(s,\eta) \al^\mu \wh{P_{k_1}\phi_{\theo,\cL_1}}(s,\xi - \eta) }  d\xi d\eta ds ,\label{eq:e-space-1}\\
			&\int_{I_m}\int_{\R^{3+3}} s^{-1} q_m(s) e^{isp_\Theta(\xi,\eta)} 	\wt{M_{n, \bm k}^2} \bra{ \wh{P_k\phi_{\theta,\cL}}(s,\xi),  \wh{P_{k_2}x V_{\mu, \cL_2, \thet}}(s,\eta) \al^\mu \wh{P_{k_1}\phi_{\theo,\cL_1}}(s,\xi - \eta) }  d\xi d\eta ds ,\label{eq:e-space-2}\\
			&\int_{I_m}\int_{\R^{3+3}} s^{-1} q_m(s) e^{isp_\Theta(\xi,\eta)} 	\wt{M_{n, \bm k}^2}  \bra{\wh{P_k\phi_{\theta,\cL}}(s,\xi),  \wh{P_{k_2}V_{\mu, \cL_2, \thet}}(s,\eta) \al^\mu \wh{P_{k_1} x \phi_{\theo,\cL_1}}(s,\xi - \eta)   }  d\xi d\eta ds ,\label{eq:e-space-3}
		\end{align}
		where the multiplier $\wt{M_n^2}$ is defined by
		\begin{align*}
			\wt{M_{n, \bm k}^2}(\xi,\eta) = \frac{\nabla_\eta p_\Theta(\xi,\eta) M_n^2(\xi,\eta) }{|\nabla_\eta p_\Theta(\xi,\eta)|^2}\rho_k(\xi)\rho_{k_1}(\xi+\eta)\rho_{k_2}(\eta).
		\end{align*}
		In view of \eqref{eq:nonresonance-space}, this multiplier satisfies
		\begin{align*}
			\left|\nabla_\eta^\ell\wt{M_{n, \bm k}^2}(\xi,\eta)  \right|\les 2^{-\frac{k_2}2 - \ell k_2}\bra{2^k}^{2N(n)}\bra{2^{k_1}}^{2 + 2\ell}
		\end{align*}
		for $\ell = 0,1$. Then, \eqref{eq:high} leads us to the bound
		\begin{align*}
			|\eqref{eq:e-space-1}| \les \ve_1^3 2^{\left(H(n) + H(n_1) +H(n_2) \right) \de m}2^{\frac{3\min(\bm k)}{2}-\frac{3 k_2}2}  \bra{2^k}^{N(n)}\bra{2^{k_1}}^{-N(n_1)+4} \bra{2^{k_2}}^{-N(n_2)}.
		\end{align*}
		The sum of \eqref{eq:e-space-1} over $2^{\min(\bm k)} \ge 2^{-m}$ now gives us \eqref{eq:energy-high-j}.
		
		As for \eqref{eq:e-space-2}, by using \eqref{eq:weight} and \eqref{eq:decay-dirac}, we estimate as follows:
		\begin{align}\label{eq:e-space-esti-1}
			|\eqref{eq:e-space-2}| \les \ve_1^3 2^{\left(H(n) + H(n_1) +H(n_2+1) \right) \de m} 2^{\frac{3\min(\bm k)}2-k_2}  \bra{2^{\max(\bm k)}}^{N(n)}\bra{2^{k_1}}^{-N(n_1)+2} \bra{2^{k_2}}^{-N(n_2+1)}
		\end{align}
		and
		\begin{align}\label{eq:e-space-esti-2}
			|\eqref{eq:e-space-2}| \les \ve_1^3 2^{\left(H(n) + H(n_1+1) +H(n_2+1) \right) \de m} 2^{-m}2^{-\frac{k_2}2}  \bra{2^{\max(\bm k)}}^{N(n)}\bra{2^{k_1}}^{-N(n_1+1)+8} \bra{2^{k_2}}^{-N(n_2+1)}.
		\end{align}
		Since $ H(n_1) +H(n_2+1) \le H(n) +10$ and
		\begin{align*}
			2^{-k_2}\bra{2^{\max(\bm k)}}^{N(n)}\bra{2^{k_1}}^{-N(n_1)+2} \bra{2^{k_2}}^{-N(n_2+1)} \les \bra{2^{\min(\bm k)}}^{-10}\bra{2^{\max(\bm k)}}^{-1},
		\end{align*}
		one gets \eqref{eq:energy-high-j} by summing \eqref{eq:e-space-esti-1} over $2^{\max(\bm k)} \ge 2^{10\de m}$ and \eqref{eq:e-space-esti-2} over $2^{\max(\bm k)} \le 2^{10\de m}$.
		The estimates for \eqref{eq:e-space-3} can be obtained similarly with $L^2 \times L^2 \times L^\infty$ estimates due to the restriction $2^{\max (\bm k)} \le 2^{\frac m{10}}$.	
		
		We prove \eqref{eq:energy-high-1}. By \eqref{eq:commu-decom} we have
		\begin{align*}
			\sum_{\cL' \in \cV_{n'}} C(\cL') \cR_{\cL'} \cL' \p_t (A_\mu\al^\mu \psi) =   C_1 (\cL_0)   \cL_0 (A_\mu\al^\mu \psi) +  \sum_{\ol \cL \in \ol{\cV}_{\ol  n}} C_2( \ol \cL) \cR_{\ol \cL} \ol \cL \p_t (A_\mu\al^\mu \psi),
		\end{align*}
		for $\ol n, \wt n \le n-1$. By \eqref{eq:dirac-pt} and \eqref{eq:maxwell-pt}, it suffices to consider
		\begin{align}
&\int_0^t \! \int_{\R^{3+3}}  M_n^3(\xi,\eta)  \bra{\wh{\psi_{\theta,\cL}}(s,\xi),  \wh{  W_{\mu, \cL_2, \thet}}(s,\eta) \al^\mu \wh{ \psi_{\theo,\cL_1}}(s,\xi - \eta)  }  d\xi d\eta ds ,\label{eq:commu-1}\\
&\int_0^t \! \int_{\R^{3+3}}  M_n^4(\xi,\eta)  \bra{\wh{\psi_{\theta,\cL}}(s,\xi),  \wh{  W_{\mu, \cL_2, \thet}}(s,\eta) \al^\mu \wh{ \mathfrak N^{\bm D}_{\theo,\cL_1}}(s,\xi - \eta)  }  d\xi d\eta ds ,\label{eq:commu-2}\\
&\int_0^t \! \int_{\R^{3+3}}  M_n^4(\xi,\eta)  \bra{\wh{\psi_{\theta,\cL}}(s,\xi),  \wh{  \mathfrak N_{\mu, \cL_2}^{\bm M}}(s,\eta) \al^\mu \wh{ \psi_{\theo,\cL_1}}(s,\xi - \eta)  }  d\xi d\eta ds ,\label{eq:commu-3}
		\end{align}
	where
	\begin{align*}
		M_n^3(\xi,\eta) = |\eta|^{-\frac12} \left(1 + \cR_{\ol \cL}(\xi) \right) \bra{\xi}^{2N(n)} (1+|\eta| + \bra{\xi-\eta}) \;\;\mbox{ and }\;\; M_n^4(\xi,\eta) = |\eta|^{-\frac12} \cR_{\ol \cL}(\xi) \bra{\xi}^{2N(n)}.
	\end{align*}
		  Since  $R_{\ol \cL}$ is smooth and plays a similar role to $\Pi_\theta$,  the proof of \eqref{eq:commu-1} can be done similarly to the estimates for $\cI_m^2$ by replacing $\Pi_\theta, \cL$ with $\cR_{\ol \cL}, \ol \cL$, respectively. Indeed, the estimates for \eqref{eq:energy-high-1} can be obtained more easily from the fact that $n_0, \ol n \le  n-1$. For the proofs of \eqref{eq:commu-2} and \eqref{eq:commu-3}, Lemmas \ref{lem:esti-nonlinear-dirac} and \ref{lem:esti-nonlinear-wave} imply  the desired bounds straightforwardly. Therefore, we omit the proof of \eqref{eq:energy-high-1}.		
	\end{proof}

	\subsection{Proof of \eqref{eq:energy-high-m}: the bound on $W_{\mu, \cl, \theta'}$}

	With $\cP^M:= \bra{D}^{N(n)}|D|^\frac12 \cl$, we define the energy functional for Maxwell part
	\begin{align}\label{eq:energy-maxwell}
		\mathcal E_\cL^{\bm M} (t) := \frac 14\int_{\R^3}  \left[   (\cP^M A_\mu(t))^2 +  (\p_t  |D|^{-1}\cP^M A_{\mu}(t))^2   \right] dx.
	\end{align}
	Then, by the definition of $W_{\mu, \cL, \theta'}$, we have
	\begin{align*}
		\cE_\cL^{\bm M} (t) \sim \|W_{\mu, \cL, \theta'}(t)\|_{H^{N(n)}}^2.
	\end{align*}
	Using the commutation relation of wave operator, we see that
	\begin{align*}
		%\begin{aligned}\label{eq:time-deri-energy-maxwell}
		\frac{d}{dt}\mathcal E_\cL^{\bm M} (t) &=  \frac12 \int_{\R^3} |D|^{-1} \cP^M [\bra{\psi,\al_\mu \psi}](t) \p_t  |D|^{-1}\cP^M A_{\mu}(t) dx\\
		&= i(2\pi)^{-3}\sum_{\theta' \in \{\pm\}} \theta' \int_{\R^3} |\xi|^{-\frac12}\bra{\xi}^{2N(n)}\mathcal F[\cl \bra{\psi,\al_\mu \psi}](t,\xi) \overline{\cF \left[ W_{\mu,\cL, \theta'}\right]}(t,\xi) d\xi.
		%\end{aligned}
	\end{align*}
	To prove \eqref{eq:energy-maxwell}, it suffices to show
	\begin{align*}
		\left| \int_{I_m}\int_{\R^3} q_m(s) |\xi|^{-\frac12} \bra{\xi}^{2N(n)} \mathcal F\bra{ \psi_{\theta_1, \cL_1}, \al_\mu  \psi_{\theta_2, \cL_2}}(t,\xi) \overline{\wh{ W_{\mu, \cL, \theta'}}}(t,\xi)   d\xi ds \right|\les \ve_1^3 2^{2H(n)\de m},
	\end{align*}
	for any $t \in [0,T], m \in \{0,\cdots, L+1\}, \cl_1 \in \mathcal V_{n_1}, \cl_2 \in \mathcal V_{n_2} (n_1+n_2 = n)$, and $\theta',\theta_1,\theta_2 \in \{+,-\}$. %In particular, we denoted $\Pi_{\theta'j}\left(\cl_j \psi\right)$ by $\psi_{\theta'j, \cl_j}$.
	Then, we focus on the bound
	\begin{align}
		\begin{aligned}\label{eq:aim-energy}
			&\left| \int_{I_m}\int_{\R^{3+3}} q_m(s) |\xi|^{-\frac12} \bra{\xi}^{2N(n)}  \bra{\wh{ \psi_{\theta_1, \cl_1}}(s,\eta), \al_\mu \wh{ \psi_{\theta_2,\cl_2}}(s,\xi+\eta)} \overline{\wh{ W_{\mu, \cL, \theta'}}}(s,\xi)   d\eta d\xi ds \right|\\
			&\les \ve_1^3 2^{2H(n)\de m}.
		\end{aligned}
	\end{align}
	
	By the frequency decomposition, we define
	\begin{align*}
		\cJ_{\mu, m,\bm k} &:= \int_{I_m} q_m(s) M(\xi)\int_{\R^{3+3}}    \bra{\wh{ P_{k_2}\psi_{\theta_2,\cl_2}}(s,\eta), \al_\mu \wh{  P_{k_1}\psi_{\theta_1,\cl_1}}(s, \xi+\eta)} \overline{\wh{ P_{k} W_{\mu, \cl, \theta'}}}(s,\xi)  \, d\eta d\xi ds,
	\end{align*}
	where $M(\xi) = |\xi|^{-\frac12}\bra{\xi}^{2N(n)}$ and $\bm k =(k,k_1,k_2)$. To prove \eqref{eq:aim-energy}, it suffices to handle
	\begin{align}\label{eq:aim-energy-decom}
		\sum_{\bm k \in \Z^3} |\cJ_{\mu,m,\,\bm k}| \les \ve_1^3 2^{2H(n)\de m}
	\end{align}
	for  $n_1, n_2 \le n \le 3$, $\cL_\ell \in \mathcal V_{n_\ell}$. To observe the resonances, we recall the phase interaction:
	\begin{align*}
		\int_{I_m} q_m(s)\int_{\R^{3+3}} M(\xi)e^{is q_{\Theta}(\xi,\eta)}  \bra{\wh{ P_{k_2}\phi_{\theta_2,\cl_2}}(s,\eta), \al_\mu \wh{  P_{k_1}\phi_{\theta_1,\cl_1}}(s, \xi+\eta)}\overline{\wh{P_k V_{\mu, \cl, \theta'}}(s,\xi)} d\eta d\xi ds,
	\end{align*}
	where $q_\Theta(\xi,\eta)$ is as in \eqref{eq:phase-maxwell}.

	Without loss of generality, we may assume that $n_1 \le n_2$. The proof of \eqref{eq:aim-energy-decom} is divided into the cases $n_1=0$ and $n_1 \ge 1$. The case $n_1 = 0 (n \ge 1)$ will be treated in Section \ref{sec:bulk-dirac}.  	Let us first consider the vector field free case $n=0$. From the symmetry between spinors, we assume that $k_1 \le k_2$. Then the bounds \eqref{eq:high} and \eqref{eq:zero-vec-1} yield that
	\begin{align*}
	|\cJ_{\mu,m,\bm k} | &\les |I_m| 2^{\frac{3\min(\bm k) - k}2} \bra{2^k}^{N(0)} \| P_k W_{\mu,\theta'}(s)\|_{H^{N(0)}} \|\pko\psi_{\theo}(s)\|_{L^2}\|\pkt\psi_{\theo}(s)\|_{L^2}\\
	 &\les \ve_1^3 |I_m|  2^{2H(0)\de m}2^{\frac{3\min(\bm k) - k}2 + (1+\frac1{100})k_1 } \bra{2^k}^{N(0)} \bra{2^{k_1}}^{-N(0)+2} \bra{2^{k_2}}^{-N(0)}.
	\end{align*}
	This implies \eqref{eq:aim-energy-decom} for $2^{k+k_1} \le 2^{-m}$.  Let us consider $2^{k+k_1} \ge 2^{-m}$.  When $2^{k} \sim 2^{k_2}$, the estimate above yields the desired bound. Thus, we handle the case $2^k \ll 2^{k_1} \sim 2^{k_2}$.  Using the decomposition and their bounds \eqref{eq:decay-decomposition}, we see that
		\begin{align*}
		|\cJ_{\mu,m,\bm k} | &\les |I_m| 2^{-\frac k2} \bra{2^k}^{N(0)} \| P_k W_{\mu,\theta'}(s)\|_{H^{N(0)}} \|e^{-\theo it \bra{D}}\phi_{\theo}^{\le J, k_1}(s)\|_{L^\infty}\|\pkt\psi_{\theo}(s)\|_{L^2}\\ 
		&\les \ve_1^3 |I_m|  2^{H(0)\de m -\frac32 m}2^{-\frac{k+k_1}2+\frac{k_1}{1000} + (1+\frac1{100})k_2 } \bra{2^k}^{N(0)} \bra{2^{k_1}}^{-19} \bra{2^{k_2}}^{-N(0)+2}
	\end{align*}
and 
		\begin{align*}
	|\cJ_{\mu,m,\bm k} | &\les |I_m| 2^{-\frac k2} \bra{2^k}^{N(0)} \| P_k W_{\mu,\theta'}(s)\|_{H^{N(0)}} \|\phi_{\theo}^{> J, k_1}(s)\|_{L^2}\|\pkt\psi_{\theo}(s)\|_{L^\infty}\\ 
	&\les \ve_1^3 |I_m|  2^{[H(0)+2H(1)]\de m - 2m}2^{-\frac{k}2 - k_1 + \frac{k_2}2 } \bra{2^k}^{N(0)} \bra{2^{k_1}}^{-N(1)} \bra{2^{k_2}}^{-N(1)+2}.
\end{align*}
These estimates give us \eqref{eq:aim-energy-decom}.

	Here we focus on the case $n_1 \ge 1$, in which at least one vector field falls on each spinor.	We first consider the cases $2^{\min(\bm k) } \le 2^{-m}$ and $2^{\min(\bm k)} \ge 2^{-m}$. When  $2^{\min(\bm k)} \le 2^{-m}$. By \eqref{eq:high} we obtain the bounds
	\begin{align}\label{eq:energy-maxwell-high}
		\begin{aligned}
			\normo{P_k V_{\mu,\cL, \theta'}(s)}_{L^2} &\les \ve_1 2^{H(n)\de m } \bra{2^k}^{-N(n)},\\
			\normo{P_{k_1} \phi_{\theta_1,\cL_1}(s)}_{L^2} &\les \ve_1 2^{H(n_1)\de m } \bra{2^{k_1}}^{-N(n_1)},\\
			\normo{P_{k_2} \phi_{\theta_2,\cL_2}(s)}_{L^2} &\les \ve_1 2^{H(n_2)\de m } \bra{2^{k_2}}^{-N(n_2)}.
		\end{aligned}
	\end{align}
	By H\"older's inequality and \eqref{eq:energy-maxwell-high} we have
	\begin{align}\begin{aligned}\label{eq:proof-energy-maxwell}
			&	2^{-\frac k2}\bra{2^k}^{2N(n)} |\cJ_{m,\,\bm k}|\\
			& \les \ve_1^3 |I_m| 2^{[H(n)+H(n_1)+H(n_2)]\de m} 2^{-\frac k2} 2^{ \frac32 \min(\bm k)} \bra{2^k}^{N(n)}\bra{2^{k_1}}^{-N(n_1)}\bra{2^{k_2}}^{-N(n_2)}.
	\end{aligned}\end{align}
	Since   $H(n_1) + H(n_2) = H(n) -190 $ and $N(n) \le N(n_2) -10$, this directly implies \eqref{eq:aim-energy-decom}. Let us now handle the case $2^{\min(\bm k)} \ge 2^{-m}$. By integration by parts in time, $\cJ_{\mu, m,\bm k}$ is written as the sum of the following
	\begin{align}
		&\int_{I_m}\! q_m'(s) \int_{\R^{3+3}} \frac{e^{is q_{\Theta'}(\xi,\eta)}}{q_{\Theta'}(\xi,\eta)}  \bra{\wh{ P_{k_2}\phi_{\theta_2,\cl_2}}(s,\eta), \al_\mu \wh{  P_{k_1}\phi_{\theta_1,\cl_1}}(s,\xi+ \eta)} \overline{ \wh{ P_{k} V_{\mu, \cl, \theta'}}}(s,\xi)   d\eta d\xi ds, \label{eq:normal-maxwell1}\\
		&\int_{I_m}\!\! q_m(s)\int_{\R^{3+3}} \frac{e^{is q_{\Theta'}(\xi,\eta)}}{q_{\Theta'}(\xi,\eta)}  \p_s \left[ \bra{\wh{ P_{k_2}\phi_{\theta_2,\cl_2}}(s,\eta), \al_\mu \wh{  P_{k_1}\phi_{\theta_1,\cl_1}}(s,\xi+ \eta)} \overline{ \wh{ P_{k} V_{\mu,\theta',\cl}}}(s,\xi)  \right] d\eta d\xi ds.\label{eq:normal-maxwell2}
	\end{align}
	%with
	%\begin{align*}
	%	\mathcal Q_\mu^*(f,g,h):= \int_{\R^{3+3}} \frac{e^{is q_K(\xi,\eta)}}{q_K(\xi,\eta)}  \bra{\wh{h}(\eta), \al^\mu\wh{g}(\xi+\eta) }\overline{\wh{f}(\xi)} d\eta d\xi.
	%\end{align*}
	Using H\"older's inequality, we see that
	\begin{align*}
		|\eqref{eq:normal-maxwell1}| \les 2^{\frac32 \min(\bm k) - k } \bra{2^{\max(\bm k)}}^2  \|P_k V_{\mu,\theta',\cL}(s)\|_{L^2}\|P_{k_1}\phi_{\theo,\cL_1}(s)\|_{L^2}\|P_{k_2}\phi_{\thet,\cL_1}(s)\|_{L^2}.
	\end{align*}
	Analogously to \eqref{eq:proof-energy-maxwell}, we estimate
	\begin{align*}
		\sum_{2^{\min(\bm k)} \ge 2^{-m}} 2^{-\frac k2}\bra{2^k}^{2N(n)} |\eqref{eq:normal-maxwell1}| \les \ve_1^32^{2H(n)\de m}.
	\end{align*}
	Similar estimate to the above can be obtained for \eqref{eq:normal-maxwell2} by \eqref{eq:energy-max-time}. Since $n_1 \ge 1$, we have
	\begin{align*}
		\wt{ H} (n) + H(n_1) + H(n_2) \le 2H(n) -30,\\
		H (n) + \wt{H}(n_1) + H(n_2) \le 2H(n) -30,\\
		H (n) + H(n_1) + \wt H(n_2) \le 2H(n) -30.
	\end{align*}
	From this it follows that
	\begin{align*}
		2^{-\frac k2}\bra{2^k}^{2N(n)} |\eqref{eq:normal-maxwell2}| \les \ve_1^3 2^{2H(n) \de m - 30 \de m } \max\left(\bra{2^{k_1}}^{-1}, \bra{2^{k_2}}^{-1}\right).
	\end{align*}
	This completes the proof of \eqref{eq:aim-energy-decom} for $n_1 \ge 1$.

	\subsection{Bulk estimate on Dirac part}\label{sec:bulk-dirac}
	In this section, we finish the proof of the energy estimates of Dirac part for the cases that all vector fields fall on only one of the profiles.  Let us invoke \eqref{eq:e-3-dec-2} with $n_2 = n (n \ge 1)$ and write
	$$
	\cK_{m, \bm k} := \int_{I_m} q_m(s)\iint_{\R^3 \times \R^3}  M_n^2(\xi,\eta)  \bra{\wh{P_k\psi_{\theta,\cL}}(s,\xi),   \wh{P_{k_2}W_{\mu, \cL, \thet,}}(s,\eta) \al^\mu \wh{P_{k_1}\psi_{\theo}}(s,\xi+\eta)  }  d\xi d\eta ds.
	$$
	Note that $M_n^2(\xi,\eta) = |\eta|^{-\frac12}\bra{\xi}^{2N(n)}$. The following proposition gives us the proof of energy estimate on Dirac part.
	\begin{prop}
		Assume that $\cL \in \cV_n$ for $1 \le n \le 3$ and $t \in [0,T]$. Let $\bm{k}=(k_0,k_1,k_2)$ and $m \in \{0 ,\cdots, L+1\}$. Then
		we have
		\begin{align}\label{eq:aim-bulk}
			\sum_{\bm{k} \in \Z^3 }  |\cK_{m,\bm{k}}| \les \ve_1^3 2^{2 H(n) \de m}.
		\end{align}
	%	where $e(0)=0$ and $e(n) =1$ otherwise.
	\end{prop}
	
	\begin{proof}
		We prove this proposition by decomposing the case into four cases.

		\textbf{Case 1: $2^{\min(\bm k)} \le 2^{-m}$.} 	We begin with $L^\infty \times L^2 \times L^2$ estimate. By H\"older's inequality with \eqref{eq:high} and  \eqref{eq:zero-vec-1}, we have
		\begin{align}\label{eq:bulk-l2}
			|\cK_{ m,\bm{k}} | \les \ve_1^3  |I_m| 2^{2H(n) \de m } 2^{\frac{3\min(\bm{k})-k_2}2}   2^{(1 + \de_1) k_1}  \bra{2^k}^{N(n)}   \bra{2^{k_1}}^{-N(0) + 2}\bra{2^{k_2}}^{-N(n) }.
		\end{align}
		This estimate implies \eqref{eq:aim-bulk} if  $\min(\bm k) \le -m$.

		\textbf{Case 2: $2^{ \min(\bm k)} \ge 2^{-m}$ and $2^k \le 2^{-\frac35m}$.} We claim that
		\begin{align}\label{eq:bulk-case2}
			\sum_{\rm Case \, 2} |\cK_{m,\bm{k}}| \les \ve_1^3 2^{2 H(n) \de m}.
		\end{align}
		When $2^k \le 2^{-\frac35 m}$ and $2^{k_1} \le 2^{-\frac25 m}$, \eqref{eq:bulk-case2} follows from \eqref{eq:bulk-l2}. When
		$$
		2^{k} \le 2^{-\frac35 m} \;\;\mbox{ and }\;\; 2^{\min(\bm k) }2^{ (1+\de_1) k_1} \bra{2^{k_1}}^{- 35} \le 2^{-m},
		$$
		\eqref{eq:bulk-case2} can be shown treated similarly. Then, it remains to prove \eqref{eq:bulk-case2} when
		\begin{align}\label{eq:bulk-support}
			2^{-m} \le 2^{k} \le 2^{-\frac35 m} \;\;\mbox{ and }\;\; 2^{\min(\bm k) }2^{ (1+\de_1) k_1} \bra{2^{k_1}}^{- 35} \ge 2^{-m}.
		\end{align}
		From this relation, we deduce $2^{k_1} \ge 2^{-\frac25m}$. Setting $2^J:= C\bra{s}2^{k_1} $ for some constant $0 < C \ll 1$, we decompose $\pko \phi_{\theo}= \phi_{\theo}^{ \le J,k_1}+ \phi_{\theo}^{> J,k_1}$ as follows: $\cK_{m, \bm k}$ is the sum of the integrals: 	
		\begin{align}
			\begin{aligned}\label{eq:bulk-2-1}
				&\int_{I_m} q_m(s)  \int_{\R^{3+3}}  M_n^2(\xi,\eta)e^{isp_{\Theta}(\xi,\eta)} \\
				&\hspace{3cm} \times  \bra{\wh{P_k\phi_{\theta,\cL}}(s,\xi),   \wh{P_{k_2}V_{\mu, \cL, \thet,}}(s,\eta) \al^\mu \wh{\phi_{\theo}^{\le J, k_1}}(s,\xi+\eta)  }  d\xi d\eta  ds,
			\end{aligned}
		\end{align}
		\begin{align}
			\begin{aligned}\label{eq:bulk-2-2}
				& \int_{I_m} q_m(s)  \int_{\R^{3+3}}  M_n^2(\xi,\eta)e^{isp_{\Theta}(\xi,\eta)} \\
				&\hspace{3cm} \times \bra{\wh{P_k\phi_{\theta,\cL}}(s,\xi), \wh{P_{k_2}V_{\mu, \cL, \thet,}}(s,\eta) \al^\mu \wh{\phi_{\theo}^{> J, k_1}}(s,\xi+\eta)  }  d\xi d\eta  ds,
			\end{aligned}
		\end{align}	
		where $p_\Theta(\xi, \eta)$ is in \eqref{eq:phase-dirac}. By \eqref{eq:zero-vec-2} and \eqref{eq:elliptic-q}, we obtain
		\begin{align*}
			\|e^{-\theo it \brad} \phi_{\theo}^{ \le J,k_1}(t)\|_{L^\infty} &\les \ve_1 \bra{t}^{-\frac32} 2^{-(\frac12 - \frac1{1000})k_1}  \bra{2^{k_1}}^{-19},\\
			\| \phi_{\theo}^{ > J,k_1}(t)\|_{L^2} &\les \ve_1 \bra{t}^{-1+H(1)\de} 2^{-k_1}\bra{2^{k_1}}^{-N(1)-1}.
		\end{align*}
		This leads us to the bound
		\begin{align*}
			|\eqref{eq:bulk-2-1}| &\les 2^m 2^{-\frac {k_2}2}\bra{2^k}^{N(n)} \sup_{s\in I_m}   \|\pk \psi_{\theta, \cL}(s)\|_{L^2}\|e^{-\theo it \brad}  \phi_{\theo}^{ \le J,k_1}(s)\|_{L^\infty} \|\pkt W_{\mu, \cL, \thet}(s)\|_{L^2}\\
			&\les \ve_1^3 2^{2H(n)\de m -\frac m2} 2^{-\frac{k_1}2 } 2^{-\frac {k_2}2} \bra{2^k}^{N(n)}\bra{2^{k_1}}^{-19}\bra{2^{k_2}}^{- N(n)}.
		\end{align*}
		When $\min(\bm k ) = k$, we have, for small $0< \zeta \ll 1$,
		\begin{align*}
			|\eqref{eq:bulk-2-1}| \les \ve_1^32^{2H(n)\de m} 2^{-\frac m2 + \zeta m} 2^{\zeta k_2} 2^{-k_1} \bra{2^k}^{N(n)}\bra{2^{k_1}}^{-19}\bra{2^{k_2}}^{- N(n)},
		\end{align*}
		which implies the desired bound for \eqref{eq:bulk-2-1}  from the fact that $k_1 \ge - \frac25 m$. When $\min(\bm k) = k_2$ or $k_1$, by the second condition of \eqref{eq:bulk-support}, we directly obtain the desired results. Analogously, using Bernstein's inequality, we estimate
		\begin{align*}
			|\eqref{eq:bulk-2-2}|  &\les 2^m 2^{k}  \bra{2^k}^{N(n)}\sup_{s\in I_m}  \|\pk \psi_{\theta, \cL}(s)\|_{L^2} \|\phi_{\theo}^{> J,k_1}(s)\|_{L^2}\|\pkt W_{\mu, \cL, \theta_2}(s)\|_{L^2} \\
			&\les \ve_1^3 2^{[H(1)+2H(n)]\de m}  2^{k-k_1} \bra{2^k}^{N(n)}\bra{2^{k_1}}^{-N(1)}\bra{2^{k_2}}^{- N(n)}.
		\end{align*}
		Since $2^{k-k_1} \les 2^{-\frac15 m}$, this finishes the proof for {\bf Case 2}.

		\textbf{Case 3: $2^{\min(\bm k)} \ge 2^{-m} $, $2^k \ge 2^{-\frac35 m}$,   and $2^{\max(\bm{k}) } \le 2^{20 \de m}$.} Using the integration by parts in time, 	$\cK_{m,\bm{k}}$ is decomposed by
		\begin{align}
			&\begin{aligned}\label{eq:bulk3-1}
				&\int_{I_m} q_m'(s)  \int_{\R^{3+3}}  \frac{M_n^2(\xi,\eta)}{p_\Theta(\xi,\eta)}e^{isp_{\Theta}(\xi,\eta)} \\
				&\hspace{3cm} \times  \bra{\wh{P_k\phi_{\theta,\cL}}(s,\xi),   \wh{P_{k_2}V_{\mu, \cL, \thet,}}(s,\eta) \al^\mu \wh{P_{k_1}\phi_{\theo}}(s,\xi+\eta)  }  d\xi d\eta  ds,
			\end{aligned}\\
			&\begin{aligned}\label{eq:bulk3-2}
				&\int_{I_m} q_m'(s)  \int_{\R^{3+3}}  \frac{M_n^2(\xi,\eta)}{p_\Theta(\xi,\eta)}e^{isp_{\Theta}(\xi,\eta)} \\
				&\hspace{3cm} \times  \p_s\bra{\wh{P_k\phi_{\theta,\cL}}(s,\xi),   \wh{P_{k_2}V_{\mu, \cL, \thet,}}(s,\eta) \al^\mu \wh{P_{k_1}\phi_{\theo}}(s,\xi+\eta)  }  d\xi d\eta  ds.
			\end{aligned}
		\end{align}
		%	\begin{align}
			%		\begin{aligned}\label{eq:bulk-normalform}
				%			{\bf I}_{m,\,\bm k}^0 &:= \cQ[P_{k_1}\phi_{\theo},       P_{k_2} (\al^\mu V_{\mu,\cl, \theta_2}),        P_{k}\phi_{\theta, \cL} ], \\
				%			{\bf I}_{m,\,\bm k}^1 &:= \cQ[P_{k_1}\phi_{\theo},       P_{k_2}( \al^\mu \p_s V_{\mu, \cl, \theta_2}),  P_{k}\phi_{\theta, \cL}], \\
				%			{\bf I}_{m,\,\bm k}^2 &:= \cQ[P_{k_1}(\p_s\phi_{\theo}), P_{k_2} (\al^\mu V_{\mu, \cl, \theta_2}),       P_{k}\phi_{\theta, \cL}], \\
				%			{\bf I}_{m,\,\bm k}^3 &:= \cQ[P_{k_1}\phi_{\theo},       P_{k_2}(\al^\mu V_{\mu, \cl, \theta_2}),        P_{k}(\p_s\phi_{\theta, \cL})],
				%		\end{aligned}
			%	\end{align}
		Using \eqref{eq:high}, \eqref{eq:esti-non-dirac-time}, \eqref{eq:esti-nonl-max-time}, and  \eqref{eq:decay-dirac}, we estimate
		\begin{align}\label{eq:bulk-esti-normal}
			\begin{aligned}
				\| \pk \phi_{\theta,\cL}(s)\|_{L^2} + 2^m \|\pk (\p_s \phi_{\thet,\cL})(s)\|_{L^2} &\les \ve_1 2^{\wt H(n)\de m } \bra{2^k}^{-N(n+1) - 5},\\
				\|e^{-\theo is\brad} \pko \phi_{\theo}\|_{L^\infty} &\les  \ve_1 2^{\frac{k_1}2} 2^{-m + 10\de m } \bra{2^{k_1}}^{-N(1)+2},\\
				\| \pkt V_{\mu, \cL, \theta_2}(s)\|_{L^2} + 2^m \|\pkt (\p_s V_{\mu, \cL, \theta_2})(s)\|_{L^2} &\les \ve_1 2^{\frac{k_2}2}2^{\wt H(n)\de m } \bra{2^{k_2}}^{-N(n)+5}.
			\end{aligned}
		\end{align}
		This yields that
		\begin{align*}
			\left| \eqref{eq:bulk3-1} \right| &\les \ve_1^3 2^{-m + (2\wt H(n)+10)\de m } 2^{\frac{k_1 + k_2}2} \bra{2^k}^{2N(n)-N(n+1) - 5} \bra{2^{k_1}}^{-N(1) +2 }   \bra{2^{k_2}}^{-N(n) + 5}.
		\end{align*}
		In \eqref{eq:bulk3-2}, when the time derivative falls on $\pk \phi_{\theta,\cL}$ or $\pkt V_{\mu,\cL,\thet}$, we estimate	
		\begin{align*}
			\left|\eqref{eq:bulk3-2} \right| \les \ve_1^3 2^{-2m + (2\wt H(n)+10)\de m } 2^{\frac{k_1 + k_2}2} \bra{2^k}^{2N(n)-N(n+1) - 5} \bra{2^{k_1}}^{-N(1)  +2 }   \bra{2^{k_2}}^{-N(n) +5}.
		\end{align*}
		It remains to handle the case that $\p_s$ falls on $\pko \phi_{\theo}$ in \eqref{eq:bulk3-2}. By the condition of $\textbf{Case 3}$, \eqref{eq:decay-maxwell} and \eqref{eq:decay-cor} show us that
		\begin{align}
			\begin{aligned}\label{eq:bulk-zeta}
				\| e^{-\theo  is \brad} &\pko (\p_s \phi_\theo (s) ) \|_{L^\infty}\\
				&\sim \|\Pi_{\theo}(D)\pko(A_\mu(s) \al^\mu \psi(s))\|_{L^\infty}\\
				&\les 2^{3\zeta k_1} \|\pko(A_\mu(s) \al^\mu \psi(s))\|_{L^{\frac1\zeta}} \\
				& \les 2^{3\zeta k_1} \sum_{\theta_3,\theta_4
					\in \{\pm\}} \sum_{\substack{j_\ell \in \cU_{k_\ell}\\ k_\ell \in \Z, \ell =3,4}} \|Q_{j_3,k_3}\psi_{\theta_3}(s)\|_{L^\frac1\zeta}\|Q_{j_4,k_4}A_{\mu,\theta_4}(s)\|_{L^\infty} \\
				& \les \ve_1^2  2^{-2m + 12\de m} \bra{2^{k_1}}^{-N(1) + 2}
			\end{aligned}
		\end{align}
		for sufficiently small $\zeta > 0$. Note that we used the frequency relation $2^{k_1} \sim 2^{\max(k_3,k_4)}$ in the fourth inequality. Then we have
		\begin{align*}
			|\eqref{eq:bulk3-2}| \les \ve_1^3 2^{-2m + 20\de m} 2^{-\frac {k_2}2}.
		\end{align*}
		Since $\max(\bm k) \le 4\de_1 m$, we readily obtain
		\begin{align*}
			\sum_{ \rm Case \; 3 }  |\cK_{m,\,\bm{k}}| \les \ve_1^3 2^{2 H(n) \de m}.
		\end{align*}

		\textbf{Case 4:   $ 2^{\min(\bm k)} \ge 2^{-m}$, $2^k \ge 2^{-\frac35 m}$, and $2^{\max(\bm{k})} \ge 2^{20\de m}$.} We prove
		\begin{align}\label{eq:bulk-case4}
			\sum_{ \rm Case \; 4 }  |\cK_{m,\,\bm{k}}| \les \ve_1^3 2^{2 H(n) \de m}.
		\end{align}
		To this end, we consider the subcases $\min(\bm k)=k_1$, $ \min(\bm k) = {k_2} $, and $\min(\bm k) = k$.
		If $\min(\bm k)=k_1$, then by the multiplier estimates, we see that
		\begin{align*}%\label{eq:bulk-step3}
			|\cK_{m,\,\bm k} | \les  |I_m| 2^{-\frac{k_2}2}\bra{2^k}^{2N(n)}\|\pk \psi_{\theta,\cL}(s)\|_{L^2} \|\pko \psi_{\theo}(s) \|_{L^\infty}  \|\pkt V_{\mu, \cL, \theta_2}(s)\|_{L^2} .
		\end{align*}
		Hence \eqref{eq:high} and \eqref{eq:decay-dirac} show that
		\begin{align}\label{eq:bulk-casei}
			|\cK_{m,\,\bm k} | \les \ve_1^3 2^{(10 +2H(n))\de m} 2^{-\frac{k_2}2 } 2^{\frac{k_1}2 } \bra{2^k}^{N(n)}\bra{2^{k_1}}^{-N(1)+2}\bra{2^{k_2}}^{-N(n)}.
		\end{align}
		The RHS of \eqref{eq:bulk-casei} is summable over \textbf{Case 4} with $\min(\bm k)=k_1$.
		
		For the high frequency sum with respect to $k$ and $k_2$, \eqref{eq:bulk-casei} is not sufficient for \eqref{eq:bulk-case4} in the case of $\min(\bm k) = k $ or $\min(\bm k) = k_2$. If $\min(\bm k)=k$, then we use $ L^\infty \times L^2 \times L^2$ estimate. By \eqref{eq:high}, and \eqref{eq:decay-dirac}, we obtain
		\begin{align}\label{eq:bulk-case4-1}
			|\cK_{m,\,\bm k} | \les \ve_1^3 2^{[2H(n) + 1]\de m }2^{-\frac{k_2}2 } \bra{2^k}^{N(n)}\bra{2^{k_1}}^{-N(n+1) - N(0)}
		\end{align}
		for $n = 1,2$. Now the RHS of \eqref{eq:bulk-case4-1} is summable.  Analogously, for the case $\min(\bm k) = k_2$, we use $L^2 \times L^2  \times  L^\infty$ estimate. Then \eqref{eq:decay-maxwell} yield that, for $n =0,1,2$,
		\begin{align}\label{eq:bulk-case4-2}
			|\cK_{m,\,\bm k} | \les \ve_1^3 2^{[2H(n)+H(1)+\frac12]\de m }  \bra{2^{k_2}}^{-N(n+1)+2}  \bra{2^{k_1}}^{N(n) - N(0)},
		\end{align}
		which finishes \eqref{eq:bulk-case4} for $n = 1, 2$.
		
		The estimates \eqref{eq:bulk-case4-1} and \eqref{eq:bulk-case4-2} do not cover \eqref{eq:bulk-case4} when $n = 3$. To handle this case, we use the integration by parts in time and obtain \eqref{eq:bulk3-1} and \eqref{eq:bulk3-2}. Since $2^{\max(\bm k)} \sim 2^{k_1}$, the second estimate in \eqref{eq:bulk-esti-normal} guarantees the high frequency sum for \eqref{eq:bulk3-1} and the sum for the cases that $\p_s$ falls on $\phi_{\theta,\cL}$ and $V_{\mu,\cL,\thet}$ in \eqref{eq:bulk3-2}. To estimate remainder case, one can utilize \eqref{eq:bulk-zeta} for the high frequency sum. This completes the proof of \eqref{eq:bulk-case4}.
	\end{proof}

	\subsection{Bulk estimate on Maxwell part}
	Let us consider the remaining case of Maxwell part:
	\begin{align}\label{eq:bulk-maxwell}
		\int_{I_m}\int_{\R^{3+3}} q_m(s) |\xi|^{-\frac12} \bra{\xi}^{2N(n)}  \bra{\wh{ \psi_{\theta_1, \cl_1}}(s, \eta), \al_\mu \wh{ \psi_{\theta_2,\cl_2}}(s,\xi+\eta)} \overline{\wh{ W_{\mu, \cL, \theta'}}}(s,\xi)   d\eta d\xi ds.
	\end{align}
	Compared to $\cK_{m,\bm K}$ discussed in Section \ref{sec:bulk-dirac}, the only difference, resulting from the change of variables $(\xi,\eta) \mapsto (\eta',\xi'-\eta')$, is in the multipliers $\bra{\xi'}$ and $\bra{\eta'}$. Consequently, the bound for \eqref{eq:bulk-maxwell} can be obtained in the same manner. We omit the details.

%\section{Energy estimates II: extended vector fields}\label{sec:energy2}

\section{Proof of weighted energy estimates} \label{sec:weight}

In this section, we prove Proposition \ref{prop:weighted}.  At first we consider \eqref{eq:weighted-dirac}. By Plancherel's theorem, we have
\begin{align*}
	\|P_k (x_j \phi_\theta(t))\|_{H^1} \sim \bra{2^k}\|\rho_k(\xi) \p_{\xi_j}\wh{\phi_\theta}(t,\xi)\|_{L^2}.
\end{align*}
Using \eqref{eq:weight-lorentz-dirac}, we see that
\begin{align*}
	 e^{-\theta it \bra{\xi}}\bra{\xi}\p_{\xi_j}\wh{\phi_\theta}(t,\xi) = - \theta \wh{\Gam_j \psi_\theta}(t,\xi) +\theta \p_{\xi_j} \left[\Pi_\theta(\xi) \wh{A_\mu \al^\mu \psi}(t,\xi) \right] - e^{-\theta it \bra{\xi}} \frac{\xi_j}{\bra{\xi}}\wh{\phi_\theta}(t,\xi).
\end{align*}
This gives us that
\begin{align*}
	\bra{2^k}\|\rho_k(\xi) \p_{\xi_j}\wh{\phi_{\theta,\cL}}(t,\xi)\|_{L^2} &\les  \| P_k \Gam_j \psi_{\theta,\cL}\|_{L^2} + \left\|\rho_k(\xi) \p_{\xi_j} \left[\Pi_\theta(\xi) \wh{A_\mu \al^\mu \psi}(t,\xi) \right] \right\|_{L_\xi^2} + \left\|P_k \phi_{\theta,\cL} \right\|_{L^2}\\
	&\les \ve_1^2 \bra{t}^{H(n+1)\de} \bra{2^k}^{-N(n+1)}.
\end{align*}
For this we used Proposition \ref{prop:energy} and \eqref{eq:esti-nonlinear-dirac-space}.

Let us move on to the proof of \eqref{eq:weighted-wave}. By \eqref{eq:weight-lorentz-wave}, we obtain
\begin{align*}
	e^{\theta it |\xi|}|\xi|\p_{\xi_j}\wh{V_{\mu,\cL,\theta}  }(t,\xi) = - \theta \wh{\Gam_j W_{\mu,\cL,\theta}}(t,\xi) +\frac12  \p_{\xi_j} \left[|\xi|^{-\frac12}\wh{\cL\bra{\psi,\al_\mu \psi}}(\xi) \right] - e^{\theta it |\xi|} \frac{\xi_j}{|\xi|}\wh{V_{\mu,\cL,\theta}}(t,\xi).
\end{align*}
Then, Proposition \ref{prop:energy} and \eqref{eq:esti-nonlinear-maxwell-space} yield that
\begin{align*}
	2^k\|\rho_k(\xi) \p_{\xi_j}\wh{V_{\mu, \cL, \theta}}(t,\xi)\|_{L^2} &\les  \| P_k \Gam_j W_{\mu, \cL, \theta}\|_{L^2} + \left\|\rho_k(\xi) \p_{\xi_j} \left[|\xi|^{-\frac12}\wh{\cL \bra{\psi,\al_\mu \psi}}(\xi) \right] \right\|_{L_\xi^2} + \left\|P_k W_{\mu, \cL, \theta} \right\|_{L^2}\\
	&\les \ve_1^2 \bra{t}^{H(n+1)\de} \bra{2^k}^{-N(n+1)}.
\end{align*}
This completes the proof of \eqref{eq:weighted-wave}.

\section{Asymptotic behavior for the spinor field}\label{sec:s-dirac}
This section  is devoted to proving the bound of scattering norm and the modified scattering for spinor fields under the a priori assumptions \eqref{assum:energy}--\eqref{assum:s-norm} on $[0,T]$. In view of the definition of scattering norm \eqref{def:s-norm-dirac}, we need to show the following estimates
\begin{align}\label{eq:aim-asymptotic-dirac}
	\begin{aligned}
		\left\|\wh{P_k\phi_\theta}(t) \right\|_{L_\xi^\infty} &\les \ve_1^2 2^{-\de m} 2^{-\frac k2 + \frac k{100} } \bra{2^k}^{-20},\\
		\left\|\wh{P_k\phi_\theta}(t) \right\|_{L_\xi^2} &\les \ve_1^2 2^{-\frac\de2 m} 2^{k + \frac k{100} } \bra{2^k}^{-38}
	\end{aligned}
\end{align}
for $t \in [2^m-2, 2^{m+1}]\cap [0, T]$.  This gives us the bound \eqref{eq:s-norm-d}. Now  we define the phase modification symbol by
\begin{align}\label{eq:phase-correction}
		B_\theta(t,\xi) &= \Pi_\theta(\xi)\al^\mu \int_0^t P_{\le K} A_{\mu}\left(s, \frac{\theta s \xi}{\bra{\xi}}\right)ds,
\end{align}
where $K = K(s)$ is an integer satisfying $2^K \le \bra{s}^{-\frac23 -10H(2)\de}< 2^{K+1}$ and $\theta \in \{+,-\}$. We also define a profile associated with scattering by
\begin{align*}
	\Psi_\theta (t,\xi) := e^{-iB_\theta(t,\xi)} \wh{\phi_\theta}(t,\xi)=e^{-iB_\theta(t,\xi)}e^{\theta it \bra{\xi}}\wh{\psi_\theta}(t,\xi).
\end{align*}
Then the estimates of \eqref{eq:aim-asymptotic-dirac} are equivalent to the bounds
\begin{align}
&	\left\| \wh{P_k\Psi_\theta}(t_2,\xi) -\wh{P_k\Psi_\theta}(t_1,\xi) \right \|_{L_\xi^\infty} \les \ve_1^2\bra{t_2}^{-\de}2^{-\frac k2 + \frac k{100} } \bra{2^k}^{-20},\label{eq:aim-scattering-dirac}\\
&	\left\| \wh{P_k\Psi_\theta}(t_2,\xi) -\wh{P_k\Psi_\theta}(t_1,\xi) \right \|_{L_\xi^2} \les \ve_1^2\bra{t_2}^{-\frac\de2}2^{k + \frac k{100} } \bra{2^k}^{-38},\label{eq:aim-scattering-2}
\end{align}
for $k\in\Z$ and $t_1 \le t_2 \in [0,T]$. Let $t_1 \le t_2 \in [2^m -2, 2^{m+1}] \cap [0,T]$. Once \eqref{eq:aim-scattering-dirac} is established, the proof of \eqref{eq:aim-scattering-2} is quite straightforward. More precisely,  the bound  \eqref{eq:aim-scattering-2} follows from \eqref{eq:high} if $2^k \ge 2^{\frac m{40}\de}$. On the other hand, the estimate \eqref{eq:aim-scattering-dirac} implies \eqref{eq:aim-scattering-2} directly if $2^k \le 2^{\frac m{40}\de}$.

 For \eqref{eq:aim-scattering-dirac} let us observe that
\begin{align}\label{eq:diff-s-profile}
	\wh{P_k\Psi_\theta}(t_2,\xi) -\wh{P_k\Psi_\theta}(t_1,\xi) = \int_{t_1}^{t_2}   \p_s \wh{P_k\Psi_\theta}(s,\xi)ds
\end{align}
and
\begin{align}\label{eq:deri-phase}
	\p_s \wh{P_k\Psi_\theta}	(s,\xi) = e^{-iB_\theta(s,\xi)}\left( \p_s \wh{P_k \phi_\theta}(s,\xi) - i\left[\p_s B_\theta(s,\xi)\right] \wh{P_k\phi_\theta}(s,\xi) \right).
\end{align}
The phase modification \eqref{eq:phase-correction} plays a role of subtracting the resonance interaction in \eqref{eq:deri-phase}. Indeed, the explicit form can be derived to track the resonance case of nonlinearity of Dirac part. Taking Fourier transform, we see that
\begin{align}\label{eq:nonlinear-correction}
	\p_s \wh{\phi_\theta}(s,\xi) =  c	\Pi_\theta(\xi)\sum_{\theo \in \{\pm\}}\int_{\R^3} e^{is\left( \theta \bra{\xi} - \theo \bra{\xi-\eta}\right)} \wh{A_{\mu}}(s,\eta) \al^\mu \wh{\phi_\theo}(s,\xi-\eta)\, d\eta,
\end{align}
where $c = \frac{i}{(2\pi)^3}$.
In view of \eqref{eq:resonance-time}, since the resonance does not occur when $\theta \neq \theo$, we take into account the case $\theta =\theo$. We simply observe that the phase interaction satisfies
\begin{align}\label{eq:phase-approximation}
	\theta \bra{\xi} - \theta \bra{\xi-\eta}= \theta \frac{\xi\eta}{\bra{\xi}} + O(|\eta|^2).
\end{align}
Using \eqref{eq:phase-approximation}, we can rewrite \eqref{eq:nonlinear-correction} and hence \eqref{eq:diff-s-profile} as follows:
 \begin{align}
	&\wh{P_k\Psi_\theta}(t_2,\xi) - \wh{P_k\Psi_\theta}(t_1,\xi) =\nonumber\\
	&\quad c	\int_{t_1}^{t_2}e^{-iB_\theta(s,\xi)}\Pi_\theta(\xi)\rho_k(\xi)\int_{\R^3} e^{\theta is\frac{\xi\eta}{\bra{\xi}}} \rho_{\le K}(\eta)\wh{A_{\mu}}(s,\eta) \al^\mu \left(\wh{\phi_\theta}(s,\xi-\eta) - \wh{\phi_\theta}(s,\xi)\right)\, d\eta ds\label{eq:correction-1}\\
	&+ c	\int_{t_1}^{t_2}e^{-iB_\theta(s,\xi)}\Pi_\theta(\xi)\rho_k(\xi) \int_{\R^3} \left( e^{\theta is(\bra{\xi}-\bra{\xi-\eta})} -  e^{\theta is\frac{\xi\eta}{\bra{\xi}}}  \right)\rho_{\le K}(\eta) \wh{A_{\mu}}(s,\eta) \al^\mu \wh{\phi_\theta}(s,\xi-\eta) \,d\eta ds\label{eq:correction-2}\\
	&+  c	\int_{t_1}^{t_2}e^{-iB_\theta(s,\xi)}\Pi_\theta(\xi)\rho_k(\xi) \sum_{\theo \in \{\pm\}}\int_{\R^3} e^{\theta is\left(  \bra{\xi} - \bra{\xi-\eta}\right)} \rho_{> K}(\eta)\wh{A_{\mu}}(s,\eta) \al^\mu \wh{\phi_\theo}(s,\xi-\eta)\, d\eta ds\label{eq:correction-3}\\
	&+c	\int_{t_1}^{t_2}e^{-iB_\theta(s,\xi)}\Pi_\theta(\xi)\rho_k(\xi) \int_{\R^3} e^{\theta is\left(  \bra{\xi} + \bra{\xi-\eta}\right)} \rho_{\le  K}(\eta)\wh{A_{\mu}}(s,\eta) \al^\mu \wh{\phi_{-\theta}}(s,\xi-\eta)\, d\eta  ds.\label{eq:correction-4}
\end{align}
Note that $2^K \sim 2^{-\frac23m -10H(2)\de m}$.

We begin with the proof of \eqref{eq:aim-asymptotic-dirac} and \eqref{eq:aim-scattering-dirac} with the  dyadic decomposition into low, high, and mid frequencies. For the low and high frequencies, we estimate \eqref{eq:aim-asymptotic-dirac} directly (See Section \ref{sec:s-dirac-hl} below). For the mid frequency part, we prove \eqref{eq:aim-scattering-dirac} by handling the resonance decomposition \eqref{eq:correction-1}--\eqref{eq:correction-4} (See Section \ref{sec:s-dirac-mid} below).

\subsection{Proof for the high and low frequencies.}\label{sec:s-dirac-hl} We prove \eqref{eq:aim-asymptotic-dirac} for the frequency range:
\begin{align*}%\label{eq:s-support-hl}
	2^k  \le 2^{- 30H(2)\de m} \;\;\mbox{ or }\;\; 2^{k} \ge 2^{2H(2)\de m}.
\end{align*}
By \eqref{eq:interpolation-linfty}, we have
\begin{align*}
	\|\wh{P_k \phi_{\theta}}(t)\|_{L_\xi^\infty} \les 2^{-\frac{3k}2} \left(\sup_{j \in \cU_k}\|\qjk \phi_{\theta}\|_{H_\Om^{0,1}}\right)^{\frac{1-\zeta}2}\left(\sup_{j \in \cU_k}2^{j+k}\|\qjk \phi_{\theta}\|_{H_\Om^{0,1}}\right)^{\frac{1+\zeta}2}
\end{align*}
for any $\zeta \in (0, 1)$.
%From \eqref{eq:weight}, it follows that
%\begin{align*}
%	\normo{P_k \left( x_j\phi_{\theta,\cL} \right)(t)}_{L^2}  \les \ve_1^2  \bra{t}^{H(n+1) \de} \bra{2^k}^{-N(n+1)-1},
%\end{align*}
%for $\cL \in \cV_n \,(0 \le n \le 2)$, which implies
%\begin{align}\label{eq:esti-qjk-l2}
%	\sup_{j\in \cU_k} 2^j \|\qjk \phi_{\theta,\cL}(t)\|_{L^2} \les \ve_1   \bra{t}^{H(n+1) \de} \bra{2^k}^{-N(n+1)-1},
%\end{align}
%due to \eqref{eq:esti-hardy2}.
From \eqref{eq:elliptic-q}, it follows that
\begin{align}\label{eq:esti-qjk-l2}
	\sup_{j\in \cU_k} 2^j \|\qjk \phi_{\theta,\cL}(t)\|_{L^2} \les \ve_1^2   \bra{t}^{H(n+1) \de} \bra{2^k}^{-N(n+1)-1}
\end{align}
for $\cL \in \cV_n \,(0 \le n \le 2)$, which implies
\begin{align}\label{eq:amplitude-high}
	\|\wh{P_k \phi_{\theta}}(t)\|_{L_\xi^\infty} \les \ve_1 ^2  2^{H(2) \de m}2^{-\frac k2}  \bra{2^k}^{-N(2)-1}.
\end{align}
For the high frequencies $2^k \ge 2^{2H(2) \de m}$, the bounds \eqref{eq:amplitude-high} and \eqref{eq:energy-high-d} yield the desired bound \eqref{eq:aim-asymptotic-dirac}.

Let us move on to the low frequency regime $2^k \le 2^{-30H(2)\de m}$. This can be done by obtaining
\begin{align}\label{eq:esti-omega-l2}
	\|P_k \phi_{\theta,\Om}(t)\|_{L^2} \les \ve_1^2 2^{2H(2)\de m} 2^{k+\frac k6}.
\end{align}
Note that the above estimate is an improvement of \eqref{eq:esti-qjk-l2} by the factor $2^{\frac k6}$. Then \eqref{eq:esti-hardy2} yields that
\begin{align*}
	\normo{\wh{P_k \phi_\theta}(t)}_{L_\xi^\infty} \les \ve_1^2 2^{2H(2)\de m}2^{-\frac k2 + \frac{1-\zeta}{12}k}.
%&	\normo{\wh{P_k \phi_\theta}(t)}_{L_\xi^2} \les \ve_1^2 2^{2H(2)\de m}2^{k + \frac{1-\zeta}{12}k}.
\end{align*}
This finishes the proof of \eqref{eq:aim-asymptotic-dirac} for $2^k \le 2^{-30H(2)\de m}$.

We now prove \eqref{eq:esti-omega-l2}. By \eqref{eq:esti-hardy3}, it suffices to show
\begin{align*}
	\|P_k \phi_{\theta,\Om}\|_{L^2} + \sum_{l=1}^3 \normo{\rho_k(\xi)\wh{ x_l \phi_{\theta,\Om}}(t,\xi)}_{L_\xi^2} \les \ve_1^2 \bra{t}^{2H(2)\de}2^{\frac k6}.
\end{align*}
Since the first term in the left-hand side can be handled by \eqref{eq:high} straightforwardly, we focus on the second term. Using \eqref{eq:weight-lorentz-dirac}, we prove
\begin{align*}
	\|P_k\Gam_l \psi_{\theta,\Om}\|_{L^2} + 	\normo{P_k(x_l \mathfrak{N}_\Om^{\bm D})(t)}_{L_\xi^2} \les  \ve_1^2 \bra{t}^{2H(2)\de}2^{\frac k6}.
\end{align*}
The estimates for the first term also is obtained by Proposition \ref{prop:energy}. Thus, we consider
\begin{align*}
	\normo{P_k \mathfrak{N}_\Om^{\bm D}(t)}_{L_\xi^2} &\les \ve_1^2 \bra{t}^{2 H(2) \de} 2^{\frac k6}\min( \bra{t}^{-1}, 2^k) ,
\end{align*}
which implies
	\begin{align*}
		\normo{P_k(x_l \mathfrak{N}_\Om^{\bm D})(t)}_{L_\xi^2} &\les \ve_1^2 \bra{t}^{2\wt H(2) \de} 2^{\frac k6}.
	\end{align*}

Let $k \in \Z$ satisfy $2^{k} \le 2^{-30H(2)\de m}$. We prove
\begin{align}
	\sum_{k_1,k_2 \in \Z}  \|P_k\left( P_{k_2} A_{\mu,\thet,\cL_2}(t) \al^\mu \pko \psi_{\theo,\cL_1}(t)\right) \|_{L^2} &\les \ve_1^2 \bra{t}^{2 H(2)\de}2^{\frac k6} \min( \bra{t}^{-1}, 2^k),\label{eq:aim-asym-bilinear1}\\
	\sum_{k_1,k_2 \in \Z}  \|P_kx_l \left( P_{k_2} A_{\mu,\thet,\cL_2}(t) \al^\mu \pko \psi_{\theo,\cL_1}(t)\right) \|_{L^2} &\les \ve_1^2 \bra{t}^{2 H(2)\de}2^{\frac k6}\label{eq:aim-asym-bilinear2}
\end{align}
for $l = 1, 2, 3$, $\cL_\ell \in  \cV_{n_\ell} (\ell = 1, 2)$, and $n_1 + n_2 \le 1$.
To show \eqref{eq:aim-asym-bilinear1}, we utilize the estimates in the proof of Lemma \ref{lem:esti-nonlinear-dirac}.  We first  consider the cases: $2^k \le \bra{t}^{-1}$ and $\bra{t}^{-1} \le 2^{k} \le 2^{-\de m}$.  By \eqref{eq:high} and \eqref{eq:esti-l2-k}, we have
\begin{align*}
	\|P_k( P_{k_2} A_{\mu, \cL_2, \thet}(t) &\al^\mu \pko \psi_{\theo,\cL_1}(t)) \|_{L^2}\\
	&\les 2^{\frac{3\min(\bm k) - k_2}2} \|\pko \psi_{\theo,\cL_1}(t)\|_{L^2}\|\pkt W_{\mu, \cL_2, \thet}(t)\|_{L^2}\\
	&\les \ve_1^2 \bra{t}^{2H(2)\de} 2^{\frac{3\min(\bm k) - k_2}2}2^{k_1}\bra{2^{k_1}}^{-N(n_1)}\bra{2^{k_1}}^{-N(n_2)}.
\end{align*}
This estimate shows \eqref{eq:aim-asym-bilinear1} when $2^k \le \bra{t}^{-1}$ .

If  $\bra{t}^{-1} \le 2^{k} \le 2^{-30H(2)\de m}$, we use \eqref{eq:esti-l2-k} and \eqref{eq:l-infinity} to obtain
\begin{align*}
\|P_k( &P_{k_2} A_{\mu, \cL_2, \thet}(t) \al^\mu \pko \psi_{\theo,\cL_1}(t)) \|_{L^2}\\
	&\les 2^{-\frac{k_2}2} \|\psi_{\theo,\cL_1}(t)\|_{L^2}\|W_{\mu, \cL_2, \thet}(t)\|_{L^\infty}\\
	&\les \ve_1^2 \bra{t}^{\left[H(n_1+1)+H(n_2+1)\right]\de+\frac \de2 }2^{k_1}\min(\bra{t}^{-1},2^{k_2})\bra{2^{k_1}}^{-N(n_1+1)+1}\bra{2^{k_2}}^{-N(n_2+1)+2}.
\end{align*}
The bound \eqref{eq:aim-asym-bilinear1} follows from the summation over $k,k_1,k_2 \in \Z $ apart from the case $2^k \ll 2^{k_1} \sim 2^{k_2}$ and $2^{k_1} \ge 2^{\frac k3}$. For the remaining case, we divide $\pko \phi_{\theo,\cL_1} = \phi_{\theo,\cL_1}^{\le J,k_1} + \phi_{\theo,\cL_1}^{> J,k_1}$ with $2^J = C\bra{t}2^{k_1}$ for some $C \ll 1$. Then, \eqref{eq:decay-dirac-2} and \eqref{eq:elliptic-q} yield that
\begin{align*}
	\|e^{-\theo it\bra{D}}\phi_{\theo,\cL_1}^{\le J,k_1}(t)\|_{L^\infty} &\les \ve_1 \bra{t}^{H(n_1+2)-\frac32 +\frac \de2} 2^{-\frac {k_1}2},\\
	\|\phi_{\theo,\cL_1}^{> J,k_1}(t)\|_{L^2} &\les \ve_1 \bra{t}^{H(n_1+1)-1} 2^{-k_1}.
\end{align*}
By these we have
\begin{align*}
	\|P_k( P_{k_2} A_{\mu, \cL_2, \thet} (t)&\al^\mu  e^{-\theo it\bra{D}}\phi_{\theo,\cL_1}^{\le J,k_1}(t))\|_{L^2}\\
	&\les 2^{-\frac{k_2}2} \|e^{-\theo it\bra{D}}\phi_{\theo,\cL_1}^{\le J,k_1}(t)\|_{L^\infty}\|W_{\mu, \cL_2, \thet}(t)\|_{L^2}\\
	&\les \ve_1^2 \bra{t}^{2H(2)\de -\frac32}  2^{-\frac {k_1+ k_2}2} \bra{2^{k_2}}^{-N(n_2)}
\end{align*}
and
\begin{align*}
	 \|P_k( P_{k_2} A_{\mu, \cL_2, \thet} (t)&\al^\mu e^{-\theo it \bra{D}} \phi_{\theo,\cL_1}^{> J,k_1}(t))\|_{L^2}\\
	&\les 2^{\frac{3k-k_2}2} \|\phi_{\theo,\cL_1}^{> J,k_1}(t)\|_{L^2}\|W_{\mu, \cL_2, \thet}(t)\|_{L^2}\\
	&\les \ve_1^2 \bra{t}^{\left[H(n_1+1)+H(n_2)\right]\de-1}  2^{\frac {3k-2k_1- k_2}2}.
\end{align*}
We get the bound \eqref{eq:aim-asym-bilinear1} for $2^k \ll 2^{k_1} \sim 2^{k_2}$ and $2^{k_1} \ge 2^{\frac k3}$.
Since the proof of \eqref{eq:aim-asym-bilinear2} is similar to that of Lemma \ref{lem:nonlinear-wei-dirac}, we omit the details.

\subsection{Proof for mid frequencies}\label{sec:s-dirac-mid} This section is devoted to taking on the mid frequency regime
\begin{align}\label{eq:mid-frequency}
	2^{-30H(2) \de m} \le 2^k \le 2^{2H(2)\de m}.
\end{align}
As observed in \eqref{eq:nonlinear-correction}, we estimate \eqref{eq:aim-asymptotic-dirac} on the regime \eqref{eq:mid-frequency} by decomposition \eqref{eq:correction-1}--\eqref{eq:correction-4}.

\emph{Bound for \eqref{eq:correction-1}.}
%We approximate \eqref{eq:correction-1} as follows:
%\begin{align*}
%	\int_{t_1}^{t_2}e^{-iB_\theta(s,\xi)}\int_{\R^3} e^{\theta is\frac{\xi\eta}{\bra{\xi}}} \rho_{\le L}(\eta)\wh{A_{\mu}}(s,\eta) \al^\mu \left[ \wh{\phi_\theta}(s,\xi-\eta) -\wh{\phi_\theta}(s,\xi)\right] d\eta ds,
%\end{align*}
%for $t_1,t_2 \in [2^{m-2},2^{2m} ]$.
By the frequency localization, we see that
\begin{align*}%\label{eq:aim-cor-1}
	\begin{aligned}
		\eqref{eq:correction-1} &= c\int_{t_1}^{t_2}\sum_{k_1,k_2 \in
			\Z} \Pi_\theta(\xi)\rho_k(\xi)e^{-iB_\theta(s,\xi)}\int_{\R^3} e^{\theta is\frac{\xi\eta}{\bra{\xi}}} \rho_{\le K}(\eta)|\eta|^{-\frac12}\wh{\pkt W_{\mu}}(s,\eta)\\
		& \hspace{5cm}\times\al^\mu \left[ \wh{\pko\phi_\theta}(s,\xi-\eta) -\wh{P_{k_1}\phi_\theta}(s,\xi)\right] d\eta ds.
	\end{aligned}
\end{align*}
%The frequency relation can be divided into three cases. However,
By the frequency cut-off function with respect to $\xi$ and $\eta$, it suffices to consider the case $2^{k_2} \ll 2^{k_1} \sim 2^k$. To investigate the difference between spinors in the integrand, we make a further decomposition $\pko \phi_\theta = \sum_{j_1 \in \cU_{k_1}} \phi_\theta^{j_1,k_1}$. By \eqref{eq:linear-amplitude} and \eqref{eq:elliptic-q}, we obtain
\begin{align}\label{eq:amplitude-dirac}
\left\|\wh{\phi_\theta^{j_1,k_1}} (s)\right\|_{L_\xi^\infty} \les \ve_1  2^{ H(2)\de m} 2^{-\frac {j_1}2 + \frac{\de j_1}8}2^{-{k_1}} \bra{2^{k_1}}^{-N(2)}.
\end{align}
Using this estimate and mean value theorem, we see from \eqref{eq:derivative-j} that
\begin{align}\label{eq:amplitude-difference}
	\left| \wh{\phi_\theta^{j_1,k_1}}(s,\xi-\eta) -\wh{\phi_\theta^{j_1,k_1}}(s,\xi) \right| \les \ve_1 2^{ H(2)\de m} 2^{k_2}2^{\frac {j_1}2 + \frac{\de j_1}8} 2^{-{k_1}} \bra{2^{k_1}}^{-N(2)}.
\end{align}
Let $J_0$ be an integer such that $2^{J_0} \sim 2^{\frac23 m + 20H(2)  \de m} $. Then the bounds \eqref{eq:amplitude-dirac} and \eqref{eq:amplitude-difference} imply that
\begin{align*}
&\sum_{j_1 > J_0}\left\|\wh{\phi_\theta^{j_1,k_1}} (s)\right\|_{L_\xi^\infty} \les \ve_1  2^{ H(2)\de m} 2^{-\frac {J_0}2 + \frac{\de J_0}8}2^{-{k_1}} \bra{2^{k_1}}^{-N(2)},\\
&\sum_{j_1 \le J_0}\left| \wh{\phi_\theta^{j_1,k_1}}(s,\xi-\eta) -\wh{\phi_\theta^{j_1,k_1}}(s,\xi) \right| \les \ve_1 2^{ H(2)\de m} 2^{k_2}2^{\frac {J_0}2+ \frac{\de J_0}8} 2^{-{k_1}} \bra{2^{k_1}}^{-N(2)},
\end{align*}
and hence we obtain
\begin{align*}
\left| \eqref{eq:correction-1}\right| &\les \sum_{2^{k_2} \le 2^K,  2^{k_1} \sim 2^k} \ve_1 2^{m}2^{H(2)\de m}2^{-\frac{k_2}2} 2^{-\frac{J_0}2 + \frac{\de J_0}8}2^{-{k_1}}\bra{2^{k_1}}^{-N(2)}\|\rho_{k_2}\|_{L_\eta^2} \|\pkt W_{\mu}(s)\|_{L^2} \\
&\quad + \sum_{2^{k_2}\le 2^K,   2^{k_1} \sim 2^k} \ve_1 2^{m}2^{H(2)\de m}2^{\frac{k_2}2}  2^{\frac {J_0}2+ \frac{\de J_0}8} 2^{-\frac{k_1}2}  \bra{2^{k_1}}^{-N(2)} \|\rho_{\le K}\|_{L_\eta^2} \|\pkt W_{\mu}(s)\|_{L^2}\\
&\les \sum_{2^{k_1} \sim 2^k} \ve_1^2 2^{m}2^{[16H(2)+1]\de m}2^K 2^{-\frac {J_0}2}2^{-\frac{k_1}2}\bra{2^{k_1}}^{-N(2)}\\
&\quad  + \sum_{2^{k_1} \sim 2^k} \ve_1^2 2^{m}2^{[H(2)+1]\de m}2^{2K} 2^{\frac {J_0}2}2^{2H(0)\de m} 2^{-\frac{k_1}2}\bra{2^{k_1}}^{-N(2)}.
\end{align*}
Since $2^K \sim 2^{m(-\frac23 + 10H(2)\de )}$, $2^{J_0} \sim 2^{\frac23m  + 20H(2)\de m}$, using the restriction \eqref{eq:mid-frequency} and frequency relation $2^{k} \sim 2^{k_1}$, we get the first part of \eqref{eq:aim-asymptotic-dirac}.

\emph{Bound for \eqref{eq:correction-2}.}
\eqref{eq:correction-2} can be treated similarly to \eqref{eq:correction-1}. By dyadic decomposition, one gets
\begin{align}\label{eq:aim-cor-2}
	\begin{aligned}
		&\int_{t_1}^{t_2}\Pi_\theta(\xi)\rho_k(\xi)	e^{-iB_\theta(s,\xi)}\int_{\R^3} \left[ e^{\theta is(\bra{\xi}-\bra{\xi-\eta})} -  e^{\theta is\frac{\xi\eta}{\bra{\xi}}}  \right]\rho_{\le K}(\eta) \wh{\pkt A_{\mu}}(s,\eta) \\
		&\hspace{8cm}\times \al^\mu \wh{\pko \phi_\theta}(s,\xi-\eta)\, d\eta ds
	\end{aligned}
\end{align}
for $k_1,k_2 \in \Z$. As we observed in the proof of bound for \eqref{eq:correction-1}, we have only to consider the case $2^{k_2} \ll 2^{k_1} \sim 2^{k}$. Using \eqref{eq:phase-approximation} and \eqref{eq:amplitude-dirac}, we see that
\begin{align*}
	\sum_{2^{k_2} \ll 2^{k_1}}  \left| \eqref{eq:aim-cor-2}\right| &\les 	\sum_{2^{k_2} \ll 2^{k_1}} 2^{2m} 2^{3K} \|\pkt W_{\mu,\thet}(s)\|_{L^2} \left\| \wh{\pko \phi_\theta}(s)\right\|_{L_\xi^\infty}\\
	&\les 	\sum_{2^{k_2} \ll 2^{k_1}} \ve_1 2^{[1+ H(2)]\de m -30H(2) \de m} 2^{-\frac{k_1}2} \bra{2^{k_1}}^{-N(2)}\bra{2^{k_2}}^{-N(0)}\\
	&\les \ve_1 2^{-\de m}2^{-\frac k2 + \frac1{100} k}\bra{2^{k}}^{-N(2)}.
\end{align*}

\emph{Bound for \eqref{eq:correction-3}.} Analogously, we consider the following integrand
\begin{align}\label{eq:aim-cor-3}
	c\int_{t_1}^{t_2} \Pi_\theta(\xi)	\rho_k(\xi) e^{-iB_\theta(s,\xi)}\int_{\R^3} e^{\theta isp_\Theta(\xi,\eta)} \rho_{> K}(\eta) |\eta|^{-\frac12} \wh{\pkt V_{\mu,\thet}}(s,\eta) \al^\mu \wh{\pko \phi_\theo}(s,\xi-\eta) \, d\eta ds
\end{align}
for $k_1, k_2 \in \Z$. For this we exploit the space-time resonances \eqref{eq:resonance-time} and \eqref{eq:nonresonance-space} suitably.  In view of the phase interaction \eqref{eq:resonance-time}, the most delicate part is the resonance case $\theta=\theo$. We can deal with the other case similarly but with a better bound and hence we will omit it. For the case $\theta =\theo$ we use the integration by parts in time and see that \eqref{eq:aim-cor-3} is the linear combination of the following:
\begin{align}
	& e^{-iB_\theta(t,\xi)}\int_{\R^3} e^{ itp_\Theta(\xi,\eta)}M_{\bm k}(\xi,\eta)\wh{\pkt V_{\mu,\thet}}(t,\eta) \al^\mu \wh{\pko \phi_\theta}(t,\xi-\eta) \,d\eta ds, \quad (t=t_1,t_2),\label{eq:cor-3-1}\\
	&\int_{t_1}^{t_2}   \p_s B_\theta(s,\xi) e^{-iB_\theta(s,\xi)}\int_{\R^3} e^{ isp_\Theta(\xi,\eta)}  M_{\bm k}(\xi,\eta)\wh{\pkt V_{\mu,\thet}}(s,\eta) \al^\mu \wh{\pko \phi_\theta}(s,\xi-\eta) \,d\eta ds,\label{eq:cor-3-2}\\
	&\int_{t_1}^{t_2}  e^{-iB_\theta(s,\xi)}\int_{\R^3} e^{ isp_\Theta(\xi,\eta)} M_{\bm k}(\xi,\eta)\wh{\pkt  \p_s V_{\mu,\thet}}(s,\eta) \al^\mu \wh{\pko \phi_\theta}(s,\xi-\eta)\, d\eta ds,\label{eq:cor-3-3}\\
	&\int_{t_1}^{t_2}  e^{-iB_\theta(s,\xi)}\int_{\R^3} e^{ isp_\Theta(\xi,\eta)} M_{\bm k}(\xi,\eta)\wh{\pkt V_{\mu,\thet}}(s,\eta) \al^\mu \wh{\pko  \p_s \phi_\theta}(s,\xi-\eta)\, d\eta ds,\label{eq:cor-3-4}
\end{align}
where
\begin{align*}
	M_{\bm k}(\xi,\eta) :=  \frac{\Pi_\theta(\xi)|\eta|^{-\frac12}}{ p_\Theta (\xi,\eta)}\rho_{\bm k}(\xi,\eta)\rho_{> K}(\eta)\;\;\mbox{and}\;\; \rho_{\bm k}(\xi, \eta) = \rho_k(\xi)\wt{\rho}_{k_1}(\xi-\eta)\wt{\rho}_{k_2}(\eta).
\end{align*}
Then we have
\begin{align*}
	\left| M_{\bm k}(\xi,\eta)\right| \les 2^{-\frac{3k_2}2}\bra{2^k}\bra{2^{\max(\bm k)}}\les 2^{-\frac{3k_2}2}2^{4H(2)\de m}.
\end{align*}

Let us handle \eqref{eq:cor-3-1}. To this end, we make a decomposition $\pko \phi_\theta = \phi_\theta^{\le J, k_1} + \phi_\theta^{>J, k_1}$ with $2^J = 2^{m}$. Then,
by \eqref{eq:amplitude-dirac}, one obtains
\begin{align}\label{eq:amplitude-dirac-high}
	\normo{\wh{\phi_\theta^{>J, k_1}}(s)}_{L_\xi^\infty} \les \ve_1 2^{H(2) \de m -\frac 12m} 2^{-k_1} \bra{2^{k_1}}^{-N(2)}.
\end{align}
This leads us to the estimate
\begin{align*}
	\sum_{k_1,k_2 \in \Z}|\eqref{eq:cor-3-1}| &\les \sum_{k_1,k_2 \in \Z} 2^{4H(2)\de m}  2^{\frac{3\min(\bm k)- 3k_2}2} 	\normo{\phi_\theta^{>J, k_1}(s)}_{L_\xi^\infty}\|\pkt V_{\mu,\thet}(s)\|_{L^2} \\
	&\les \sum_{k_1,k_2 \in \Z} \ve_1^2 2^{[H(0)+5H(2)]\de m -\frac m2}   2^{\frac{3\min(\bm k)- 3k_2}2-k_1} \bra{2^{k_1}}^{-N(2)}\bra{2^{k_2}}^{-N(0)}\\
	&\les \ve_1^2 2^{-\frac m4 } 2^{-\frac k2 +\frac k{100} } \bra{2^k}^{-20}.
\end{align*}
Note that $2^{\frac{3\min(\bm k)- 3k_2}2-k_1} \les 2^{30H(2)\de m}$ from the frequency relation and \eqref{eq:mid-frequency}. On the other hand,  for $\phi_\theta^{\le J,k_1}$, we  exploit the space resonance \eqref{eq:nonresonance-space}. Indeed, we integrate by parts in $\eta$. Then, \eqref{eq:cor-3-1} equipped with $\phi_\theta^{\le J,k_1}$ is the linear combination of the following:
\begin{align}
	&t^{-1} e^{-iB_\theta(t,\xi)}\int_{\R^3} e^{ itp_\Theta(\xi,\eta)} \nabla_\eta   M_{\bm k}^1(\xi,\eta)\wh{\pkt V_{\mu,\thet}}(t,\eta) \al^\mu \wh{ \phi_\theta^{\le J,k_1}}(t,\xi-\eta) \,d\eta ds,\label{eq:cor-3-11}\\
	&t^{-1} e^{-iB_\theta(t,\xi)}\int_{\R^3} e^{ itp_\Theta(\xi,\eta)}  M_{\bm k}^1(\xi,\eta)\wh{\pkt x V_{\mu,\thet}}(s,\eta) \al^\mu \wh{\phi_\theta^{\le J,k_1}}(t,\xi-\eta)\, d\eta ds,\label{eq:cor-3-12}\\
	&t^{-1} e^{-iB_\theta(t,\xi)}\int_{\R^3} e^{ itp_\Theta(\xi,\eta)}  M_{\bm k}^1(\xi,\eta)\wh{\pkt V_{\mu,\thet}}(t,\eta) \al^\mu \wh{x \phi_\theta^{\le J,k_1}}(s,\xi-\eta)\, d\eta ds,\label{eq:cor-3-13}
\end{align}
where $t= t_1, \, t_2$ and
\begin{align*}
	M_{\bm k}^1(\xi,\eta) =  \frac{\nabla_\eta p_\Theta (\xi,\eta)}{|\nabla_\eta p_\Theta (\xi,\eta)|^2} M_{\bm k}(\xi,\eta).
\end{align*}
A direct calculation gives us the bound
\begin{align*}
	\left| \nabla_\eta^\ell  M^1_{\bm k}(\xi,\eta)\right| \les 2^{-\frac{3k_2}2-\ell k_2} 2^{2(4+2\ell)H(2)\de m}
\end{align*}
for $\ell =0,1$. By \eqref{eq:amplitude-dirac}, one readily gets
\begin{align}\label{eq:amplitude-dirac-low}
	\normo{\wh{\phi_\theta^{\le J, k_1}}(s)}_{L_\xi^\infty} \les \ve_1 2^{H(2)\de m} 2^{-\frac{k_1}2} \bra{2^{k_1}}^{-N(2)}.
\end{align}
Then, we obtain
\begin{align*}
	\sum_{k_1,k_2 \in\Z}\left|\eqref{eq:cor-3-11}\right| &\les \sum_{k_1,k_2 \in\Z} 2^{-m+12H(2)\de m} 2^{-k_2} \|P_{k_2} W_{\mu,\thet}(s)\|_{L^2} \normo{\wh{\phi_\theta^{\le J, k_1}}(s)}_{L_\xi^\infty}\\
	&\les \ve_1^2 2^{-\frac m6 } \bra{2^k}^{-20}.
\end{align*}
The estimates for \eqref{eq:cor-3-12} can be done similarly by using \eqref{eq:weight}. In view of \eqref{eq:amplitude-dirac} and \eqref{eq:derivative-j}, one can see that
\begin{align}\label{eq:amplitude-dirac-spatial}
	\left\| \wh{x \phi_\theta^{\le J,k_1}}(s)\right\|_{L_\xi^\infty} \les \sum_{j \le J}\ve_1 2^{H(2)\de m} 2^{\frac{j_1}2}2^{-k_1}\bra{2^{k_1}}^{-N(2)}  \les 	\ve_1 2^{H(2)\de m+ \frac12 m} 2^{-k_1}\bra{2^{k_1}}^{-N(2)} .
\end{align}
Then, we have
\begin{align*}
	\sum_{k_1,k_2 \in\Z}\left|\eqref{eq:cor-3-13}\right| &\les \sum_{k_1,k_2 \in\Z} 2^{-m}  \bra{2^{k_1}}^4\|P_{k_2} W_{\mu,\thet}(s)\|_{L^2} \left\| \wh{x \phi_\theta^{\le J,k_1}}(s)\right\|_{L_\xi^\infty}\\
	&\les \sum_{k_1,k_2 \in\Z} \ve_1^2 2^{[H(0)+9H(2)]\de m-\frac m2 } \bra{2^{k_1}}^{-N(2)}\bra{2^{k_2}}^{-N(0)}\\
	&\les \ve_1^2 2^{-\frac m4 } \bra{2^k}^{-20}.
\end{align*}
Therefore, we have the desired bounds for \eqref{eq:cor-3-11}--\eqref{eq:cor-3-13}.

Let us move on to estimates for  \eqref{eq:cor-3-2}. From the definition of $B_\theta(s,\xi)$ in \eqref{eq:phase-correction} and \eqref{eq:decay-maxwell}, we have
\begin{align}\label{eq:decay-correction}
	|\p_s B_\theta(s,\xi)| \les \left| P_{\le K} A_{\mu} \left(s, \frac{s\xi}{\bra{\xi}}\right)\right| \les \ve_1 2^{11 \de m - m} \bra{2^{k}}^{-N(1)+2} . 
\end{align}
By the proof for \eqref{eq:cor-3-1} this finishes the proof for \eqref{eq:cor-3-2}.

To estimate \eqref{eq:cor-3-3}, we use decomposition  $\pko \phi_\theta = \phi_\theta^{\le J, k_1} + \phi_\theta^{>J, k_1}$ with $2^J =  2^{m-60H(2)\de m}$ and have
\begin{align*}
	\normo{\wh{\phi_\theta^{>J, k_1}}}_{L_\xi^\infty} \les \ve_1 2^{31H(2) \de m -\frac 12m} 2^{-k_1} \bra{2^{k_1}}^{-N(2)}.
\end{align*}
Then, by \eqref{eq:esti-nonl-max-time}, one gets
\begin{align*}
	\sum_{k_1,k_2\in \Z}|\eqref{eq:cor-3-3} | &\les \sum_{k_1,k_2\in \Z} 2^m 2^{4H(2)\de m }  2^{\frac{3\min(\bm k) - 3k_2}2}\| \pkt \p_s V_{\mu,\thet}(s)\|_{L^2}\normo{\wh{\phi_\theta^{>J, k_1}}(s)}_{L_\xi^\infty} \\
	&\les \sum_{k_1,k_2\in \Z}\ve_1^2 2^{37H(2)\de m -\frac12 m} 2^{\frac{3\min(\bm k) - 3k_2}2 -k_1} \bra{2^{k_1}}^{-N(2)}\bra{2^{k_2}}^{-N(0)+5}\\
	&\les \ve_1^2 2^{-\frac m4 } \bra{2^k}^{-25}.
\end{align*}
We now consider $ \phi_\theta^{\le J, k_1}$ in \eqref{eq:cor-3-3}. For this case, we can obtain an extra time decay by exploiting the space resonance \eqref{eq:nonresonance-space}. Indeed, we write \eqref{eq:cor-3-3} with $ \phi_\theta^{\le J, k_1}$ as the linear combination of the following:
\begin{align}
	&	\int_{t_1}^{t_2} s^{-1} e^{-iB_\theta(s,\xi)}\int_{\R^3} e^{ isp_\Theta(\xi,\eta)} \nabla_\eta M_{\bm k}^1(\xi,\eta)\p_s\wh{\pkt  V_{\mu,\thet}}(s,\eta) \al^\mu \wh{ \phi_\theta^{\le J, k_1 }}(s,\xi-\eta)\, d\eta ds,\label{eq:cor-3-31}\\
	&	\int_{t_1}^{t_2} s^{-1} e^{-iB_\theta(s,\xi)}\int_{\R^3} e^{ isp_\Theta(\xi,\eta)}  M_{\bm k}^1(\xi,\eta)\p_s\wh{\pkt  x V_{\mu,\thet}}(s,\eta) \al^\mu \wh{ \phi_\theta^{\le J, k_1 }}(s,\xi-\eta)\, d\eta ds,\label{eq:cor-3-32}\\
	&	\int_{t_1}^{t_2}  s^{-1} e^{-iB_\theta(s,\xi)}\int_{\R^3} e^{ isp_\Theta(\xi,\eta)}  M_{\bm k}^1(\xi,\eta)\p_s\wh{\pkt   V_{\mu,\thet}}(s,\eta) \al^\mu \wh{ x \phi_\theta^{\le J, k_1 }}(s,\xi-\eta)\, d\eta ds.\label{eq:cor-3-33}
\end{align}
By H\"older's inequality and \eqref{eq:amplitude-dirac}, we estimate
\begin{align*}
	\sum_{k_1,k_2 \in \Z} \left| \eqref{eq:cor-3-31}\right| &\les \sum_{k_1,k_2 \in \Z} 2^{12H(2)\de m} 2^{\frac{3\min(\bm k)-5k_2}2} \| \pkt \p_s V_{\mu,\thet} (s)\|_{L^2}\normo{\wh{ \phi_\theta^{\le J, k_1}}(s)}_{L_\xi^\infty} \\
	&\les \sum_{k_1,k_2 \in \Z} 2^{-m+14H(2)\de m} 2^{\frac{3\min(\bm k)-5k_2-k_1}2 } \bra{2^{k_1}}^{-N(2)} \bra{2^{k_2}}^{-N(0)+5} \\
	&\les 2^{-\frac m6 }\bra{2^k}^{-20}.
\end{align*}
Note that $2^{k_2} \gtrsim 2^K \sim 2^{-\frac23m -10H(2)\de m}$. Regarding \eqref{eq:cor-3-32}, we use Lemma \ref{lem:esti-wei-time} to obtain
\begin{align*}
	\sum_{k_1,k_2 \in \Z} \left| \eqref{eq:cor-3-32}\right| &\les \sum_{k_1,k_2 \in \Z} 2^{8H(2)\de m} 2^{\frac{3\min(\bm k)-3k_2}2} \| \pkt \p_s (x V_{\mu,\thet}(s))\|_{L^2}\normo{\wh{ \phi_\theta^{\le J, k_1}}(s)}_{L_\xi^\infty} \\
	&\les \sum_{k_1,k_2 \in \Z} 2^{-\frac m8+9H(2)\de m} 2^{\frac{3\min(\bm k)-3k_2-k_1}2 } \bra{2^{k_1}}^{-N(2)} \bra{2^{k_2}}^{-N(0)+5} \\
	&\les 2^{-\frac m{10}}\bra{2^k}^{-20}.
\end{align*}
Using the similar bound to \eqref{eq:amplitude-dirac-spatial}, we see that
\begin{align*}
	\sum_{k_1,k_2 \in \Z} \left| \eqref{eq:cor-3-33}\right| &\les \sum_{k_1,k_2 \in \Z} 2^{8H(2)\de m} 2^{\frac{3\min(\bm k)-3k_2}2} \| \pkt \p_s V_{\mu,\thet}(s)\|_{L^2}\normo{\wh{x \phi_\theta^{\le J, k_1}}(s)}_{L_\xi^\infty} \\
	&\les \sum_{k_1,k_2 \in \Z} 2^{-50H(2)\de m} 2^{\frac{3\min(\bm k)-3k_2}2 -k_1} \bra{2^{k_1}}^{-N(2)} \bra{2^{k_2}}^{-N(0)+5} \\
	&\les 2^{-\de m}2^{-\frac k2 +\frac k{100}}\bra{2^k}^{-20}.
\end{align*}

It remains to handle \eqref{eq:cor-3-4}. Note that the frequency relation has the trichotomy: $2^{k} \les 2^{k_1} \sim 2^{k_2}$, $2^{k_1} \ll 2^{k} \sim 2^{k_1}$, and $2^{k_2} \ll 2^{k} \sim 2^{k_1}$. As observed previously, the last case becomes the most difficult case. Hence, we consider only the case  $2^{k_2} \ll 2^{k} \sim 2^{k_1}$.  By direct calculation and dyadic decomposition,  \eqref{eq:cor-3-4} can be written as the linear combination of the following:
\begin{align}
	&\begin{aligned}
		&\int_{t_1}^{t_2}  e^{-iB_\theta(s,\xi)}\int_{\R^{3+3}} e^{ isr_\Theta(\xi,\eta,\sigma)} \wt M_{\bm k}(\xi,\eta,\sigma)\wh{\pkt V_{\mu,\thet}}(s,\eta)  \wh{\pkf V_{\nu,\thef}}(s,\sigma) \\
		&\hspace{7cm}\times \al^\mu \al^\nu \wh{ \phi_{\theth}^{> J, k_3}}(s,\xi-\eta-\sigma)\, d\sigma  d\eta ds,
	\end{aligned}\label{eq:cor-3-41}\\
	&	\begin{aligned}
		&\sum_{a,b=1,2}\int_{t_1}^{t_2}  e^{-iB_\theta(s,\xi)}\int_{\R^{3+3}} e^{ isr_\Theta(\xi,\eta,\sigma)} \wt M_{\bm k}(\xi,\eta,\sigma)\wh{ V_{\mu,\thet}^{a}}(s,\eta)  \wh{ V_{\nu,\thef}^{b}}(s,\sigma)\\
		&\hspace{7cm} \times  \al^\mu \al^\nu \wh{ \phi_{\theth}^{\le J, k_3}}(s,\xi-\eta-\sigma)\, d\sigma  d\eta ds,
	\end{aligned}\label{eq:cor-3-42}
\end{align}
where $2^{J}= 2^{m}$,
\begin{align*}
	r_\Theta(\xi,\eta,\sigma) &= \theta \bra{\xi}+\thet |\eta|  -\theth \bra{\xi-\eta-\sigma} +\thef |\sigma|,\\
	\wt M_{\bm k}(\xi,\eta,\sigma)&= M_{\bm k}(\xi,\eta) |\sigma|^{-\frac12} \rho_{k_3}(\xi-\eta-\sigma)\rho_{k_4}(\sigma),
\end{align*}
and $V_{\lam,\ka}^{1}:= V_{\lam,\ka}^{> J,k}$, $V_{\lam,\ka}^{2}:= V_{\lam,\ka}^{\le J,k}$ for $\lam \in \{\mu,\nu\}$ and $\ka \in\{\thet,\thef\}$. Note that the frequency relation can be divided into $2^{k_1} \les 2^{k_3} \sim 2^{k_4}$, $2^{k_3} \ll 2^{k_1} \sim 2^{k_4}$, and $2^{k_4} \ll 2^{k_1} \sim 2^{k_3}$ and the multiplier has the bound
\begin{align*}
	\left| \wt M_{\bm k}(\xi,\eta,\sigma)\right| \les 2^{4H(2)\de m}2^{-\frac{3k_2+k_4}2}.
\end{align*}

We first consider \eqref{eq:cor-3-41}. If $2^{k_4} \ge 2^{-\frac{7}{12} m }$, \eqref{eq:decay-maxwell}, \eqref{eq:elliptic-q}, and  \eqref{eq:high} yield that
\begin{align*}
	\sum_{\substack{k_2,k_3 \in \Z\\ 2^{k_4} \ge 2^{-\frac{7}{12} m}}}	&\left|\eqref{eq:cor-3-41}\right|\\
	&\les  \sum_{\substack{k_2,k_3 \in \Z\\ 2^{k_4} \ge 2^{-\frac{7}{12} m}}} 2^{m+4H(2)\de m} 2^{-\frac{3k_2+k_4}2} \|\pkt W_{\mu,\thet}(s)\|_{L^\infty} \|\phi_\theth^{>J,k_3}(s)\|_{L^2}  \|\pkf W_{\mu,\thef}(s)\|_{L^2}\\
	&\les  \sum_{\substack{k_2,k_3 \in \Z\\ 2^{k_4} \ge 2^{-\frac{7}{12} m}}}   \ve_1^3 2^{-m + 7H(2) \de m} 2^{-\frac{2k_2+k_4}2}\bra{2^{k_2}}^{-N(1)+2}\bra{2^{k_3}}^{-N(1)}\bra{2^{k_4}}^{-N(0)}\\
	&\les \ve_1^3 2^{-\de m}  2^{-\frac k2 +\frac k{100}}\bra{2^k}^{-20}.
\end{align*}
On the other hand, if $2^{k_4} \le 2^{-\frac7{12}m}$,
from the fact that $2^{k_1} \sim 2^{k} \in [2^{-30H(2)\de m} , 2^{2H(2)\de m}]$, only the case $2^{k_4} \ll 2^{k_1} \sim 2^{k_3}$ appears. Hence by \eqref{eq:amplitude-dirac-high}, we get
\begin{align*}
	\sum_{\substack{ 2^{k_4} \le 2^{-\frac7{12}m}\\ 2^{k_4} \ll 2^{k_1} \sim 2^{k_3}} }	&\left|\eqref{eq:cor-3-41}\right| \\
	&\les  \sum_{\substack{2^{k_4} \le 2^{-\frac7{12}m}\\2^{k_4} \ll 2^{k_1} \sim 2^{k_3} }} 2^{m+4H(2)\de m} 2^{k_4} \|\pkt W_{\mu,\thet}(s)\|_{L^2} \left\|\wh{\phi_\theth^{>J,k_3}}(s)\right\|_{L_\xi^\infty}  \|\pkf W_{\mu,\thef}(s)\|_{L^2}\\
	&\les  \sum_{\substack{ 2^{k_4} \le 2^{-\frac7{12}m}\\2^{k_4} \ll 2^{k_1} \sim 2^{k_3}} }   \ve_1^3 2^{\frac m2 +7H(2)\de m} 2^{k_4} 2^{-\frac{k_3}2}\bra{2^{k_2}}^{-N(0)}\bra{2^{k_3}}^{-N(2)}\bra{2^{k_4}}^{-N(0)}\\
	&\les \ve_1^3 2^{-\de m} 2^{-\frac k2 +\frac k{100} } \bra{2^k}^{-20}.
\end{align*}
In the last inequality, we used the condition $2^{-30H(2)\de m} \les 2^{k_3}$ and $2^k \sim 2^{k_3}$.

We consider \eqref{eq:cor-3-42} with $(a,b)= (1,1)$. If $ 2^{k_3} \ge 2^{-\frac7{12} m }$, by \eqref{eq:elliptic-q}, \eqref{eq:amplitude-dirac-low}, and H\"older's inequality, we estimate
\begin{align*}
	\sum_{\substack{k_2, k_4 \in \Z\\ 2^{k_3} \ge 2^{-\frac7{12} m }}}	&\left|\eqref{eq:cor-3-42}\right|\\
	&\les  	\sum_{\substack{k_2, k_4 \in \Z\\ 2^{k_3} \ge 2^{-\frac7{12} m }}} 2^{m+4H(2)\de m} 2^{-\frac{k_4}2}\|\rho_{k_4}\|_{L_\sigma^2} \|V_{\mu,\thet}^{>J, k_2}(s)\|_{L^2} \|\phi_\theth^{\le J,k_3}(s)\|_{L_\xi^\infty}  \|\pkf V_{\mu,\thef}^{>J,k_4}(s)\|_{L^2}\\
	&\les  	\sum_{\substack{k_2, k_4 \in \Z\\ 2^{k_3} \ge 2^{-\frac7{12} m }}}   \ve_1^3 2^{-m +7H(2)\de m} 2^{-k_2-\frac{k_3}2} \bra{2^{k_2}}^{-N(1)}\bra{2^{k_3}}^{-N(1)}\bra{2^{k_4}}^{-N(1)}\\
	&\les \ve_1^3 2^{-\de m}2^{-\frac k2 +\frac k{100}} \bra{2^k}^{-20}.
\end{align*}
If $2^{k_3} \le 2^{-\frac7{12}m}$, then the frequency relation gives us only the case $2^{k_3}\ll 2^{k_4} \sim 2^{k_1} \in [2^{-30H(2)\de m}, 2^{2H(2)\de m}]$. Then, by \eqref{eq:decay-maxwell}, \eqref{eq:elliptic-q}, and the above estimate done with $\|\rho_{k_3}\|_{L_\sigma^2}$ instead of $\|\rho_{k_4}\|_{L_\sigma^2}$, we obtain
\begin{align*}
	\sum_{\substack{2^{k_3} \le 2^{-\frac7{12}m}\\ 2^{k_3}\ll 2^{k_4} \sim 2^{k_1}}}\left|\eqref{eq:cor-3-42}\right| &\les  \sum_{\substack{2^{k_3} \le 2^{-\frac7{12}m}\\ 2^{k_3}\ll 2^{k_4} \sim 2^{k_1}}} 2^{m+4H(2)\de m} 2^{\frac{3k_3 - k_4}2} \left\|V_{\mu,\thet}^{>J, k_2}(s)\right\|_{L^2} \left\|\phi_\theth^{\le J,k_3}(s)\right\|_{L_\xi^\infty}  \left\|V_{\mu,\thef}^{>J,k_4}(s)\right\|_{L^2}\\
	&\les  \sum_{\substack{2^{k_3} \le 2^{-\frac7{12}m}\\ 2^{k_3}\ll 2^{k_4} \sim 2^{k_1}}}   \ve_1^3 2^{-m +7H(2)\de m} 2^{-\frac{2k_2+ 3k_4}2}2^{k_3}\bra{2^{k_2}}^{-N(1)}\bra{2^{k_3}}^{-N(2)}\bra{2^{k_4}}^{-N(1)}\\
	&\les \ve_1^3 2^{-\de m} 2^{-\frac k2 +\frac k{100}} \bra{2^k}^{-20}.
\end{align*}

For the estimate of \eqref{eq:cor-3-42} with $(a,b)= (2,1)$. Analogously to \eqref{eq:nonresonance-space}, we have
\begin{align*}
	\left| \nabla_\eta r_\Theta(\xi,\eta,\sigma) \right| \ge \bra{\xi-\eta-\sigma}^{-2}.
\end{align*}
Keeping this lower bound in mind, we integrate by parts in $\eta$ and write \eqref{eq:cor-3-42} as the linear combination of the following:
\begin{align}
	&	\begin{aligned}
		&\int_{t_1}^{t_2}  s^{-1} e^{-iB_\theta(s,\xi)}\int_{\R^{3+3}} e^{ isr_\Theta(\xi,\eta,\sigma)} \nabla_\eta\wt M_{\bm k}^1(\xi,\eta,\sigma)\wh{ V_{\mu,\thet}^{\le J, k_2}}(s,\eta)  \wh{ V_{\nu,\thef}^{> J, k_4}}(\sigma)\\
		&\hspace{8cm} \times  \al^\mu \al^\nu \wh{ \phi_{\theth}^{\le J, k_3}}(s,\xi-\eta-\sigma)\, d\sigma  d\eta ds,
	\end{aligned}\label{eq:cor-3-42-eta-1}\\
	&	\begin{aligned}
		&\int_{t_1}^{t_2}  s^{-1} e^{-iB_\theta(s,\xi)}\int_{\R^{3+3}} e^{ isr_\Theta(\xi,\eta,\sigma)} \wt M_{\bm k}^1(\xi,\eta,\sigma)\wh{x V_{\mu,\thet}^{\le J, k_2}}(s,\eta)  \wh{ V_{\nu,\thef}^{> J, k_4}}(\sigma)\\
		&\hspace{8cm} \times  \al^\mu \al^\nu \wh{ \phi_{\theth}^{\le J, k_3}}(s,\xi-\eta-\sigma)\, d\sigma  d\eta ds,
	\end{aligned}\label{eq:cor-3-42-eta-2}\\
	&	\begin{aligned}
		&\int_{t_1}^{t_2}  s^{-1} e^{-iB_\theta(s,\xi)}\int_{\R^{3+3}} e^{ isr_\Theta(\xi,\eta,\sigma)} \wt M_{\bm k}^1(\xi,\eta,\sigma)\wh{ V_{\mu,\thet}^{\le J, k_2}}(s,\eta)  \wh{ V_{\nu,\thef}^{> J, k_4}}(\sigma)\\
		&\hspace{8cm} \times  \al^\mu \al^\nu \wh{ x\phi_{\theth}^{\le J, k_3}}(s,\xi-\eta-\sigma)\, d\sigma  d\eta ds,
	\end{aligned}\label{eq:cor-3-42-eta-3}
\end{align}
where
\begin{align*}
	\wt M_{\bm k}^1(\xi,\eta,\sigma) = \frac{\nabla_\eta r_\theta(\xi,\eta,\sigma)}{|\nabla_\eta r_\theta(\xi,\eta,\sigma)|^2}\wt M_{\bm k}(\xi,\eta,\sigma).
\end{align*}
The multiplier is bounded by
\begin{align*}
	\left| \nabla_\eta^\ell \wt M_{\bm k}^1(\xi,\eta,\sigma) \right| \les 2^{4H(2)\de m}2^{-\frac{3k_2}2 -\ell k_2}2^{-\frac{k_4}2}\bra{2^{k_3}}^{2+ 2\ell} \;\;(\ell =0,1).
\end{align*}
Then we estimate
\begin{align*}
	| \eqref{eq:cor-3-42-eta-1} | &\les 	  2^{4H(2)\de m} 2^{-k_2}2^{k_4} \bra{2^{k_3}}^4 \| V_{\mu,\thet}^{\le J, k_2}(s)\|_{L^2} \|\phi_\theth^{\le J, k_3}(s)\|_{L_\xi^\infty}
	\| V_{\mu,\thef}^{> J, k_4}(s)\|_{L^2}\\
	&\les  	 \ve_1^3 2^{ -m+ 20H(2)\de m   }2^{-k_2}2^{-\frac{k_3+k_4}2} \bra{2^{k_2}}^{-N(1)}\bra{2^{k_3}}^{-N(2)+4}\bra{2^{k_4}}^{-N(1)}.
\end{align*}
Summing both sides over $2^{\min(k_3,k_4)} \ge 2^{-\frac7{12}m}$, we get the desired bound. On the other hand, since
\begin{align*}
	\| V_{\mu,\thef}^{> J, k_4}(s)\|_{L^2} \les \ve_1 2^{-r m} 2^{-r k_4}\bra{2^{k_4}}^{-N(1)}
\end{align*}
for $0 \le r \le 1$ by \eqref{eq:elliptic-q}, it follows that, for $ \frac14 < r < 1$,
\begin{align*}
	&\sum_{k_2,k_3,k_4 \in \Z}	| \eqref{eq:cor-3-42-eta-1} |\\ &\les 	\sum_{k_2,k_3,k_4 \in \Z}  2^{4H(2)\de m} 2^{-k_2}2^{-\frac{k_4}2}2^{\frac{3\min(k_3,k_4)}2} \bra{2^{k_3}}^4 \| V_{\mu,\thet}^{\le J, k_2}(s)\|_{L^2} \|\phi_\theth^{\le J, k_3}(s)\|_{L_\xi^\infty}
	\| V_{\nu,\thef}^{> J, k_4}(s)\|_{L^2}\\
	&\les  	\sum_{k_2,k_3,k_4 \in \Z} \ve_1^3 2^{ -rm + 7H(2)\de m }2^{-k_2}2^{-\frac{k_3+k_4}2} 2^{\frac{3\min(k_3,k_4)}2} 2^{-r k_4} \bra{2^{k_2}}^{-N(1)}\bra{2^{k_3}}^{-N(2)+4}\bra{2^{k_4}}^{-N(1)}\\
	&\les \ve_1^3 2^{-\de m}2^{-\frac k2 +\frac k{100}} \bra{2^k}^{20}.
\end{align*}
Note that we summed for  $\max(k_3,k_4)$ by using the frequency relation trichotomy. The estimates for \eqref{eq:cor-3-42-eta-2} and \eqref{eq:cor-3-42-eta-3} are quite similar. Hence, we treat only the estimate for \eqref{eq:cor-3-42-eta-2}. By \eqref{eq:zero-vec-2},
\begin{align*}
	&\sum_{k_2,k_3,k_4 \in \Z}	|\eqref{eq:cor-3-42-eta-2}|\\ &\les 	\sum_{k_2,k_3,k_4 \in \Z}   2^{4H(2)\de m}2^{-\frac{k_4}2} 2^{\frac{3\min(k_3,k_4)}2} \bra{2^{k_3}}^2 \| x V_{\mu,\thet}^{\le J, k_2}(s)\|_{L^2} \|\phi_\theth^{\le J, k_3}(s)\|_{L_\xi^\infty}
	\| V_{\nu,\thef}^{> J, k_4}(s)\|_{L^2}\\
	&\les 	\sum_{k_2,k_3,k_4 \in \Z} \ve_1^3  2^{-m+7H(2)\de m}2^{-k_2} 2^{2^{\frac{3\min(k_3,k_4) -k_3 -3k_4}2}} \bra{2^{k_2}}^{-N(1)}\bra{2^{k_3}}^{-N(2)+2}\bra{2^{k_4}}^{-N(1)}\\
	&\les \ve_1^3 2^{-\de m}2^{-\frac k2 +\frac k{100}} \bra{2^k}^{20}.
\end{align*}

For the estimates of \eqref{eq:cor-3-42} with $(a,b)=(1,2)$ and $(2,2)$, we can obtain the desired bound by exploiting the resonance with respect to $\sigma$. We omit the details.

%%%%%%%%%%%%%%%%%%%%%%%%%%%%%%%%%%%%%%%%%%%%%%%%%%%%%%%%%%%%%%%%%%%%%%%%

\emph{Bound for \eqref{eq:correction-4}.} As for \eqref{eq:correction-4} the phase interaction does not exhibit the space-time resonance. Hence we can use the normal form approach as follows: The integration by parts in time shows that \eqref{eq:correction-4} consists of the following terms:
\begin{align}
	&\rho_k(\xi) e^{-iB_\theta(t,\xi)}\int_{\R^3} \frac{e^{\theta itp_\Theta(\xi,\eta)}}{p_\Theta(\xi,\eta)} \rho_{\le K}(\eta)|\eta|^{-\frac12}\wh{V_{\mu,\thet}}(t,\eta) \al^\mu \wh{\phi_{-\theta}}(t,\xi-\eta) d\eta, \qquad (t=t_1,\, t_2) \label{eq:cor-4-non-1}\\
	&\int_{t_1}^{t_2} \rho_k(\xi)  [\p_s B_\theta(s,\xi)] e^{-iB_\theta(s,\xi)}\int_{\R^3} \frac{e^{\theta isp_\Theta(\xi,\eta)}}{p_\Theta(\xi,\eta)}  \rho_{\le K}(\eta)|\eta|^{-\frac12} \wh{V_{\mu,\thet}}(s,\eta) \al^\mu \wh{\phi_{-\theta}}(s,\xi-\eta) d\eta ds,\label{eq:cor-4-non-2}\\
	&\int_{t_1}^{t_2} \rho_k(\xi) e^{-iB_\theta(s,\xi)}\int_{\R^3} \frac{e^{\theta isp_\Theta(\xi,\eta)}}{p_\Theta(\xi,\eta)} \rho_{\le K}(\eta) |\eta|^{-\frac12} \p_s \left[ \wh{V_{\mu,\thet}}(s,\eta) \al^\mu \wh{\phi_{-\theta}}(s,\xi-\eta) \right] d\eta ds.\label{eq:cor-4-non-3}
\end{align}
From the restriction $|\eta| \le 2^{-\frac23m -10H(2)\de m}$, we readily close the estimates for \eqref{eq:cor-4-non-1}--\eqref{eq:cor-4-non-3}. Especially, we can use \eqref{eq:decay-correction} for \eqref{eq:cor-4-non-2}.

\section{Asymptotic behavior for the gauge fields}\label{sec:s-maxwell}
In this section, we show the scattering norm bound \eqref{eq:s-norm-m} for Maxwell part under the a priori assumptions \eqref{assum:energy}--\eqref{assum:s-norm} on $[0,T]$ for $T >1$.

According to the definition \eqref{def:s-norm-maxwell}, the scattering norm bound \eqref{eq:s-norm-m} is equivalent to
\begin{align}\label{eq:aim-scattering-maxwell}
	\sum_{j \in \cU_k} 2^j \normo{\qjk V_{\mu,\theta}(t_2) - \qjk V_{\mu,\theta}(t_2)}_{L^2} \les \ve_1 2^{-\de m} 2^{-k -  5H(2)\de k} \bra{2^k}^{-N(1)+5}
\end{align}
for $t_1 \le t_2 \in [2^m - 2, 2^{m+1}] \cap [0,T]$ with $m \in \Z$. Analogously to the proof in Section \ref{sec:s-dirac}, we prove \eqref{eq:aim-scattering-maxwell} by decomposing
the case into low, high, and mid frequency regime. Before the decomposition, we consider the large spatial region as follows:
\begin{align}\label{eq:asymp-maxwell-j}
	\sum_{j \ge J} 2^j \normo{\qjk V_{\mu,\theta}(t_2) - \qjk V_{\mu,\theta}(t_2)}_{L^2} \les \ve_1 2^{-\de m} 2^{-k - 5H(2)\de k} \bra{2^k}^{-N(1)+5},
\end{align}
where $2^{J} \sim 2^{m(1+H(2)\de)}2^{|k|}$. For the difference in the norm of \eqref{eq:asymp-maxwell-j}, one may use
\begin{align}\label{eq:time-deri-maxwell}
	\int_{t_1}^{t_2} \p_s \qjk V_{\mu,\theta}(s) ds = \int_{t_1}^{t_2} \qjk e^{-\theta is |D|} \mathfrak{N}_{\mu}^{\bm M}(s)   ds.
\end{align}
Hence it suffices for \eqref{eq:asymp-maxwell-j} to prove
\begin{align}\label{eq:s-nonlinear-maxwell}
	\sum_{k_1,k_2 \in \Z}2^{-\frac k2}\normo{\rho_j(x)   P_k \bra{\pko \psi_\theo(t), \al_\mu \pkt\psi_\thet(t)}}_{L^2}  \les  \ve_1 2^{-(1+\de) m} 2^{-(1+\de)j} \bra{2^k}^{-N(1)+5}.
\end{align}
For the sum over $\max(k_1,k_2) \ge j$,  we have the bound \eqref{eq:asymp-maxwell-j} straightforwardly by \eqref{eq:high}. As for the sum over  $\min(k_1,k_2) \le -j$, \eqref{eq:esti-l2-k} also implies the bound \eqref{eq:asymp-maxwell-j}. We now assume $2^{k_1}, 2^{k_2} \in [2^{-j},2^j]$. For this case we further decompose
\begin{align*}
	P_{k_\ell}\psi_{\theta_\ell} = \sum_{j_\ell \in \cU_{k_\ell}} e^{-\theta_\ell it\bra{D}}\phi_{\theta_\ell}^{j_\ell,k_\ell}   \qquad \mbox{ for }\; \ell =1,\,2.
\end{align*}
If $\min(j_1,j_2) \ge j(1-\de)$, by \eqref{eq:elliptic-q}, we estimate
\begin{align*}
	&	2^{-\frac k2}\normo{\rho_j(x)   P_k \bra{\pko \psi_\theo(t), \al_\mu \pkt\psi_\thet(t)}}_{L^2}\\
	&\les \sum_{j_\ell \in \cU_{k_\ell}} 2^k \| \phi_{\theo}^{j_1,k_1}(t)\|_{L^2}\| \phi_{\thet}^{j_2,k_2}(t)\|_{L^2} \les \sum_{j_\ell \in \cU_{k_\ell}}  \ve_1^2 2^{2H(1)\de m} 2^k 2^{-j_1-j_2} \bra{2^{k_1}}^{-N(1)}\bra{2^{k_2}}^{-N(1)}\\
	&\les \ve_1^2 2^{-(1+\de)m}\bra{2^k}^{-N(1)+1}.
\end{align*}
On the other hand, if $\min(j_1,j_2) \le j(1-\de)$, using the inverse Fourier transform, we see that
\begin{align}\label{eq:integrand}
	\begin{aligned}
		&P_k \bra{\pko \psi_\theo(t), \al_\mu \pkt\psi_\thet(t)}  \\
		&=C  \int_{\R^{3+3}} e^{ix\xi} \rho_k(\xi) e^{itq_\Theta(\xi,\eta)} \bra{\wh{\pko \psi_\theo}(t,\eta), \al_\mu \wh{\pkt\psi_\thet}(t,\xi+\eta)}   d\xi d\eta.
	\end{aligned}
\end{align}
Then, using integration by parts in $\xi$ repeatedly, \eqref{eq:integrand} is bounded by
\begin{align}\label{eq:multi}
	2^{-nj} \left|\int_{\R^{3+3}} e^{ix\xi}  \nabla_\xi^n \left[\rho_k(\xi) e^{itq_\Theta(\xi,\eta)} \bra{\wh{ \psi_\theo^{j_1,k_1}}(t,\eta), \al_\mu \wh{\psi_\thet^{j_2,k_2}}(t,\xi+\eta)} \right]   d\xi d\eta \right|.
\end{align}
Note that the derivative $\nabla_\xi$ implies $2^{m}+ 2^{-k} + 2^{j_\ell}$-growth in \eqref{eq:multi}. Summing over $\min(j_1,j_2) \le j(1-\de)$ for suitable $n$, we get the desired bound \eqref{eq:s-nonlinear-maxwell}.

We consider the contribution of sum over $j \le J$ in the sequel.
\subsection{Proof for the low and high frequency regimes} Let us restrict $|\xi|$ as follows:
\begin{align}\label{eq:restriction-low-hi}
	2^k \le 2^{-\frac m{100}} \;\;\mbox{ or }\;\;	2^k \ge 2^{H(1)\de m}.
\end{align}
Then by \eqref{eq:elliptic-q}, we have
\begin{align*}
	\sup_{j \in \cU_k} 2^j \normo{\qjk V_{\mu,\theta}(t)}_{L^2} \les \ve_1 \bra{t}^{H(1)\de} 2^{-k}\bra{2^k}^{-N(1)}.
\end{align*}
Then, \eqref{eq:aim-scattering-maxwell} follows for \eqref{eq:restriction-low-hi} from the fact that $H(1) < \frac{5H(2)}{100}$.

\subsection{Proof for mid frequency regime} We now consider the mid frequency part
\begin{align}\label{eq:restriction-mid}
	2^{-\frac m{100}} \le 2^k \le  2^{H(1)\de m}.
\end{align}
Using the time derivative expression in \eqref{eq:time-deri-maxwell} and the profile expression, we can write \eqref{eq:aim-scattering-maxwell} as
\begin{align}\label{eq:s-maxwell-mid}
	\sum_{k_1,k_2 \in \Z}\int_{t_1}^{t_2} \rho_k(\xi) \int_{\R^3} e^{isq_\Theta(\xi,\eta)}|\xi|^{-\frac12} \bra{\wh{\pko\phi_\theo}(s,\eta), \al_\mu  \wh{\pkt\phi_\thet}(s,\xi+\eta)}  d\eta ds.
\end{align}
Then, for the proof of \eqref{eq:aim-scattering-maxwell} over the regime \eqref{eq:restriction-mid}, it suffices to show that
\begin{align}\label{eq:equiv-maxwell}
	2^J\|\eqref{eq:s-maxwell-mid}\|_{L_\xi^2} \les \ve_1 2^{-\de m} 2^{-k-5H(2)\de k} \bra{2^{k}}^{-N(1)-5}.
\end{align}
By the symmetry between $\psi_\theo$ and $\psi_\thet$, we assume that $\min(k_1,k_2)= k_1$.

Let us handle first the case $2^{k_1} \le 2^{-\frac9{10} m}$.  Then, by \eqref{eq:zero-vec-1}, one gets
\begin{align*}
	\|\eqref{eq:s-maxwell-mid}\|_{L_\xi^2} &\les \sum_{2^{k_1} \le 2^{-\frac9{10} m}}2^m2^{-\frac k2} 2^{3k_1} \|\wh{\pko \phi_\theo}(s)\|_{L_\xi^\infty}\|\pkt \phi_\thet(s)\|_{L^2} \\
	&\les \ve_1^2 2^{-\frac54 m}2^{-10H(2)\de m + H(0)\de m} 2^{-\frac k2} \bra{2^k}^{-N(1)+5}.
\end{align*}
This gives us the desired bound. If $2^{\max(k_1,k_2)} \ge 2^{\frac m{50}}$, then we have the frequency relation $2^k \ll 2^{k_1} \sim 2^{k_2}$ from \eqref{eq:restriction-mid}. Using \eqref{eq:high} together with the above summation taken over $2^{\max(k_1,k_2)} \ge 2^{\frac m{50}}$, we obtain \eqref{eq:equiv-maxwell}.

Now it remains to handle the case $  2^{k_1},2^{k_2} \in [2^{-\frac9{10} m}, 2^{\frac m{50}}]$. To do so, we exploit the space-time resonance \eqref{eq:resonance-time-maxwell} and \eqref{eq:resonance-space-maxwell}. By the integration by parts in time, we have the decomposition of \eqref{eq:s-maxwell-mid} as follows:
\begin{align}
	& \rho_k(\xi) \int_{\R^3} \frac{e^{itq_\Theta(\xi,\eta)}}{q_\Theta(\xi,\eta)}|\xi|^{-\frac12} \bra{\wh{\pko\phi_\theo}(t,\eta), \al_\mu  \wh{\pkt\phi_\thet}(t,\xi+\eta)}  d\eta, \qquad (t= t_1,t_2) \label{eq:s-mid-1}\\
	&\int_{t_1}^{t_2} \rho_k(\xi) \int_{\R^3} \frac{e^{isq_\Theta(\xi,\eta)}}{q_\Theta(\xi,\eta)}|\xi|^{-\frac12}  \bra{\p_s \wh{\pko\phi_\theo}(s,\eta), \al_\mu  \wh{\pkt\phi_\thet}(s,\xi+\eta)}  d\eta ds,  \label{eq:s-mid-2}
\end{align}
and the symmetric term with time derivative falling on $\wh{\pkt\phi_\thet}$. Let us treat \eqref{eq:s-mid-1} first. If $2^{k_1} \le 2^{-\frac 1{10}m}$, then we have the relation $2^{k_1} \ll 2^k \sim 2^{k_2}$. Then, by \eqref{eq:decay-dirac} and \eqref{eq:high}, one can obtain
\begin{align*}
	\sum_{\substack{2^{k_1} \le 2^{-\frac m{10}}\\2^{k_1} \ll 2^k \sim 2^{k_2}}}	\|\eqref{eq:s-mid-1}\|_{L_\xi^2} &\les \sum_{\substack{2^{k_1} \le 2^{-\frac m{10}}\\2^{k_1} \ll 2^k \sim 2^{k_2}}} \|\pko \psi_\theo\|_{L^\infty}\|\pkt\psi_\thet\|_{L^2}\\
	&\les \sum_{\substack{2^{k_1} \le 2^{-\frac 1{10}m}\\2^{k_1} \ll 2^k \sim 2^{k_2}}} \ve_1^2 2^{-m + 4H(1)\de m+ \frac m{200}}  2^{-k} 2^{\frac {k_1}2}  \bra{2^{k_2}}^{-N(0)}\\
	&\les \ve_1^2 2^{-m+ 4H(1)\de m - \frac9{200} m} 2^{-k}\bra{2^k}^{-N(1)+5}.
\end{align*}
On the other hand, if $2^{k_1} \in [ 2^{-\frac m{10}}, 2^{\frac m{50}}]$,  then  by \eqref{eq:zero-vec-2}, we see that
\begin{align*}
	\|e^{-\theo it\bra{D}}\phi_\theo^{\le J_0,k_1}\|_{L^\infty} &\les \ve_1 \bra{t}^{-\frac32} 2^{-\frac {k_1}2 } \bra{2^{k_1}}^{-N(1)},\\
	\|\phi_\theo^{> J_0,k_1}\|_{L^2} &\les \ve_1 \bra{t}^{-1+H(1)\de}  \bra{2^{k_1}}^{-N(1)-1},
\end{align*}
where let $J_0$ be an integer satisfying $2^{J_0} = c_0\bra{t}2^{k_1} $ for some fixed $0 < c_0\ll 1$.
From the above bounds and the restriction $k _1 \le k_2$ under the regime \eqref{eq:restriction-mid} it follows that
\begin{align*}
	\sum_{2^{k_1} \in [ 2^{-\frac m{10}}, 2^{\frac m{50}}]}	\|\eqref{eq:s-mid-1}\|_{L_\xi^2} &\les \sum_{2^{k_1} \in [ 2^{-\frac m{10}}, 2^{\frac m{50}}]}2^{-\frac32 k} 2^{2H(1)\de m} \|e^{-\theta it\bra{D}}\phi_\theo^{\le J_0,k_2}\|_{L^\infty} \|\pkt \psi_\thet\|_{L^2}\\
	&\les \sum_{2^{k_1} \in [ 2^{-\frac m{10}}, 2^{\frac m{50}}]} \ve_1^2 2^{-\frac32m + 4H(1)\de m+\frac m{200}} 2^{-k} 2^{-\frac {k_1}2}   \bra{2^{k_1}}^{-N(1)}\bra{2^{k_2}}^{-N(0)}\\
	&\les \ve_1^2 2^{-\frac54 m} 2^{-k} \bra{2^k}^{-N(1)+5}
\end{align*}
and
\begin{align*}
	\sum_{2^{k_1} \in [ 2^{-\frac m{10}}, 2^{\frac m{50}}]}	\|\eqref{eq:s-mid-1}\|_{L_\xi^2} &\les \sum_{2^{k_1} \in [ 2^{-\frac m{10}}, 2^{\frac m{50}}]}2^{-\frac32 k} 2^{2H(1)\de m}  \|\phi_\theo^{> J_0,k_2}\|_{L^2} \|\pkt \psi_\thet\|_{L^\infty}\\
	&\les \sum_{2^{k_1} \in [ 2^{-\frac m{10}}, 2^{\frac m{50}}]} \ve_1^2 2^{-2m + 10H(1)\de m + \frac m{200}} 2^{-k}  2^{\frac{k_2}2}  \bra{2^{k_1}}^{-N(1)-1}\bra{2^{k_2}}^{-N(1)+2}\\
	&\les \ve_1^2 2^{-\frac32m} 2^{-k} \bra{2^k}^{-N(1)+5}.
\end{align*}
These lead us to \eqref{eq:equiv-maxwell}. The estimates for \eqref{eq:s-mid-2} can be obtained by \eqref{eq:esti-non-dirac-time} similarly to those for \eqref{eq:s-mid-1}.

\section*{Acknowledgements}
Y. Cho was supported by the National Research Foundation of Korea(NRF) grant
funded by the Korea government(MSIT) (RS-2024-00333393) and K. Lee was supported
in part by NRF-2022R1I1A1A0105640813 and NRF-2019R1A5A102832422,
the National Research Foundation of Korea(NRF) grant funded by the Korea government (MOE) and (MSIT), respectively.

\bibliographystyle{plain}
\bibliography{ReferencesKiyeon}

\medskip

\end{document}